\documentclass[11pt]{amsart}
\usepackage{amscd,amsxtra,amssymb, bbm}

\makeatletter
\newcommand*{\rom}[1]{\expandafter\@slowromancap\romannumeral #1@}
\makeatother

\usepackage[dvips]{graphics}
\usepackage{epsfig,graphicx}

\usepackage{multirow} 
\usepackage{stmaryrd} 

\input xy
\xyoption{all}
\xyoption{arc}
\SelectTips{cm}{10}

\newtheorem{theorem}{Theorem}[section]
\newtheorem{corollary}[theorem]{Corollary}
\newtheorem{lemma}[theorem]{Lemma}
\newtheorem{proposition}[theorem]{Proposition}

\theoremstyle{definition}
\newtheorem{definition}[theorem]{Definition}
\newtheorem{remark}[theorem]{Remark}
\newtheorem{algorithm}[theorem]{Algorithm}
\newtheorem{example}[theorem]{Example}

\DeclareMathOperator{\Dbcoh}{\mathsf{D^{b}_{coh}}}
\DeclareMathOperator{\Perf}{\mathsf{Perf}}

\DeclareMathOperator{\FM}{\mathsf{FM}}

\DeclareMathOperator{\rk}{\mathrm{rk}}

\renewcommand{\ker}{\mathrm{ker}}

\newcommand{\conj}{\mathsf{conj}}
\newcommand{\can}{\mathsf{can}}

\newcommand{\Tr}{\mathsf{Tr}}
\newcommand{\tr}{\mathsf{tr}}

\DeclareMathOperator{\res}{\mathsf{res}}
\DeclareMathOperator{\ev}{\mathsf{ev}}
\DeclareMathOperator{\Sol}{\mathsf{Sol}}

\DeclareMathOperator{\Coh}{\mathsf{Coh}}

\DeclareMathOperator{\Hot+}{\mathsf{Hot^{+,\,b}_{coh}}}
\DeclareMathOperator{\Hotb}{\mathsf{Hot^{b}_{coh}}}
\DeclareMathOperator{\Comb}{\mathsf{Com^{b}_{coh}}}

\DeclareMathOperator{\VB}{\mathsf{VB}}

\DeclareMathOperator{\Tri}{\mathsf{Tri}}

\DeclareMathOperator{\Pic}{\mathsf{Pic}}

\DeclareMathOperator{\Hom}{\mathsf{Hom}}

\DeclareMathOperator{\Ext}{\mathsf{Ext}}
\DeclareMathOperator{\Lin}{\mathsf{Lin}}
\DeclareMathOperator{\GL}{\mathsf{GL}}
\DeclareMathOperator{\Aut}{\mathsf{Aut}}
\DeclareMathOperator{\End}{\mathsf{End}}
\DeclareMathOperator{\Mat}{\mathsf{Mat}}
\DeclareMathOperator{\Spec}{\mathsf{Spec}}

\DeclareMathOperator{\slie}{\mathfrak{sl}}
\DeclareMathOperator{\pgl}{\mathfrak{pgl}}
\DeclareMathOperator{\lieg}{\mathfrak{g}}

\DeclareMathOperator{\lien}{\mathfrak{n}}
\DeclareMathOperator{\liep}{\mathfrak{p}}
\DeclareMathOperator{\liew}{\mathfrak{w}}
\DeclareMathOperator{\liel}{\mathfrak{l}}
\DeclareMathOperator{\lieb}{\mathfrak{b}}

\DeclareMathOperator{\lief}{\mathfrak{f}}
\DeclareMathOperator{\liev}{\mathfrak{v}}
\DeclareMathOperator{\lieu}{\mathfrak{u}}
\DeclareMathOperator{\lies}{\mathfrak{s}}

\newcommand{\bul}{\scriptstyle{\bullet}}

\newcommand{\kk}{\mathbbm{k}}
\newcommand{\CC}{\mathbb{C}}
\newcommand{\ZZ}{\mathbb{Z}}

\newcommand{\catD}{\mathsf{D}}
\newcommand{\catA}{\mathsf{A}}
\newcommand{\catB}{\mathsf{B}}
\newcommand{\catI}{\mathsf{E}}

\newcommand{\FF}{\mathbb{F}}
\newcommand{\GG}{\mathbb{G}}

\renewcommand{\SS}{\mathbb{S}}

\newcommand{\PP}{\mathbb{P}}

\newcommand{\Ad}{\mathsf{Ad}}

\newcommand{\kA}{\mathcal{A}}

\newcommand{\kE}{\mathcal{E}}
\newcommand{\kF}{\mathcal{F}}
\newcommand{\kG}{\mathcal{G}}

\newcommand{\kI}{\mathcal{I}}
\newcommand{\kO}{\mathcal{O}}
\newcommand{\kL}{\mathcal{L}}
\newcommand{\kP}{\mathcal{P}}
\newcommand{\kQ}{\mathcal{Q}}
\newcommand{\kR}{\mathcal{R}}

\newcommand{\kM}{\mathcal{M}}

\newcommand{\kV}{\mathcal{V}}

\newcommand{\kT}{\mathcal{T}}

\newcommand{\kX}{\mathcal{X}}
\newcommand{\kY}{\mathcal{Y}}

\newcommand{\mm}{\mathsf{m}}

\newcommand{\cc}{\mathsf{c}}

\newcommand{\lar}{\longrightarrow}

\DeclareMathOperator{\liee}{\mathfrak{a}}

\begin{document}

\title[Vector bundles and Yang--Baxter equation]{Vector bundles on
plane cubic curves  and  the classical  Yang--Baxter equation}

\author{Igor Burban}
\address{
Mathematisches Institut,
Universit\"at Bonn,
Endenicher Allee 60, 53115 Bonn,
Germany
}
\email{burban@math.uni-bonn.de}

\author{Thilo Henrich}
\address{
Mathematisches Institut,
Universit\"at Bonn,
Endenicher Allee 60, 53115 Bonn,
Germany
}
\email{henrich@math.uni-bonn.de}

\begin{abstract}
In this article, we develop a geometric method to construct solutions of the classical Yang-Baxter equation, attaching to the Weierstrass family of plane cubic curves and a pair of coprime positive integers, a family of classical r-matrices. It turns out  that all elliptic r-matrices arise in this way from smooth cubic curves. For the cuspidal cubic curve, we prove that the obtained solutions are rational and compute them explicitly. We also describe them in terms of Stolin's classification and prove that they are degenerations of the corresponding elliptic solutions.
\end{abstract}

\maketitle
\section{Introduction}

\noindent
Let $\mathfrak{g}$ be the Lie algebra $\mathfrak{sl}_n(\mathbb{C})$ and
$U = U(\mathfrak{g})$ be its universal enveloping algebra.
The classical Yang--Baxter equation (CYBE)  is
\begin{equation}\label{eq:CYBE}
[r^{12}(x_1, x_2), r^{13}(x_1, x_3)] + [r^{13}(x_1,x_3), r^{23}(x_2,x_3)] +
[r^{12}(x_1, x_2),
r^{23}(x_2, x_3)] = 0,
\end{equation}
where $r: (\CC^2, 0) \lar \mathfrak{g} \otimes \mathfrak{g}$ is the germ of a meromorphic function.
The upper indices in this equation   indicate various embeddings of
$\mathfrak{g}\otimes \mathfrak{g}$ into $U \otimes U \otimes U$. For example,
the function $r^{13}$ is defined as
$$
r^{13}: \mathbb{C}^2 \stackrel{r}\lar \mathfrak{g}\otimes \mathfrak{g}
\stackrel{\rho_{13}}\lar U \otimes U  \otimes U,
$$
where $\rho_{13}(x\otimes y) = x \otimes {1} \otimes y$. Two other
maps $r^{12}$ and $r^{23}$ have a similar meaning.

A solution of (\ref{eq:CYBE}) (also called \emph{$r$--matrix} in the physical literature) is  \emph{unitary}
if $r(x_1, x_2) = - \rho\bigl(r(x_2, x_1)\bigr)$, where
$\rho$ is the automorphism of $\mathfrak{g} \otimes \mathfrak{g}$
permuting both factors. A solution of (\ref{eq:CYBE}) is \emph{non-degenerate} if
its  image under the isomorphism
$$\mathfrak{g} \otimes \mathfrak{g} \lar \End(\mathfrak{g}), \quad a \otimes b \mapsto \bigl(c \mapsto
\tr(ac) \cdot b\bigr)$$ is an invertible operator for some (and   hence, for a generic) value of the spectral parameters $(x_1, x_2)$.
On the set of  solutions of (\ref{eq:CYBE})  there exists  a natural action of
the group of holomorphic function germs
$\phi: (\mathbb{C},0) \lar {\mathrm{Aut}}(\mathfrak{g})$ given by the rule
\begin{equation}\label{E:equiv}
r(x_1, x_2) \mapsto \tilde{r}(x_1, x_2) :=
\bigl(\phi(x_1) \otimes \phi(x_2)\bigr) r(x_1,x_2).
\end{equation}
It is easy to see that $\tilde{r}(x_1, x_2)$ is again a solution of (\ref{eq:CYBE}). Moreover,
$\tilde{r}(x_1, x_2)$ is unitary (respectively non-degenerate) provided $r(x_1, x_2)$ is
unitary (respectively non-degenerate). The solutions $r(x_1, x_2)$ and $\tilde{r}(x_1, x_2)$ related
by the formula (\ref{E:equiv}) for some $\phi$ are called
\emph{gauge equivalent}.

According to  Belavin and Drinfeld \cite{BelavinDrinfeld2},  any non-degenerate unitary
solution of the equation (\ref{eq:CYBE})  is gauge-equivalent to  a solution
 $r(x_1, x_2) = r(x_2-x_1)$  depending just on the difference (or  the quotient) of spectral
parameters. This means that  (\ref{eq:CYBE}) is essentially equivalent to the
equation
\begin{equation}\label{E:CYBE1}
\bigl[r^{12}(x), r^{13}(x+y)\bigr] + \bigl[r^{13}(x+y), r^{23}(y)\bigr] +
\bigl[r^{12}(x), r^{23}(y)\bigr] = 0.
\end{equation}
By a   result of Belavin and Drinfeld \cite{BelavinDrinfeld}, a non-degenerate
solution of (\ref{E:CYBE1}) is automatically unitary, has a simple pole at $0$ with the residue equal
to a multiple of the Casimir element, and
is either \emph{elliptic}  or \emph{trigonometric}, or \emph{rational}.  In \cite{BelavinDrinfeld},
 Belavin and Drinfeld also
gave  a complete classification of all elliptic and trigonometric solutions of (\ref{E:CYBE1}).
A  classification of rational solutions of (\ref{E:CYBE1}) was  achieved by Stolin
in \cite{Stolin, Stolin3}.

In this paper we study a connection  between the theory of vector bundles on curves
 of genus one and solutions of the classical Yang--Baxter equation (\ref{eq:CYBE}).
Let $E = V(w v^2 - 4 u^3 - g_2 uw^2 - g_3 w^3) \subset \PP^2$ be a Weiestra\ss{}  curve
over $\CC$,
$o \in E$ some fixed smooth point  and
$0 < d < n$ two coprime integers. Consider the sheaf of Lie algebras $\kA := \Ad(\kP)$, where $\kP$ is a simple vector bundle
$\kP$ of rank $n$ and degree $d$ on $E$ (note that up to an automorphism  $\kA$ does not depend on a particular
 choice of $\kP$). For any pair of distinct smooth points
$x, y$ of $E$,  consider the linear map $\kA \big|_{x} \longrightarrow \kA\big|_{y}$
defined  as the composition:
\begin{equation}\label{E:outlineofMain}
\kA\big|_{x} \stackrel{\res_x^{-1}}\lar H^0\bigl(\kA(x)\bigr) \stackrel{\ev_y}\lar \kA\big|_{y},
\end{equation}
where $\res_x$ is the \emph{residue map} and $\ev_y$ is the \emph{evaluation map}. Choosing
a trivialization  $\kA(U) \stackrel{\xi}\lar \slie_n\bigl(\kO(U)\bigr)$ of the sheaf of Lie algebras
$\kA$
for some small   neighborhood $U$ of $o$, we get the tensor
$r^{\xi}_{(E, (n, d))}(x, y) \in \lieg \otimes \lieg$.
The  first main result
of this  paper is the following.

\medskip
\noindent
\textbf{Theorem A}.
 In the above notations we have:
\begin{itemize}
\item The tensor-valued function  $r^\xi_{(E, (n, d))}: U \times U \lar \lieg \otimes \lieg$ is
meromorphic. Moreover, it is a non-degenerate unitary solution of the classical Yang--Baxter equation
(\ref{eq:CYBE}).
\item The function  $r^{\xi}_{(E, (n, d))}$ is  \emph{analytic} with respect to the parameters $g_2$ and $g_3$.
\item A different choice of trivialization $\kA(U) \stackrel{\zeta}\lar \slie_n\bigl(\kO(U)\bigr)$ gives
a gauge equivalent solution $r^\zeta_{(E, (n, d))}$.
\end{itemize}

Our next aim  is   to  describe all solutions of (\ref{E:CYBE1}) corresponding to elliptic curves.
Let $\varepsilon = \exp\bigl(\frac{2 \pi i d}{n}\bigr)$ and
$I := \bigl\{(p, q) \in \ZZ^2 \big| 0 \le p \le n-1, 0 \le q \le n-1, (p, q) \ne (0, 0)\bigr\}$.
Consider the following matrices
\begin{equation}\label{E:matricesXY}
X = \left(
\begin{array}{cccc}
1 & 0 & \dots & 0 \\
0 & \varepsilon & \dots & 0 \\
\vdots & \vdots & \ddots & \vdots \\
0 & 0 & \dots & \varepsilon^{n-1}
\end{array}
\right)
\quad
\mbox{and} \quad
Y = \left(
\begin{array}{cccc}
0 & 1 & \dots & 0 \\
\vdots & \vdots & \ddots & \vdots \\
0 & 0 & \dots & 1 \\
1 & 0 & \dots & 0
\end{array}
\right).
\end{equation}
For any $(k, l) \in I$ denote
$
Z_{k, l} = Y^k X^{-l} \quad \mbox{and} \quad
Z_{k, l}^\vee  = \frac{1}{n} X^{l} Y^{-k}.
$

\medskip
\noindent
\textbf{Theorem B}. Let $\tau \in \CC$ be such that $\mathrm{Im}(\tau) > 0$
 and
 $E = \CC/\langle 1, \tau\rangle$ be the corresponding complex torus. Let $0 < d < n$ be two
 coprime integers.
Then we have:
\begin{equation}\label{E:SolutionElliptic}
r_{(E, (n, d))}(x, y) = \sum\limits_{(k, l) \in I}
\exp\Bigl(-\frac{2 \pi i d}{n} k z\Bigr)
\sigma\Bigl(\frac{d}{n}\bigl(l - k \tau\bigr),  z\Bigr)
Z_{k, l}^\vee \otimes Z_{k, l},
\end{equation}
where $z = y - x$ and 
\begin{equation}\label{E:KroneckerEllFunct}
\sigma(a,  z) =
 2 \pi i \sum\limits_{n \in \ZZ} \frac{\exp(- 2 \pi i n z)}{1 - \exp\bigl(-2 \pi i (a - 2 \pi i n \tau)\bigr)}
\end{equation} is the Kronecker  elliptic function. Hence,
$r_{(E, (n, d))}$ is the elliptic $r$-matrix of Belavin \cite{Belavin}, see also
\cite[Proposition 5.1]{BelavinDrinfeld}.

\medskip
Our next goal is to describe the solutions of (\ref{eq:CYBE}) corresponding to the data
$(E, (n, d))$ for  the cuspidal  cubic curve
$E = V(w v^2 - u^3)$.  Using the classification of simple vector bundles on $E$ due to
 Bodnarchuk and Drozd \cite{BodnarchukDrozd} as well as methods  developed
 by Burban and Kreu\ss{}ler
 \cite{BK4}, we  derive an explicit recipe to compute
 the tensor $r^\xi_{(E, (n,d))}(x, y)$ from Theorem A.
 It turns out that the obtained  solutions of (\ref{eq:CYBE}) are always rational.
By  Stolin's classification \cite{Stolin, Stolin3}, the  rational
solutions of (\ref{eq:CYBE}) are parameterized by certain  Lie
algebraic objects, which we shall  call \emph{Stolin triples}. Such a
triple $(\mathfrak{l},k, \omega)$ consists of
\begin{itemize}
\item a Lie subalgebra $\mathfrak{l}\subseteq\mathfrak{g}$,
\item an integer $k$ such that $0 \leq k \leq n$,
\item a skew symmetric bilinear form $\omega:\mathfrak{l}\times\mathfrak{l}\rightarrow\mathbb{C}$
which is a 2-cocycle, i.e.
\[
\omega\bigl(\left[a,b\right],c\bigr)+\omega\bigl(\left[b,c\right],a\bigr)+
\omega\bigl(\left[c,a\right],b\bigr)=0
\]
for all $a, b, c \in \mathfrak{l}$,
\end{itemize}
such that for the $k$-th parabolic Lie subalgebra
 $\mathfrak{p}_{k}$ of $\mathfrak{g}$ (with $\liep_0 = \liep_n = \lieg$)
the following two conditions are fulfilled:

\begin{itemize}
\item $\mathfrak{l}+\mathfrak{p}_{k}=\mathfrak{g}$,
\item $\omega$ is non-degenerate on $\left(\mathfrak{l}\cap\mathfrak{p}_{k}\right)\times\left(\mathfrak{l}\cap\mathfrak{p}_{k}\right)$.
\end{itemize}

Let $0 < d < n$ be two
coprime integers, $e = n - d$.
We construct a certain matrix $J \in \Mat_{n \times n}(\CC)$ whose  entries
are equal to $0$ or $1$  (and their positions are uniquely determined by $n$ and $d$)
such that the pairing
 $$\omega_{J}: \mathfrak{p}_e \times \mathfrak{p}_e \lar \CC, \quad
 (a, b) \mapsto  \tr\bigl(J^{t}\cdot\left[a,b\right]\bigr)
 $$
 is  non-degenerate.
The following result was  conjectured by Stolin.

\medskip
\noindent
\textbf{Theorem C}.  Let $E$ be the cuspidal cubic curve
 and $0 < d < n$ a pair of coprime integers. Then the solution $r_{(E, (n, d))}$ of the classical Yang--Baxter equation (\ref{eq:CYBE}),
 described in  Theorem A,  is gauge equivalent to the solution
 $r_{(\mathfrak{g}, e, \omega_J)}$ attached
 to the Stolin triple $\left(\mathfrak{g}, e, \omega_J\right)$.

 \medskip
 \noindent
  Moreover, we derive an algorithm  to compute the solution $r_{(E, (n, d))}$ explicitly. In particular, for $d = 1$ this leads to  the following closed formula (see Example \ref{E:explicitSolut}):
 {\scriptsize{
\[
r_{(E, (n, 1))} \sim r_{(\lieg, n-1, \omega)} =
\frac{c}{y-x} +
\]
\[
 + x \left[e_{{1}, {2}}\otimes\check{h}_{1}-\sum_{j=3}^{n}e_{1,j}\otimes\left(\sum_{k=1}^{n-j+1}
e_{j+k-1,k+1}\right)\right]
- y \left[\check{h}_{1}\otimes e_{1,2}-
\sum_{j=3}^{n}\left(\sum_{k=1}^{n-j+1}e_{j+k-1,k+1}\right)\otimes e_{1,j}\right]
\]
\[
+\sum_{j=2}^{n-1}e_{1,j}\otimes\left(\sum_{k=1}^{n-j}e_{j+k,k+1}\right)+
\sum_{i=2}^{n-1}e_{i,i+1}\otimes\check{h}_{i}
-\sum_{j=2}^{n-1}\left(\sum_{k=1}^{n-j}e_{j+k,k+1}\right)\otimes e_{1,j}-\sum_{i=2}^{n-1}\check{h}_{i}\otimes e_{i,i+1} \]
\[
+\sum_{i=2}^{n-2}\left(\sum_{k=2}^{n-i}\left(\sum_{l=1}^{n-i-k+1}e_{i+k+l-1,l+i}\right)\otimes e_{i,i+k}\right)
-\sum_{i=2}^{n-2}\left(\sum_{k=2}^{n-i}e_{i,i+k}\otimes\left(\sum_{l=1}^{n-i-k+1}
e_{i+k+l-1,l+i}\right)\right),
\]
}
}

\medskip
\noindent
where $c$ is the Casimir element in $\lieg \otimes \lieg$ with respect to the trace form,
$e_{i, j}$ are the matrix units for
$1 \le i, j \le n$,
and $\check{h}_l$ is the dual of $h_l = e_{l, l} - e_{l+1, l+1}$,
 $1 \le l \le n-1$. Theorem A implies, that up to  a certain (not explicitly known) gauge transformation
 and a change of variables, this rational solution is a degeneration of the Belavin's elliptic $r$--matrix (\ref{E:SolutionElliptic}) for $\varepsilon = \exp\bigl(\frac{2 \pi i}{n}\bigr)$. It seems that it is rather difficult to prove this result using just direct analytic methods.

\medskip
\noindent
 Finally, we  show that the solutions
 $r_{(E, (n, d))}$  and $r_{(E, (n, e))}$  are gauge equivalent.

\medskip
\noindent
\textbf{Notations and terminology}. In this article we shall use  the following notations.

\medskip
\noindent
$-$ $\kk$ denotes an algebraically closed field
of characteristic zero.

\medskip
\noindent
$-$ Given an algebraic variety $X$, $\Coh(X)$ respectively
$\VB(X)$ denotes the category of coherent sheaves respectively vector bundles on $X$. We denote
$\kO$ the structure sheaf of $X$.
Of course, the theory of Yang--Baxter equations is mainly interesting in  the case $\kk =
\mathbb{C}$.
In that  case, one  can (and probably should)  work in  the \emph{complex analytic category}. However, all relevant results and proofs of this article  remain valid in that case, too.

\medskip
\noindent
$-$ We denote by $\Dbcoh(X)$ the triangulated category of bounded complexes of $\kO$--modules
with coherent cohomology, whereas $\Perf(X)$ stands
 for the triangulated category of perfect complexes,  i.e.~
the full subcategory of $\Dbcoh(X)$ admitting a bounded locally free resolution.

\medskip
\noindent
$-$
We always write
$\Hom$, $\End$ and $\Ext$ when working with global morphisms and extensions between
coherent sheaves whereas $\Lin$ is used when we work with
vector spaces. If not explicitly otherwise stated, $\Ext$ always stands for $\Ext^1$.

\medskip
\noindent
$-$
For a vector bundle  $\kF$ on $X$ and $x  \in X$,  we denote by
$\kF\big|_x$ the fiber of $\kF$ over $x$, whereas $\kk_x$ denotes the skyscraper sheaf of length
one supported at $x$.

\medskip
\noindent
$-$
A \emph{Weierstra\ss{} curve} is a plane projective cubic curve given in homogeneous coordinates
by an equation $wv^2 = 4 u^3 +  g_2 uw^2 + g_3 w^3$, where $g_1, g_2 \in \kk$.
Such a curve is always irreducible. It is singular  if and only
if $\Delta(g_2, g_3) = g_2^3 + 27 g_3^2  =  0$. Unless $g_{2}=g_{3}=0$, the
  singularity is a \emph{node}, whereas in the case
  $g_{2}=g_{3}=0$ the singularity is a \emph{cusp}.

\medskip
\noindent
$-$
 A  \emph{Calabi--Yau curve}  $E$ is a reduced projective Gorenstein curve with trivial dualizing sheaf. Note that the complete list of such curves is actually known, see for
    example \cite[Section 3]{Ishii}: $E$ is either
 \begin{itemize}
    \item an elliptic curve,
    \item a Kodaira cycle of $n \ge 1 $ projective lines (for $n = 1$ it is a  nodal Weierstra\ss{} curve), also called Kodaira fiber of type $\mathrm{I}_n$,
    \item a cuspidal plane cubic curve (Kodaira fiber II), a tachnode cubic curve (Kodaira fiber III) or a generic configuration of $n$ concurrent lines in $\mathbb{P}^{n-1}$ for $n \ge 3$.
        \end{itemize}
 The irreducible Calabi--Yau curves are precisely the Weierstra\ss{} curves. For a Calabi--Yau curve
 $E$ we denote by $\breve{E}$ the regular part of $E$.

\medskip
\noindent
$-$
  We denote by $\Omega$    the sheaf of regular differential one--forms on a Calabi--Yau curve $E$, which we always view as a dualizing sheaf. Taking a non-zero section $w \in H^0(\Omega)$, we get
an isomorphism of $\kO$--modules $\kO \stackrel{w}\longrightarrow \Omega$.

\medskip
\noindent
$-$ Next, $\kP$ will always denote a simple  vector bundle  on a Calabi--Yau curve
$E$, i.e.~ a locally free coherent sheaf satisfying
$\End(\kP) = \kk$. Note that we automatically
have:  $\Ext(\kP, \kP) \cong \kk$.

\medskip
\noindent
$-$
Finally, for $n \ge 2$ we denote $\liee = \mathfrak{gl}_n(\kk)$ and $\lieg =
\slie_n(\kk)$. For $1 \le k \le n-1$ we denote by $\liep_k$ the $k$-th parabolic subalgebra
of $\lieg$.

\medskip
\noindent
\textbf{Plan of the paper and overview of methods and results}.

\medskip
\noindent
The main message of this
 article  is  the following:  to any triple $(E, (n, d))$, where
\begin{itemize}
\item $E$ is a Weierstra\ss{} curve,
\item  $0 < d < n$ is a pair of coprime integers,
\end{itemize}
one  can \emph{canonically} attach a solution $r_{(E, (n, d))}$ of the
classical Yang--Baxter equation (\ref{eq:CYBE})
for the Lie algebra $\lieg = \slie_n(\CC)$, see Section \ref{S:GenusOneandCYBE}.  The construction goes as follows.

\medskip
\noindent
Let $\kP$ be a simple vector bundle of rank $n$ and degree $d$ on $E$ and
$\kA = \mathsf{Ad}(\kP)$ be the sheaf of traceless endomorphisms of $\kP$. Obviously, $\kA$ is a sheaf
of Lie algebras on $E$ satisfying $H^0(\kA) = 0 = H^1(\kA)$. In can be shown that  $\kA$ does not depend
on the particular choice of $\kP$ and up to an isomorphism  determined by  $n$ and $d$, see
Proposition \ref{P:AdonCY}.

\medskip
\noindent
Let $x, y$ be a pair of smooth points of $E$.
Since the triangulated category $\Perf(E)$ has a (non-canonical) structure of an $A_\infty$--category,
we have the following linear map
$$
\mm_3: \Hom(\kP, \kk_x) \otimes \Ext(\kk_x, \kP) \otimes \Hom(\kP, \kk_y) \lar
\Hom(\kP, \kk_x).
$$
Using Serre duality, we get from here the  induced linear map
$$
\overline{\mm}_{x, y}: \slie\bigl(\Hom(\kP, \kk_x)\bigr) \lar \mathfrak{pgl}\bigl(\Hom(\kP, \kk_y)\bigr)
$$
and  the corresponding tensor
$\mm_{x, y} \in
\mathfrak{pgl}\bigl(\Hom(\kP, \kk_x)\bigr) \otimes \mathfrak{pgl}\bigl(\Hom(\kP, \kk_y)\bigr)$.
It turns out that this element $m_{x, y}$ is a \emph{triangulated invariant} of $\Perf(E)$, i.e.~it
does not depend on a  (non-canonical) choice of an $A_\infty$--structure on the category $\Perf(E)$.

\noindent
Let $E$ be  an elliptic curve. According to Polishchuk, \cite[Theorem 2]{Polishchuk1},
the tensor  $\mm_{x, y}$ is unitary and satisfies the classical Yang--Baxter equation
\begin{equation}\label{E:AinftyCYBE}
\bigl[\mm^{12}_{x_1, x_2}, \mm^{13}_{x_1, x_3}\bigr] + \bigl[\mm^{12}_{x_1, x_2}, \mm^{23}_{x_2, x_3}\bigr]
+ \bigl[\mm^{12}_{x_1, x_2}, \mm^{13}_{x_1, x_3}\bigr] = 0.
\end{equation}
Relation (\ref{E:AinftyCYBE}) follows from  the following two ingredients.
\begin{itemize}
\item  The $A_\infty$--constraint
$$
\mm_3 \circ \bigl(\mm_3 \otimes \mathbbm{1} \otimes \mathbbm{1}  +  \mathbbm{1} \otimes
\mm_3 \otimes \mathbbm{1}  + \mathbbm{1} \otimes \mathbbm{1} \otimes
\mm_3\bigr) + \dots = 0
$$
on
 the triple product $\mm_3$.
   \item Existence of a cyclic $A_\infty$--structure with respect of  the canonical
Serre--pairing on the triangulated category  $\Perf(E)$.
\end{itemize}
The unitarity of $\mm_{x, y}$ follows from existence of a cyclic $A_\infty$ structure as well.
To generalize the relation (\ref{E:AinftyCYBE}) on singular Weierstra\ss{}  curves as
well as  on the
relative situation of genus one fibrations, we need the following result (Theorem \ref{T:main1}): the diagram
\begin{equation}\label{E:diagcomp}
\begin{array}{c}
\xymatrix{
\kA\big|_x \ar[rr]^-{Y_1}  & & \slie\bigl(\Hom(\kP, \kk_x)\bigr) \ar[dd]^{\overline{\mm}_{x, y}} \\
H^0\bigl(\kA(x)\bigr) \ar[u]^{\res_x} \ar[d]_{\ev_y} & &  \\
\kA\big|_y \ar[rr]^-{Y_2} & & \mathfrak{pgl}\bigl(\Hom(\kP, \kk_y)\bigr)
}
\end{array}
\end{equation}
is commutative,
where
$Y_1$ and $Y_2$ are certain canonical anti-isomorphisms of Lie algebras. A version of this important
fact has been stated in \cite[Theorem 4(b)]{Polishchuk1}.

Using the commutative diagram (\ref{E:diagcomp}), we prove Theorem A. As a  consequence, we obtain  the continuity property of the solution
$r_{(E, (n, d))}$ with respect to the Weierstra\ss{} parameters $g_2$ and $g_3$ of the curve $E$. This
actually  leads to certain unexpected analytic consequences about classical $r$--matrices, see
Corollary \ref{C:AnalyticConseq}.

The above  construction can be rephrased in the following  way. Let $E$ be an arbitrary
 Weierstra\ss{} cubic curve. Then there exists a canonical meromorphic section
 $$
 r \in \Gamma\bigl(\breve{E} \times \breve{E}, p_1^{*}\kA \otimes p_2^{*}\kA\bigr),
 $$
 where $p_1, p_2: \breve{E} \times \breve{E} \rightarrow E$ are canonical projections,
 satisfying the equation
 $$
 \bigl[r^{12}, r^{13}\bigr] + \bigl[r^{13}, r^{23}\bigr] + \bigl[r^{12}, r^{23}\bigr] = 0,
 $$
 see Theorem \ref{T:main2}.
It seems that in  the case of elliptic curves,
similar ideas have been  suggested already in 1983 by Cherednik  \cite{Cherednik}.  For an elliptic curve $E$ with a marked point $o \in E$, the Lie algebra
$\mathfrak{sel}_{(E, (n, d))}: = \Gamma\bigl(E\setminus \{o\}, \kA\bigr)$ was studied  by  Ginzburg, Kapranov and Vasserot \cite{GinzburgKapranovVasserot}, who constructed its realization using ``correspondences'' in the spirit of the geometric representation theory.

Talking about the proposed method of constructing of solutions of the classical Yang--Baxter equation,
one may pose following natural question: to what extent is this method \emph{constructive}?
It turns out, that one can end up
with explicit solutions in the case of all types of the Weierstra\ss{} curves. See also
\cite{BK4}, where the similar technique in the case of solutions of the associative Yang--Baxter equation
has been developed.

We first show that for an elliptic curve $E$, the corresponding solution $r_{(E, (n, d))}$ is the 
\emph{elliptic} $r$--matrix of Belavin  given by the formula (\ref{E:SolutionElliptic}), see Theorem \ref{T:main3}.  This result can be also deduced
from \cite[Formula (2.5)]{Polishchuk1}. However, Polishchuk's proof, providing on one side a spectacular and impressive application of methods  of   mirror symmetry, is on the other hand   rather  undirect, as it requires the strong $A_\infty$--version of the homological mirror symmetry  for elliptic curves, explicit formulae for  higher products
 in the Fukaya category of a torus and finally  leads to a more complicated expression than
(\ref{E:SolutionElliptic}).

Next,  we focus on solutions of (\ref{eq:CYBE}) originating
from the cuspidal cubic curve $E = V(uv^2 - w^3)$.  The motivation to deal with this problem
comes from the fact that all
obtained solutions turn out to be \emph{rational}, which is the most complicated  class of solutions
from the point of view
of the Belavin--Drinfeld classification \cite{BelavinDrinfeld}.  Our approach is based
on the general methods of study of vector bundles on the singular curves of genus one developed
in \cite{DrozdGreuel,Thesis, Survey} and especially on Bodnarchuk--Drozd classification \cite{BodnarchukDrozd}
of  simple vector bundles on $E$. The above abstract way to construct solutions
of (\ref{eq:CYBE}) can be reduced  to a very explicit recipe  (see Algorithm \ref{Algorithm geometric}),
summarized
as follows.
\begin{itemize}
\item To any pair of positive coprime integers $d, e$ such that $e + d = n$ we  attach a certain matrix
$J = J_{(e, \, d)} \in \Mat_{n \times n}(\CC)$, whose entries are either $0$ or $1$.
\item For any  $x \in \CC$, the
matrix  $J$ defines a certain linear subspace $\Sol\bigl((e, d), x)\bigr)$ in the
Lie algebra of currents $\lieg[z]$. For any $x \in \CC$, we denote  the evaluation map by  $\phi_x: \lieg[z]
\rightarrow \lieg$.
\item Let   $\overline{\res}_x := \phi_x$ and
$\overline{\ev}_y := \frac{\displaystyle 1}{\displaystyle y-x} \phi_y$. It turns out that $\overline{\res}_x$ and $\overline{\ev}_y$  yield
isomorphisms between $\Sol\bigl((e, d), x)\bigr)$ and $\lieg$. Moreover, these maps
 are just the  coordinate versions of the sheaf-theoretic morphisms
$\res_x: H^0\bigl(\kA(x)\bigr) \rightarrow \kA\big|_x$ and $\ev_y: H^0\bigl(\kA(x)\bigr) \rightarrow \kA\big|_y$ appearing in  the diagram (\ref{E:diagcomp}).
\end{itemize}

The constructed matrix $J$ turns out to be useful in a completely different situation. Let  $\liep = \mathfrak{p}_e$ denote
the $e$-th parabolic subalgebra of $\lieg$. This Lie algebra  is known to be \emph{Frobenius}, see for example \cite{Elashvili1982} and \cite{Stolin}. We prove (see Theorem \ref{T:FrobAlg}) that the pairing
 $$\omega_{J}: \mathfrak{p} \times \mathfrak{p} \lar \CC, \quad
 (a, b) \mapsto  \tr\bigl(J^{t}\cdot\left[a,b\right]\bigr)
 $$
is  non-degenerate, making the Frobenius structure on $\mathfrak{p}$ explicit. This result will  later be used to get explicit  formulae for the  solutions $r_{(E, (n, d))}$.

The study of rational solutions of  the classical Yang--Baxter equation (\ref{eq:CYBE}) was a subject
of Stolin's investigation \cite{Stolin,Stolin3,Stolin2}. The first basic fact of his theory states that
the gauge equivalence classes of rational solutions of (\ref{eq:CYBE}) with values
in $\lieg$, which   satisfy a certain additional Ansatz on the residue,  stand in  bijection
with the conjugacy classes of  certain Lagrangian Lie subalgebras $\liew \subset \lieg((z^{-1}))$
called \emph{orders}. The  second basic result of Stolin's theory  states that Lagrangian orders are parameterized (although not in a unique way) by  certain triples $(\liel, k, \omega)$ mentioned in the Introduction.

The problem of description of all Stolin triples $(\liel, k, \omega)$ is known to be \emph{representation-wild}, as it contains as a subproblem \cite{BelavinDrinfeld, Stolin} the wild problem
of classification of all abelian Lie subalgebras of $\lieg$ \cite{Drozd}. Thus, it is natural to ask what
Stolin triples $(\liel, k, \omega)$ correspond to the
``geometric'' rational solutions $r_{(E, (n, d))}$, since the latter ones have discrete
combinatorics and obviously form a ``distinguished''
class of rational solutions.  This problem is completely solved in  Theorem C, what is the third  main result of this  article.

\medskip
\noindent
\emph{Acknowledgement}. This  work  was supported  by the DFG project Bu--1866/2--1. We are grateful    to   Alexander Stolin for introducing us in his theory rational solutions of the classical Yang--Baxter equation and sharing his ideas.

\section{Some algebraic and geometric preliminaries}
In this section we collect some known basic facts from linear algebra, homological algebra, and the
 theory of vector bundles on Calabi--Yau curves, which are crucial
for further applications.

\subsection{Preliminaries from linear algebra}
For a finite dimensional vector space $V$ over $\kk$ we denote by $\slie(V)$
the Lie subalgebra of $\End(V)$ consisting of endomorphisms with zero trace and
$\pgl(V) := \End(V)/\langle \mathbbm{1}_V\rangle$. Since the proofs of all statements from this subsection
are completely elementary, we left  them  to the reader as an exercise.

\begin{lemma}\label{L:LA1}
The non-degenerate bilinear pairing $\tr: \End(V) \times \End(V) \lar \kk$, $(f, g) \mapsto
\tr(fg)$ induces  another non-degenerate pairing $\tr: \slie(V) \times \pgl(V) \lar \kk$,
$(f, \bar{g}) \mapsto \tr(fg)$. In particular, for any finite dimensional vector space $U$ we get
 a canonical isomorphism
of vector spaces
\begin{equation*}
\pgl(U) \otimes \pgl(V) \lar \Lin\bigl(\slie(U), \pgl(V)\bigr).
\end{equation*}
\end{lemma}

\begin{lemma}\label{L:LA2}
The \emph{Yoneda map} $Y: \End(V) \lar \End(V^*)$, assigning to an endomorphism $f$ its
 adjoint $f^*$, induces an \textsl{anti-isomorphisms of Lie algebras}
\begin{itemize}
\item $Y_1: \slie(V) \lar \slie(V^*)$ and
\item $Y_2: \slie(V) \lar \pgl(V^*)$, $f \mapsto \bar{f^*}$, where $\bar{f^*}$ is the
equivalence class
of $f^*$.
\item The following diagram
$$
\xymatrix{
\slie(V) \times \slie(V) \ar[rr]^{Y_1 \times Y_2} \ar[rd]_{\tr} & & \slie(V^*) \times \pgl(V^{*}) \ar[ld]^{\tr} \\
& \kk &
}
$$
is commutative.
\end{itemize}
\end{lemma}

\noindent
Note that the fist part of the statement is valid for any field $\kk$, whereas the second one is only true
if $\mathsf{dim}_{\kk}(V)$ is invertible in $\kk$.

\begin{lemma}\label{L:LA3}
Let $H \subseteq V$ be a linear subspace.
Then we have the canonical linear map $r_H: \End(V) \lar \Lin(H, V/H)$ sending
an endomorphism $f$ to the composition  $H \lar V \stackrel{f}\lar V \lar V/H$. Moreover,
the following results are true.
\begin{itemize}
\item We have: $r_H(\mathbbm{1}_V) = 0$. In particular, there is an induced canonical map
$\bar{r}_H: \pgl(V) \lar \Lin(H, V/H)$.
\item Let $f \in \End(V)$ be such that for any one-dimensional subspace $H \subseteq V$ we have:
$r_H(f) = 0$. Then $\bar{f} = 0$ in $\pgl(V)$.
\item Let $U$ be a finite dimensional vector space and $g_1, g_2: U \lar \pgl(V)$  be two linear maps
such that for any one-dimensional subspace $H \subseteq V$ we have:
$\bar{r}_H \circ g_1 = \bar{r}_H \circ g_2$. Then $g_1 = g_2$.
\end{itemize}
\end{lemma}

\subsection{Triple Massey products}
In this article, we use the notion of triple Massey products in the following special situation.

\begin{definition}\label{D:MasseyProd}
Let $\mathsf{D}$ be a $\kk$--linear triangulated category, $\kP$, $\kX$ and $\kY$ some objects
of $\catD$ satisfying the following conditions:
\begin{equation}\label{E:InputForMassey}
\End(\kP) = \kk \quad \mathrm{and} \quad
     \Hom(\kX, \kY) = 0 = \Ext(\kX, \kY).
\end{equation}
     Consider the linear subspace
\begin{equation}\label{E:linearspaceK}
K := \mathsf{Ker}\bigl(\Hom(\kP, \kX) \otimes \Ext(\kX, \kP) \stackrel{\circ}\lar
\Ext(\kP, \kP)\bigr).
\end{equation}
and a linear subspace
$H \subseteq \Hom(\kP, \kY)$.
The triple Massey product is the  map
\begin{equation}\label{E:tripleMassey}
{M}_H: K \lar \Lin\bigl(H, \Hom(\kP, \kY)/H\bigr)
\end{equation}
defined as follows. Let $t = \sum_{i= 1}^p f_i \otimes \omega_i \in K$ and
$h \in H$. Consider the following commutative diagram in the triangulated category $\catD$:
\begin{equation*}
\begin{array}{c}
\xymatrix{
& & \kP \ar[ddl]_-{\tilde{f}} \ar[dd]^{f = \left(\begin{smallmatrix} f_1 \\ \vdots \\ f_p\end{smallmatrix}\right)} & & \\
& & & & \\
\kP \ar[r]^{\imath} \ar[d]_h & \kA \ar[r]^-{p} \ar[dl]^-{\tilde{h}} & \kX \oplus \dots \oplus \kX \ar[rr]^-{(\omega_1, \dots, \omega_p)}
&  & \kP[1]. \\
\kk_y & & & & \\
}
\end{array}
\end{equation*}
The horizontal sequence is a distinguished triangle in $\catD$ determined
by the morphism $(\omega_1, \dots, \omega_p)$.
Since $\sum_{i=1}^p \omega_i f_i = 0$ in $\Ext(\kP, \kP)$, there exists a morphism
$\tilde{f}: \kP \lar \kA$ such that $p \tilde{f} = f$. Note that such a morphism is only defined
 up to a translation $\tilde{f} \mapsto \tilde{f} + \lambda \imath$ for some
$\lambda \in \kk$. Since $\Hom(\kX, \kY) = 0 = \Ext(\kX, \kY)$, there exists a unique
morphism $\tilde{h}: \kA \lar \kY$ such that $\tilde{h} \imath = h$.
We set: $
\bigl(M_H(t)\bigr)(h) := \overline{\tilde{h} \tilde{f}}.
$
\qed
\end{definition}

\noindent
The following result is   well-known, see for instance \cite[Exercise IV.2.3]{GelfandManin}.
\begin{proposition}
The map $M_H$ is well-defined, i.e.~it is independent
of a  presentation of $t \in K$ as a sum of simple tensors and  a choice of the horizontal distinguished triangle. Moreover, $M_H$  is $\kk$--linear.
\end{proposition}

\subsection{$A_\infty$--structures and triple Massey products}
Let $\catB$ be a $\kk$--linear Gro\-then\-dieck abe\-lian category,
$\catA$ be its full subcategory of Noetherian objects
and $\catI$ the full subcategory of injective objects. For simplicity, we assume
$\catA$ to be $\Ext$--finite.
The derived category $D^+(\catB)$ is equivalent to the homotopy
category $\Hot+(\catI)$. This identifies the triangulated category $\catD = D^b_{\catA}(\catB)$
of complexes with cohomology from $\catA$ with the corresponding full subcategory of
 $\Hot+(\catI)$.
 Since
$\Hotb(\catI)$ is the  homotopy category of the dg--category
 $\Comb(\catI)$, by the homological perturbation
lemma of Kadeishvili \cite{Kadeishvili}, the triangulated category $\catD$ inherits a structure of an $A_\infty$--category. This means that for any $n \ge 2$,
 $i_1, i_2, \dots, i_n \in \ZZ$ and objects $\kF_0, \kF_1, \dots, \kF_n$ of $\catD$,   we have linear maps
 $$
 \mathsf{m}_n: \Ext^{i_1}(\kF_0, \kF_1) \otimes
 \Ext^{i_2}(\kF_1, \kF_2) \otimes \dots \otimes \Ext^{i_n}(\kF_{n-1}, \kF_n) \lar
 \Ext^{i_1 + \dots + i_n + (2-n)}(\kF_0, \kF_n)
 $$
 satisfying the identities
 \begin{equation}\label{E:structureAinfty}
 \sum\limits_{\substack{r, s, t \ge 0\\ r+s+t = n}}
 (-1)^{r+st} \mm_{r+1+t}\bigl(
 \underbrace{\mathbbm{1} \dots \otimes \mathbbm{1}}_{r \; times} \otimes \mm_s \otimes
 \underbrace{\mathbbm{1} \dots \otimes \mathbbm{1}}_{s \; times}\bigr) = 0,
 \end{equation}
 where $\mm_2$ is  the composition  of morphisms  in $\catD$. The higher operations
 $\bigl\{\mm_n\}_{n \ge 3}$ are unique up to an $A_\infty$--automorphism of
 $\catD$. On the other hand, they  are  \emph{not} determined by the triangulated structure of $\catD$, although
 they turn out to be compatible with the Massey products. Throughout this subsection,
 we fix  some $A_\infty$--structure $\bigl\{\mm_n\}_{n \ge 3}$ on $\catD$.

\medskip
\noindent
Assume we have object $\kP$, $\kX$ and $\kY$ of $\catD$ satisfying the conditions of Definition
\ref{D:MasseyProd}.
Consider the linear map
$$
\mm = \mm_3^{\infty} \;:\;
\Hom(\kP, \kX) \otimes \Ext(\kX, \kP) \otimes \Hom(\kP, \kY) \lar \Hom(\kP, \kY).
$$
It induces  another linear map
$
K \lar \End\bigl(\Hom(\kP, \kk_y)\bigr)
$
assigning to  an element $t \in K$ the functional $g \mapsto \mm(t \otimes g)$.
Composing this map with
 the canonical projection $\End\bigl(\Hom(\kP, \kY)\bigr)
\longrightarrow \mathfrak{pgl}\bigl(\Hom(\kP, \kY)\bigr)$,
we obtain the linear map
\begin{equation}\label{E:MasseyMap}
\mm^\kP_{\kX, \kY}: \; K \lar \mathfrak{pgl}\bigl(\Hom(\kP, \kY)\bigr).
\end{equation}
\begin{lemma}
The  map  $\mm^\kP_{\kX, \kY}$ does not depend on the choice of an
$A_\infty$--structure on $\catD$.
\end{lemma}

\begin{proof}
Of coarse, we may without loss of generality assume that $\Hom(\kP, \kY) \ne 0$.
First note  that for any choice of an $A_\infty$--structure on $\catD$
and any one-dimensional linear subspace $H \subseteq \Hom(\kP, \kY)$,  the following
diagram
\begin{equation}\label{E:compatMasseyAinfty}
\begin{array}{c}
\xymatrix
{
K \ar[rd]_-{M_H} \ar[rr]^-{{\mm}^\kP_{\kX, \kY}} & & \pgl\bigl(\Hom(\kP, \kY\bigr) \ar[ld]^-{\bar{r}_H} \\
 & \Lin\bigl(H, \Hom(\kP, \kY)/H\bigr) &
}
\end{array}
\end{equation}
is commutative.
Here,   $M_H$ is the triple Massey product (\ref{E:tripleMassey})
 and $\bar{r}_H$ is the canonical linear
map from Lemma \ref{L:LA3}. This compatibility between the triangulated Massey products and
 higher $A_\infty$--products is well-known, see for example
 \cite{LuPalmieri} a proof of a much more general statement.
 Let $\bigl\{\underline{\mm}_n\bigr\}_{n \ge 3}$ be another $A_\infty$--structure on
$\catD$. From  the last part of Lemma \ref{L:LA3} it follows that  ${\mm}^\kP_{\kX, \kY} =
{\underline{\mm}}^\kP_{\kX, \kY}$.
This proves the claim.
\end{proof}

\subsection{On the sheaf of Lie algebras $\mathsf{Ad}(\kF)$}

Let $X$ be an algebraic variety over $\kk$ and $\kF$ a  vector bundle on $X$.

\begin{definition} The locally free sheaf  $\Ad(\kF)$ of the traceless endomorphisms of $\kF$
is defined by via the following short exact sequence
\begin{equation}\label{E:shortexactAd}
0 \lar \Ad(\kF) \lar {\mathcal End}(\kF) \stackrel{\Tr_\kF}{\lar} \kO \lar 0,
\end{equation}
where $\Tr_\kF: {\mathcal End}(\kF) \lar \kO$ is the canonical trace map.
\end{definition}

\noindent
In the proposition below we collect some basic facts on the vector bundle $\Ad(\kF)$.

\begin{proposition}\label{P:basiconAd} In the above notation the following statements are true.
\begin{itemize}
\item The vector bundle $\Ad(\kF)$ is  a sheaf of Lie algebras on $X$.
\item Next, we have: $H^0\bigl(\Ad(\kF)\bigr) = 0$.
\item For any $\kL \in \Pic(X)$ we have the  natural isomorphism of sheaves of Lie algebras
$\Ad(\kF) \lar \Ad(\kF \otimes \kL)$ which is induced by the natural isomorphism of sheaves of algebras
$ {\mathcal End}(\kF) \lar {\mathcal End}(\kF \otimes \kL)$.
\item We have a symmetric  bilinear pairing
$
\Ad(\kF) \times \Ad(\kF) \lar \kO
$
given on the level of local sections by the rule $(f, g) \mapsto \tr(fg)$. This pairing  induces
an isomorphism of $\kO$--modules $\Ad(\kF) \lar \Ad(\kF)^\vee$.
\end{itemize}
\end{proposition}

\subsection{Serre duality pairing on a Calabi--Yau curve}
Let $E$ be a Calabi--Yau curve and  $w \in H^0(\Omega)$ a no-where vanishing
regular differential form.
 For any pair of objects $\kF, \kG \in \Perf(E)$ we have the bilinear form
 \begin{equation}\label{E:SerrePairing}
 \langle -\;,\;-\rangle = \langle -\;,\;-\rangle_{\kF,\,  \kG}^w \; : \;   \Hom(\kF, \kG) \times \Ext(\kG, \kF) \lar \kk
 \end{equation}
 defined as the composition
 $$
 \Hom(\kF, \kG) \times \Ext(\kG, \kF) \stackrel{\circ}\lar \Ext(\kF, \kF)
 \stackrel{\Tr_{\kF}}\lar H^1(\kO) \stackrel{w}\lar
 H^1(\Omega) \stackrel{t}\lar \kk,
 $$
 where $\circ$ denotes the composition operation, $\Tr_\kF$ is the trace map and $t$ is the canonical morphism described in  \cite[Subsection 4.3]{BK4}.
The following result is well-known, see for example \cite[Corollary 3.3]{BK4} for a proof.

\begin{theorem}
For any $\kF, \kG \in \Perf(E)$ the pairing  $\langle -\,,\,-\rangle_{\kF, \; \kG}^w$ is non-degenerate. In particular, we have an isomorphism of vector spaces
\begin{equation}\label{E:SerreMap}
\SS = \SS_{\kF, \; \kG}: \quad \Ext(\kG, \kF) \lar \Hom(\kF, \kG)^*,
\end{equation}
which is functorial in both arguments.
\end{theorem}

\noindent
Let $\kP$ be a simple vector bundle on $E$ and $x, y \in \breve{E}$ a pair of points from
the same irreducible components. Note that we are in the situation of Definition \ref{D:MasseyProd}
for $\catD = \Perf(E)$, $\kX = \kk_x$ and $\kY = \kk_y$.
Note the following
easy fact.

\begin{lemma}\label{L:SLandSerre}
Let $K$ be as in (\ref{E:linearspaceK}).  Then the linear isomorphism
$$
\overline{\SS}: \Hom(\kP, \kk_x) \otimes \Ext(\kk_x, \kP) \xrightarrow{\mathbbm{1} \,\otimes\;  \SS}
\Hom(\kP, \kk_x) \otimes \Hom(\kP, \kk_x)^* \stackrel{\ev}\lar \End\bigl(\Hom(\kP, \kk_x)\bigr)
$$
identifies the vector space $K$  with $\mathfrak{sl}\bigl(\Hom(\kP, \kk_x)\bigr)$.
\end{lemma}

\subsection{Simple vector bundles on Calabi--Yau curves} In this subsection, we collect some
basic results  on the classification of vector bundles on (possibly reducible)
Calabi--Yau curves.
\begin{definition}
Let $\bigl\{E^{(1)}, \dots, E^{(p)}\bigr\}$ be the set of the irreducible components
of a Calabi--Yau curve $E$. For a vector bundle $\kF$ on $E$ we denote by
$$
\underline{\deg}(\kF) = (d_1, \dots, d_p) \in \ZZ^p
$$
its  \emph{multi-degree}, where $d_i = \deg\bigl(\kF\big|_{E^{(i)}}\bigr)$ for $1 \le i \le p$.

\vspace{2mm}
\noindent
For  $\mathbbm{d} \in \ZZ^p$ we denote
$
\Pic^{\mathbbm{d}}(E) := \bigl\{\kL \in \Pic(E) \,\big| \, \underline{\deg}(\kL) = \mathbbm{d}\bigr\}.
$
In particular, for $\mathbbm{o} = (0, \dots, 0)$ we set:
$
J(E) = \Pic^{\mathbbm{o}}(E).
$
Then $J(E)$ is an algebraic group called \emph{Jacobian} of $E$.
\end{definition}

\begin{proposition} For  $\kk = \CC$  we have the following isomorphisms of Lie groups:
\begin{equation*}
J(E)  \cong
\begin{cases}
  \CC/\Lambda       & \text{if } \; $E$ \; \text{ is elliptic,}\\
  \CC^* & \text{if } \; $E$ \; \text{ is a Kodaira cycle,}\\
  \CC               &  \;  \text{in the remaining cases}.
\end{cases}
\end{equation*}
Moreover, for any multi-degree $\mathbbm{d}$ we have a (non-canonical) isomorphism
of algebraic varieties $J(E) \lar \Pic^{\mathbbm{d}}(E)$.
\end{proposition}

\noindent
A proof of this result follows from  \cite[Exercise II.6.9]{Hartshorne} or
\cite[Theorem 16]{Survey}.

\medskip
\noindent
Next, recall the description of  simple vector bundles on Calabi--Yau curves.

\begin{theorem}\label{T:simplebundles}
 Let $E$ be a reduced \textsl{plane cubic curve} with $p$ irreducible components
and $\kP$ be a simple vector bundle on $E$.
 Then the following statements are true.
\begin{itemize}
\item Let $n = \rk(\kP)$ be the rank of $\kP$ and $d = d_1(\kP) + \dots + d_p(\kP) = \chi(\kP)$
its degree. Then $n$ and $d$ are mutually prime.
\item If $E$ is irreducible then $\kP$ is stable.
\item Let $n \in \mathbb{N}$ and $\mathbbm{d} = (d_1, \dots, d_p) \in \ZZ^p$ be such
that $\gcd(n, d_1 + \dots + d_p) = 1$. Denote by  $M_E(n, \mathbbm{d})$ the set of simple
vector bundles on $E$ of rank $n$ and multi-degree $\mathbbm{d}$. Then the map
$
\det: M_E(n, \mathbbm{d}) \lar \Pic^{\mathbbm{d}}(E)
$
is a bijection. Moreover, for any $\kP \not\cong \kP' \in M_E(n, \mathbbm{d})$ we have:
$
\Hom(\kP, \kP') = 0 = \Ext(\kP, \kP').
$
\item The group $J(E)$ acts transitively on $M_E(n, \mathbbm{d})$. Moreover, given
$\kP \in M_E(n, \mathbbm{d})$ and $\kL \in J(E)$, we have:
$
\kP \cong \kP \otimes \kL \; \Longleftrightarrow \; \kL^{\otimes n} \cong \kO.
$
\end{itemize}
\end{theorem}

\noindent
\emph{Comment on  the proof}. In the case of elliptic curves all these  statements are  due to Atiyah \cite{Atiyah}.
The case of a nodal Weierstra\ss{}
 curve has been treated by the first-named author in \cite{Burban1},
the corresponding result for a  cuspidal cubic curve is due to Bodnarchuk and Drozd \cite{BodnarchukDrozd}. The remaining cases (Kodaira fibers of type  $\mathrm{I}_2$, $\mathrm{I}_3$, III and IV) are due to
Bodnarchuk, Drozd and Greuel \cite{BodDroGre}. Their method actually allows to prove this theorem  for  arbitrary Kodaira cycles of projective lines. In that case, one can also deduce this result from
another description of simple vector bundles obtained in \cite[Theorem 5.3]{OldSurvey}. On the other hand,
this result is still open for $n$ concurrent lines in $\mathbb{P}^{n-1}$ if $n \ge 4$.

\begin{proposition}\label{P:AdonCY}
Let $E$ be a reduced plane cubic curve and $\kP$ be a simple vector bundle on $E$ of  rank $n$
and multi-degree $\mathbbm{d}$. Then the following results are true.
\begin{itemize}
\item The sheaf of Lie algebras $\kA = \kA_{n, \mathbbm{d}}:= \Ad(\kP)$ does not depend
on the choice of $\kP \in M_E(n, \mathbbm{d})$.
\item We have: $H^0(\kA) = 0 =  H^1(\kA)$. Moreover, this result remains true for   an
arbitrary Calabi--Yau curve.
\item For $\kL \in J(E) \setminus \{\kO\}$ we have: $H^0(\kA \otimes \kL) \ne 0$ if and only if
$\kL^{\otimes n} \cong \kO$. Moreover, in this case we have:
$H^0(\kA \otimes \kL) \cong \kk \cong H^1(\kA \otimes \kL)$.
\end{itemize}
\end{proposition}

\begin{proof} The first part follows from the transitivity of the action of $J(E)$ on
$M_E(n, \mathbbm{d})$ (see Theorem \ref{T:simplebundles}) and the fact that
$\Ad(\kP) \cong \Ad(\kP \otimes \kL)$ for any line bundle $\kL$ (see Proposition
\ref{P:basiconAd}). The second statement follows from the long exact sequence
$$
0 \rightarrow H^0(\kA) \lar \End(\kP) \xrightarrow{H^0(\Tr_\kP)} H^0(\kO) \lar
H^1(\kA) \lar \Ext(\kP, \kP) \lar H^1(\kO) \rightarrow 0,
$$
the isomorphisms  $\End(\kP) \cong  \kk \cong \Ext(\kP, \kP)$,
 $H^0(\kO) \cong \kk \cong H^1(\kO)$ and the fact that $H^0(\Tr_\kP)(\mathbbm{1}_\kP) =
\rk(\kP)$.

\vspace{1mm}
\noindent
To show  the last statement, note that we have the exact sequence
$$
0 \lar H^0(\kA \otimes \kL) \lar \Hom(\kP, \kP \otimes \kL) \lar H^0(\kL)
$$
and $H^0(\kL) = 0$. By Theorem \ref{T:simplebundles} we know that $\Hom(\kP, \kP \otimes \kL) = 0$
unless $\kL^{\otimes n} \cong \kO$. In the latter case,
$H^0(\kA \otimes \kL) \cong \End(\kP) \cong \kk$. Since $\kA \otimes \kL$ is a vector bundle of degree zero, by the Riemann-Roch formula we obtain that  $H^1(\kA \otimes \kL) \cong \kk$.
\end{proof}

\section{Triple products on Calabi--Yau curves and the classical Yang--Baxter equation}

In this section we shall explain an interplay between the theory of vector bundles on
Calabi--Yau curves, triple Massey products, $A_\infty$--structures and the classical Yang--Baxter equation. Let $E$ be a Calabi--Yau curve, $x, y \in \breve{E}$ a pair of points from  the same irreducible
component of $E$ and $\kP$ a simple vector bundle on $E$.
By (\ref{E:MasseyMap}) and Lemma \ref{L:SLandSerre}, we have the canonical linear map
\begin{equation}\label{E:mapm_xy}
\overline{\mm}_{x, \, y}:= \mm^\kP_{\kk_x, \kk_y}:  \; \; \mathfrak{sl}\bigl(\Hom(\kP, \kk_x)\bigr) \lar \mathfrak{pgl}\bigl(\Hom(\kP, \kk_y)\bigr).
\end{equation}
By Lemma \ref{L:LA1}, this map corresponds to a certain  (canonical)  tensor
\begin{equation}
\mm_{x,y} \in \mathfrak{pgl}\bigl(\Hom(\kP, \kk_x)\bigr) \otimes \mathfrak{pgl}\bigl(\Hom(\kP, \kk_y)\bigr).
\end{equation}

\subsection{The case of an elliptic curve}
The following result is due to Polishchuk, see
\cite[Theorem 2]{Polishchuk1}.

\begin{theorem}
Let $E$ be an elliptic curve, $\kP$ be a simple vector bundle on $E$ and $x_1, x_2, x_3 \in E$
be pairwise distinct. Then we have the following equality
\begin{equation}\label{E:fromAinftyToCYBE}
\bigl[\mm^{12}_{x_1, x_2}, \mm^{13}_{x_1, x_3}\bigr] + \bigl[\mm^{12}_{x_1, x_2}, \mm^{23}_{x_2, x_3}\bigr]
+ \bigl[\mm^{12}_{x_1, x_2}, \mm^{13}_{x_1, x_3}\bigr] = 0,
\end{equation}
where both sides are
viewed as elements of $\mathfrak{g}_1 \otimes \mathfrak{g}_2 \otimes \mathfrak{g}_3$. Here,
$\mathfrak{g}_i = \mathfrak{pgl}\bigl(\Hom(\kP, \kk_{x_i})\bigr)$ for  $i = 1, 2, 3$.
 Moreover, the tensor $\mm_{x_1, x_2}$ is unitary:
\begin{equation}\label{E:unitary}
\mm_{x_2, x_1} = - \tau\bigl(\mm_{x_1, x_2}\bigr)
\end{equation}
where $\tau : \mathfrak{g}_1 \otimes \mathfrak{g}_2 \lar \mathfrak{g}_2 \otimes \mathfrak{g}_1$ is the map permuting both factors.
\end{theorem}

\medskip
\noindent
\emph{Idea of the proof}. The equality (\ref{E:unitary}) follows from existence of an
$A_\infty$--structure on $\Dbcoh(E)$ which is \emph{cyclic} with respect to the pairing
(\ref{E:SerrePairing}). In particular, this means that for any objects
$\kF_1, \kF_2, \kG_1, \kG_2$ in $\Dbcoh(E)$ and morphisms
$a_1 \in \Hom(\kF_1, \kG_1), a_2 \in \Hom(\kF_2, \kG_2), \omega_1 \in
\Ext(\kG_1, \kF_2)$ and $\omega_2 \in \Ext(\kF_2, \kG_1)$ we have:
\begin{equation}\label{E:cyclic}
\bigl\langle \mm(a_1 \otimes \omega_1 \otimes a_2), \omega_2\bigr\rangle =
- \bigl\langle a_1, \mm(\omega_1 \otimes a_2 \otimes \omega_2)\bigr\rangle =
- \bigl\langle \mm(a_2 \otimes \omega_2 \otimes a_1), \omega_1\bigr\rangle,
\end{equation}
where $\mm = \mm_3^{\infty}$ is the triple product this $A_\infty$--structure.
A proof of the existence of such an $A_\infty$--structure has been outlined by Polishchuk in
\cite[Theorem 1.1]{PolishchukMP}, see also \cite[Theorem 10.2.2]{KoSo} for a different approach using non-commutative
symplectic geometry.
 The identity (\ref{E:cyclic}) applied to
$\kF_1 = \kF_2 = \kP$ and $\kG_i = \kk_{x_i}$ leads to  the equality (\ref{E:unitary}).
The fact that $\mm_{x_1, x_2}$ satisfies the classical Yang--Baxter equation (\ref{E:fromAinftyToCYBE})
follows from (\ref{E:unitary}) and the equality
$$
\mm \circ (\mm \otimes \mathbbm{1} \otimes \mathbbm{1} + \mathbbm{1} \otimes m \otimes \mathbbm{1}
+ \mathbbm{1} \otimes \mathbbm{1} \otimes \mm) + \mathrm{other \; terms} = 0
$$
(which is one of the equalities (\ref{E:structureAinfty}))
viewed as a linear map
$$
\Hom(\kP, \kk_{x_1}) \otimes \Ext(\kk_{x_1}, \kP) \otimes
\Hom(\kP, \kk_{x_2}) \otimes \Ext(\kk_{x_2}, \kP) \otimes \Hom(\kP, \kk_{x_3}) \rightarrow
\Hom(\kP, \kk_{x_3}).
$$
\qed
\begin{remark} Up to now, we are not aware of a  complete   proof of existence of
an $A_\infty$--structure on the triangulated category $\Perf(E)$
for a singular Calabi--Yau curve $E$, which is cyclic with respect to the pairing
(\ref{E:SerrePairing}).
Hence, in order to derive the identities (\ref{E:fromAinftyToCYBE}) and (\ref{E:unitary}) for
a singular Weierstra\ss{}  curve $E$,
we use a different approach which is similar in spirit to the work \cite{BK4}.
Following \cite{Polishchuk1},
 we give another description of the tensor $\mm_{x,y}$ and show some kind of its continuity
with respect to   the degeneration of the complex structure on $E$. This approach
also provides a constructive way of computing of the   tensor $\mm_{x,y}$.
\end{remark}

\subsection{Residues and traces} Let $\Omega$ be the sheaf of regular differential one--forms
on a (possibly reducible) Calabi--Yau curve $E$,  $w \in H^{0}(\Omega)$ some no-where vanishing regular
differential form   and
$x, y \in \breve{E}$ a pair points from the same irreducible component of $E$.
First recall
that we have the following canonical short exact sequence
\begin{equation}\label{E:ResSequence}
0 \lar \Omega \lar \Omega(x) \xrightarrow{\underline{\res}_x} \kk_x \lar 0.
\end{equation}
The section $w$ induces the  short exact sequence
\begin{equation}\label{E:inducedRes}
0 \lar \kO \lar \kO(x) \lar \kk_x \lar 0.
\end{equation}
Hence, for any vector bundle $\kF$ we get a short exact sequence of coherent sheaves
\begin{equation}\label{E:resVectBundle}
0 \lar \kF \stackrel{\imath}\lar   \kF(x) \xrightarrow{\underline{\res}_x^\kF}
 \kF\otimes \kk_x \lar 0.
\end{equation}

\noindent
Next, recall the following result relating categorical traces with the  usual trace
 of an endomorphism of a finite dimensional vector space.

\begin{proposition}\label{P:resandtrace} In the above notation, the following results are true.
\begin{itemize}
\item There is an isomorphism of functors $\delta_x: \Hom(\kk_x, \,-\,\otimes \kk_x) \lar
\Ext(\kk_x,\,-\,)$ from the category  of vector bundles on $E$ to the category of vector spaces
over $\kk$,  given by the boundary map induced by the short exact sequence (\ref{E:resVectBundle}).
\item For any vector bundle $\kF$ on the curve $E$ and  a pair of morphisms
$b: \kF \lar \kk_x,  a: \kk_x \lar \kF \otimes \kk_x$,  we have the equality:
\begin{equation}\label{E:tracevstrace}
t^w\bigl(\Tr_\kF(\delta_x(a) \circ b)\bigr) = \tr(a \circ b_x),
\end{equation}
where $\Tr_\kF: \Ext(\kF, \kF) \lar H^1(\kO)$ is the trace map and
$t^w$ is the composition $H^0(\kO) \stackrel{w}\lar H^0(\Omega) \stackrel{t}\lar \kk$  of the isomorphism induced by $w$ and the canonical map $t$ described in \cite[Subsection 4.3]{BK4}.
\end{itemize}
\end{proposition}

\noindent
\emph{Comment on the proof}. The first part of the statement is just \cite[Lemma 4.18]{BK4}. The content
of the second part is explained by the following commutative diagram:
\begin{equation*}
\xymatrix{
0 \ar[r] & \kF \ar[r] \ar[d] & \kQ \ar[r] \ar[d] & \kF \ar[r] \ar[d]^b & 0 \\
0 \ar[r] & \kF \ar[r] \ar[d] & \kR \ar[r] \ar[d] & \kk_x \ar[r] \ar[d]^a & 0 \\
0 \ar[r] & \kF \ar[r]^-\imath  & \kF(x) \ar[r]^-{\underline{\res}_x^\kF}  & \kF\otimes \kk_x \ar[r]  & 0.
}
\end{equation*}
The lowest horizontal sequence of this diagram is
 (\ref{E:resVectBundle}). The middle sequence corresponds to the element
$\delta_x(a) \in \Ext(\kk_x, \kF)$ and the top one corresponds to
$\delta_x(a) \circ b \in \Ext(\kF, \kF)$. The endomorphism $a \circ b_x \in \End(\kF\big|_x)$ is the
induced map in the fiber of $\kF$ over $x$. The equality (\ref{E:tracevstrace}) follows
from \cite[Lemma 4.20]{BK4}. \qed

\begin{proposition}\label{P:traces}
The following diagram of vector spaces is commutative.
$$
\xymatrix{
\Hom(\kF, \kk_x) \otimes \Ext(\kk_x, \kF) \ar[rr]^-{\mathbbm{1} \;\otimes\; \SS} & & \Hom(\kF, \kk_x)
\otimes \Hom(\kF, \kk_x)^* \\
\Hom(\kF, \kk_x) \otimes \Hom(\kk_x, \kF \otimes \kk_x) \ar[u]^{\mathbbm{1} \,\otimes\, \delta_x^\kF}
\ar[d]_\circ
\ar[rr]^-{\mathbbm{1} \;\otimes\; \tr} & & \Hom(\kF, \kk_x)
\otimes \Hom(\kF \otimes \kk_x, \kk_x)^* \ar[u]_{\mathbbm{1} \,\otimes\, \can} \ar[d]^{\ev}\\
\Lin(\kF\big|_x, \kF\big|_x) \ar[rr]^{Y_1} & & \End\bigl(\Lin(\kF\big|_x, \, \kk)\bigr).
}
$$
Here, $\SS$ is given by  (\ref{E:SerreMap}),
$\delta_x^\kF$ is the isomorphism from Proposition \ref{P:resandtrace}, $\circ$ is
$\mm_2$ composed with the induced  map in the fiber over $x$,
$Y_1$ is the canonical isomorphism of vector spaces from Lemma \ref{L:LA1}, $\ev$ and $\tr$
are   canonical
isomorphisms of vector spaces and $\can$ is the isomorphism induced by $\underline{\res}_x^\kF$.
\end{proposition}

\begin{proof}
The commutativity of the top square is given  by \cite[Lemma 4.21]{BK4}. The commutativity
of the lower square can be easily verified by  diagram chasing.
\end{proof}


\begin{lemma}\label{L:factontraces}
The following diagram of vector spaces is commutative.
\begin{equation*}
\xymatrix{
\Hom(\kF, \kk_x) \otimes \Ext(\kk_x, \kF) \ar[rrr]^-{\overline\SS} \ar[ddd]_T & & & \End\bigl(\Hom(\kF, \kk_x)\bigr)
\ar[ddd] \\
 & K \ar[r]^-{\overline\SS} \ar[d]_-{\overline{T}} \ar@{_{(}->}[lu]& \slie\bigl(\Hom(\kF, \kk_x)\bigr) \ar@{^{(}->}[ru] \ar[d] & \\
 &\slie(\kF\big|_x) \ar[r]^-{Y_1} \ar@{^{(}->}[ld] & \slie\bigl(\Lin(\kF\big|_x, \kk)\bigr) \ar@{^{(}->}[rd]& \\
\End(\kF\big|_x) \ar[rrr]^-{Y} &  & & \End\bigl(\Lin(\kF\big|_x, \kk)\bigr)
}
\end{equation*}
In this diagram, $\overline\SS$ is the isomorphism induced by the Serre duality  (\ref{E:SerreMap}),
$Y$ and $Y_1$ are canonical isomorphisms from Lemma \ref{L:LA2}, $K$ is the subspace
of $\Hom(\kF, \otimes \kk_x) \otimes \Ext(\kk_x, \kF)$ defined in (\ref{E:linearspaceK}), $T$ is the composition
of $\mathbbm{1} \otimes (\delta_x^\kF)^{-1}$ from Proposition \ref{P:traces} and $\circ$, whereas
$\overline{T}$ is the restriction of $T$. The remaining arrows are canonical morphisms of vector spaces.
\end{lemma}

\begin{proof}
Commutativity of the big square is given by Proposition \ref{P:traces}. For the left
small square it follows from the equality (\ref{E:tracevstrace}) whereas the commutativity of the remaining parts
of this  diagram
is obvious.
\end{proof}

\subsection{Geometric description of the triple Massey products}
Let $E$, $\kP$, $x$ and $y$ be as at the beginning of this section. In what
follows, we shall frequently use the notation $\kA:= \Ad(\kP)$ and
$\kE := {\mathcal End}(\kP)$.

\begin{lemma}\label{L:basicsonresandev}
We have a canonical  isomorphism of vector spaces
\begin{equation}\label{E:residue}
\res_x: = H^0(\underline{\res}_x^\kA):  \quad H^0\bigl(\kA(x)\bigr) \lar \kA\big|_x
\end{equation}
induced by the short exact sequence (\ref{E:resVectBundle}). Moreover, we have
the canonical morphism
\begin{equation}\label{E:evaluation}
\ev_y: = H^0(\underline{\ev}_y^\kA): \quad  H^0\bigl(\kA(x)\bigr) \lar \kA\big|_y
\end{equation}
obtained by composing the induced   map in the fibers with  the canonical isomorphism
$\kA(x)\big|_y \lar \kA\big|_y $. When $E$ is a reduced plane cubic curve, $\ev_y$
is an isomorphism if and only if $n\cdot\bigl([x] - [y]\bigr) \ne 0$ in
$J(E)$, where $n = \rk(\kP)$.
\end{lemma}

\begin{proof}
The short exact sequence
$$
0 \lar \kA \stackrel{\imath}\lar \kA(x) \stackrel{\underline{\res}_x^\kA}\lar \kA \otimes \kk_x \lar 0
$$
yields the long exact sequence
$$
0 \lar H^0(\kA) \lar H^0\bigl(\kA(x)\bigr) \xrightarrow{\res_x} \kA\big|_x \lar H^1(\kA).
$$
Thus, the first part of the statement follows from the vanishing $H^0(\kA) = 0 = H^1(\kA)$ given
by Proposition \ref{P:AdonCY}.

\vspace{2mm}
\noindent
In order to show the second part note that we have the canonical short exact sequence
$$
0 \lar \kO(-y) \lar \kO \stackrel{\underline{\ev}_y}\lar \kk_y \lar 0
$$
yielding the short exact sequence
$$
0 \lar \kA(x-y) \lar \kA(x) \lar \kA(x) \otimes \kk_y \lar 0.
$$
Hence, we get the long exact sequence
$$
0 \lar H^0\bigl(\kA(x-y)\bigr) \lar H^0\bigl(\kA(x)\bigr) \xrightarrow{\ev_y} \kA\big|_y \lar
H^1\bigl(\kA(x-y)\bigr).
$$
Since the dimensions of $H^0\bigl(\kA(x)\bigr)$ and $\kA\big|_y$ are the same,
$\ev_y$ is an isomorphism if and only if $H^0\bigl(\kA(x-y)\bigr) = 0$.
By Proposition \ref{P:AdonCY},  this vanishing  is equivalent to the condition $n\cdot\bigl([x] -
[y]\bigr) \ne 0$ in $J(E)$.
\end{proof}

\medskip
\noindent
The  following key result was  stated for the first time
in \cite[Theorem 4]{Polishchuk1}.

\begin{theorem}\label{T:main1}
In the notation as at the beginning of this section, the following diagram of vector spaces
is commutative:
\begin{equation}\label{E:maindiagram1}
\begin{array}{c}
\xymatrix{
\kA\big|_x \ar[rr]^-{Y_1}  & & \slie\bigl(\Hom(\kP, \kk_x)\bigr) \ar[dd]^{\overline{\mm}_{x, y}} \\
H^0\bigl(\kA(x)\bigr) \ar[u]^{\res_x} \ar[d]_{\ev_y} & &  \\
\kA\big|_y \ar[rr]^-{Y_2} & & \mathfrak{pgl}\bigl(\Hom(\kP, \kk_y)\bigr). 
}
\end{array}
\end{equation}
In this diagram, $\overline{\mm}_{x, y}$ is the linear map (\ref{E:mapm_xy}) induced by
the triple $A_\infty$--product in $\Perf(E)$,
$\res_x$ and $\ev_y$ are the linear maps (\ref{E:residue}) and (\ref{E:evaluation}), whereas
$\overline{Y}_1$ and $\overline{Y_2}$ are obtained by composing the canonical isomorphisms
$Y_1$ and $Y_2$ from Lemma \ref{L:LA2} with the canonical isomorphisms
induced by $\Hom(\kP, \, \kk_z) \lar \Lin\bigl(\kP\big|_z, \, \kk\bigr)$ for $z \in \{x, y\}$.
\end{theorem}

\subsection{A proof of the Comparison Theorem}
\noindent
We split  the proof of Theorem \ref{T:main1} into three smaller  logical  steps.

\medskip
\noindent
\underline{Step 1}. First note that we have a well-defined linear map
$$
\imath^{!}: \Hom\bigl(\kP, \kP(x)\bigr) \lar \End\bigl(\Hom(\kP, \kk_y)\bigr)
$$
defined as follows. Let $g \in \Hom\bigl(\kP, \kP(x)\bigr)$ and $h \in \Hom(\kP, \kk_y)$ be arbitrary
morphisms. Then there exists a unique morphism $\tilde{h} \in \Hom(\kP, \kk_y)$ such
that $\imath \circ \tilde{h} = h$, where $\imath: \kP \lar \kP(x)$ is the canonical inclusion.
Then we set:
$
\imath^{!}(g)(h) = \tilde{h} \circ g.
$
It follows from the definition that $\imath^{!}(\imath) = \mathbbm{1}_{\Hom(\kP, \kk_y)}$. This yields
the following result.

\begin{lemma}
We have a well-defined linear map
\begin{equation*}
\bar{\imath}^{!}: \quad \frac{\Hom\bigl(\kP, \kP(x)\bigr)}{\langle\imath\rangle} \lar
\pgl\bigl(\Hom(\kP, \kk_y)\bigr)
\end{equation*}
given by the rule: $\bar{\imath}^{!}(\bar{g}) = \overline{h \mapsto g \circ \tilde{h}}$.
\end{lemma}

\begin{lemma}
The canonical morphism of vector spaces
\begin{equation}\label{E:morphismj}
\jmath: H^0\bigl(\kA(x)\bigr) \lar \frac{\Hom\bigl(\kP, \kP(x)\bigr)}{\langle \imath \rangle}
\end{equation}
given by the composition
$$
H^0\bigl(\Ad(\kP)(x)\bigr) \hookrightarrow
H^0\bigl({\mathcal End}(\kP)(x)\bigr) \lar \Hom\bigl(\kP, \kP(x)\bigr)
\lar \frac{\Hom\bigl(\kP, \kP(x)\bigr)}{\langle \imath \rangle}
$$
is an isomorphism.
\end{lemma}

\begin{proof}
The short exact sequences (\ref{E:shortexactAd}) and  (\ref{E:inducedRes})  together with
the vanishing $H^0(\kA) = 0 = H^1(\kA)$ imply that we have the following commutative diagram
\begin{equation*}
\xymatrix{
0 \ar[r] & H^0(\kO) \ar[r] & H^0\bigl(\kO(x)\bigr) \ar[r]^-0 & \kk \ar[r] & H^1(\kO) \\
0 \ar[r] & H^0(\kE) \ar[r] \ar[u] & H^0\bigl(\kE(x)\bigr) \ar[u] \ar[r] & \kE\big|_x \ar[u]_{\tr} \ar[r] &
H^1(\kE) \ar[u] \\
0 \ar[r] & 0 \ar[u] \ar[r] & H^0\bigl(\kA(x)\bigr) \ar[u] \ar[r]^-{\res_x} & \kA\big|_x \ar[u]
\ar[r] &
0. \ar[u]
}
\end{equation*}
The  fact that $\jmath$ is an isomorphism  follows from   a straightforward diagram chase.
\end{proof}

\begin{lemma}\label{L:commdiag1}
The following diagram  is commutative.
\begin{equation*}
\xymatrix{
\Hom\bigl(\kP, \kP(x)\bigr) \ar[rr]^{\ev_y} \ar[d]_{\imath^!} & & \End(\kP\big|_y) \ar[d]^{Y} \\
\End\bigl(\Hom(\kP, \kk_y)\bigr) \ar[rr]^{\can} & & \End\bigl(\Lin(\kP\big|_y, \kk)\bigr).
}
\end{equation*}
\end{lemma}

\begin{proof}
The result follows from  a straightforward diagram chase.
\end{proof}

\begin{proposition}\label{P:diagwithresandev}
The following diagram is commutative.
\begin{equation}\label{E:diagwithresandev}
\begin{array}{c}
\xymatrix{
H^0\bigl(\kA(x)\bigr) \ar[rr]^{\ev_y} \ar[d]_-{\jmath} & & \slie(\kP\big|_y) \ar[d]^-{Y_2} \\
 \frac{\displaystyle \Hom\bigl(\kP, \kP(x)\bigr)}{\displaystyle \langle \imath \rangle}
 \ar[r]^-{\bar{\imath}^{!}} & \pgl\bigl(\Hom(\kP, \kk_y)\bigr) \ar[r] &
 \pgl\bigl(\Lin(\kP\big|_y, \kk)\bigr).
}
\end{array}
\end{equation}
In particular, if $E$ is a reduced plane cubic curve then $\bar{\imath}^{!}$ is an isomorphism
if and only if $n\cdot\bigl([x]-[y]\bigr) \ne 0$ in $J(E)$.
\end{proposition}

\begin{proof} Note that the following diagram is commutative:
\begin{equation*}
\xymatrix{
\slie\bigl(\kP\big|_y\bigr) \ar@{^{(}->}[r] & \End(\kP\big|_y) \ar[rrd]^-{Y} & & \\
H^0\bigl(\kA(x)\bigr) \ar[u]^-{\ev_y} \ar@{^{(}->}[r] \ar[rd]_-\jmath &
\Hom\bigl(\kP, \kP(x)\bigr) \ar[r]^-{\imath^{!}} \ar[d] \ar[u]_{\ev_y} &
\End\bigl(\Hom(\kP, \kk_y)\bigr) \ar[r] \ar[d] & \End\bigl(\Lin(\kP\big|_y, \kk)\bigr) \ar[d] \\
& \frac{\displaystyle \Hom\bigl(\kP, \kP(x)\bigr)}{\displaystyle \langle \imath \rangle}
 \ar[r]^{\bar{\imath}^{!}} & \pgl\bigl(\Hom(\kP, \kk_y)\bigr) \ar[r] &
\pgl\bigl(\Lin(\kP\big|_y, \kk)\bigr).
}
\end{equation*}
Indeed, the right top square is commutative by Lemma \ref{L:commdiag1}, the commutativity
of the remaining parts is straightforward. Hence, the diagram (\ref{E:diagwithresandev})
is commutative, too.

Next, observe that all maps in the diagram   (\ref{E:diagwithresandev}) but
$\bar{\imath}^{!}$ and $\ev_y$ are isomorphisms. By Lemma \ref{L:basicsonresandev},
the map $\ev_y$ is an isomorphism if and only if $n \cdot \bigl([x]-[y]\bigr) \ne 0$ in $J(E)$.
This proves the second part of this Proposition.
\end{proof}

\noindent
Note that from the exact sequence (\ref{E:resVectBundle}) we get
the induced map
$$
R:= H^0\bigl(\underline{\res}_x^{{\mathcal End}(\kP)}\bigr):  \quad \Hom\bigl(\kP, \kP(x)\bigr) \lar \End\bigl(\kP\big|_x\bigr)
$$
sending an element $g \in \Hom\bigl(\kP, \kP(x)\bigr)$ to $\bigl(\underline{\res}_x^\kP \circ g)_x
\in \End\bigl(\kP\big|_x\bigr)$. Clearly, $R(\imath) = 0$, thus we obtain the induced map
\begin{equation}\label{E:mapRbar}
\overline{R}: \quad \frac{\displaystyle \Hom\bigl(\kP, \kP(x)\bigr)}{\displaystyle \langle \imath \rangle}
\lar \End\bigl(\kP\big|_x\bigr).
\end{equation}
\begin{lemma}
In the above notation, the following statements are true.
\begin{enumerate}
\item $\mathsf{Im}(\overline{R}) = \slie\bigl(\kP\big|_x\bigr)$.
\item Moreover, the  map $\overline{R}: \quad \frac{\displaystyle \Hom\bigl(\kP, \kP(x)\bigr)}{\displaystyle \langle \imath \rangle}
\lar \slie\bigl(\kP\big|_x\bigr)$ is an isomorphism.
\end{enumerate}
\end{lemma}

\begin{proof}
The result follows from the commutativity of the diagram
\begin{equation*}
\xymatrix{
H^0\bigl(\Ad(\kP)(x)\bigr) \ar[r]^-{\res_x} \ar[d]_-{\jmath} & \slie\bigl(\kP\big|_x\bigr)
\ar@{_{(}->}[d] \\
\frac{\displaystyle \Hom\bigl(\kP, \kP(x)\bigr)}{\displaystyle \langle \imath \rangle}
\ar[r]^-{\overline{R}} & \End\bigl(\kP\big|_x\bigr)
}
\end{equation*}
and the fact that the morphisms $\res_x$ and $\jmath$ are isomorphisms.
\end{proof}

\noindent
\underline{Step 2}. The next  result is the key part of the proof of  Theorem \ref{T:main1}.

\begin{proposition}\label{P:keyprop}
The following diagram is commutative.
\begin{equation*}
\begin{array}{c}
\xymatrix{
\slie\bigl(\kP\big|_x\bigr) \ar[r]^-{T} & K \ar[rd]^-{M_H} & \\
\frac{\displaystyle \Hom\bigl(\kP, \kP(x)\bigr)}{\displaystyle \langle \imath \rangle}
\ar[u]^-{\overline{R}} \ar[r]^-{\imath^{!}} & \pgl\bigl(\Hom(\kP, \kk_y)\bigr) \ar[r]^-{\bar{r}_H} &
\Lin\bigl(H, \Hom(\kP, \kk_y)/H\bigr).
}
\end{array}
\end{equation*}
\end{proposition}

\begin{proof}
We show this result by diagram chasing. Recall that the vector space $K$ is the  linear span of
 the simple tensors $f \otimes \omega \in \Hom(\kP, \kk_x) \otimes \Ext(\kk_x, \kP)$ such that
$\omega \circ f = 0$.
Let
$
0 \lar \kP \stackrel{\kappa}\lar \kQ \stackrel{p}\lar \kk_x \lar 0
$
be a short exact sequence corresponding to an element
$\omega \in \Ext(\kk_x, \kP)$. Recall that by Proposition \ref{P:resandtrace}  there exists a unique
$a \in \Hom(\kk_x, \kP \otimes \kk_x)$ such that $\omega = \delta_x(a)$.

Since $\Hom(\kk_x, \kk_y) = 0 = \Ext(\kk_x, \kk_y)$, for any
$h \in \Hom(\kP, \kk_y)$ there exist unique elements $\tilde{h} \in
\Hom(\kQ, \kk_y)$ and $\tilde{h}' \in \Hom\bigl(\kP(x), \kk_y\bigr)$ such that
the following diagram is commutative:
\begin{equation*}
\begin{array}{c}
\xymatrix{
 & & & & & & \kP \ar[dd]^-{f}  \ar[lldd]_-{\tilde{f}} & &\\
 & & & & & & &   & \\
0 \ar[rr] & & \kP \ar[rr]^-{\kappa} \ar[rd]_-{h} \ar[dd]_{\mathbbm{1}_\kP} & &  \kQ \ar[rr]^-{p} \ar[dd]^-{t} \ar[ld]^-{\tilde{h}}
& & \kk_x \ar[rr] \ar[dd]^a & &  0 \\
& & & \kk_y & & & & \\
0 \ar[rr] & & \kP \ar[rr]^{\imath} \ar[ru]^-{h} & &  \kP(x) \ar[lu]_-{\tilde{h}'}
\ar[rr]^-{\underline{\res}_x^\kP}  & &\kP \otimes \kk_x \ar[rr]
& & 0.
}
\end{array}
\end{equation*}
Although a  lift $\tilde{f} \in \Hom(\kP, \kQ)$ is only defined up
to a translation $\tilde{f} \mapsto \tilde{f} + \lambda \kappa$ for some
$\lambda \in \kk$, we have a well-defined element
$\overline{t \circ \tilde{f}} \in \frac{\displaystyle \Hom\bigl(\kP, \kP(x)\bigr)}{\displaystyle \langle \imath \rangle}$ such that $\overline{R}\bigl(\overline{t \circ \tilde{f}}\bigr) = a \circ f_x$.
By definition, $T(a \circ f_x) = f \otimes \omega$.  It remains to observe that
$$
\bigl(\bar{r}_H \circ \bar{\imath}^{!}([t \tilde{f}])\bigr)(h)
= [\tilde{h}' t \tilde{f}] = [\tilde{h} f] =
\bigl(M_H(f \otimes \omega)\bigr)(h).
$$
Since $\overline{R}$ and $T$ are isomorphisms  and the vector space $K$ is generated by simple tensors, this
concludes the proof.
\end{proof}

\medskip
\noindent
\underline{Step 3}. Now we are ready to proceed with the  proof of Theorem \ref{T:main1}. Note
that the following diagram is commutative.

\begin{equation*}
\xymatrix{
 & \slie\bigl(\Hom(\kP, \kk_x)\bigr) \ar[r]^-{\can_1} & \slie\bigl(\Lin(\kP\big|_x, \kk)\bigr) & \\
 K \ar[ru]^-{\overline{\SS}} \ar[ddd]_-{M_H} \ar[rdd]^{\widetilde{\mm}_{x,y}} & & &  \slie(\kP\big|_x) \ar[ul]_-{Y_1} \ar[lll]_-T \\
 & & \frac{\displaystyle \Hom\bigl(\kP, \kP(x)\bigr)}{\displaystyle \langle \imath \rangle}
 \ar[ru]^-{\overline{R}} \ar[ld]_-{\bar{\imath}^{!}} &
 H^0\bigl(\kA(x)\bigr) \ar[l]_-{\jmath} \ar[u]_-{\res_x} \ar[dd]^-{\ev_y} \\
 & \pgl\bigl(\Hom(\kP, \kk_y)\bigr) \ar[ld]^{\bar{r}_H} \ar[rd]_{\can_2} &  & \\
 \Lin\bigl(H, \Hom(\kP, \kk_y)/H\bigr) & & \pgl\bigl(\Lin(\kP\big|_y, \kk)\bigr) &
 \slie\bigl(\kP\big|_y\bigr) \ar[l]_-{Y_2}
 }
\end{equation*}
where $\widetilde{\mm}_{x, y} = \mm^\kP_{\kk_x, \kk_y}$ from (\ref{E:MasseyMap}).
Indeed, by Lemma \ref{L:factontraces} we have the equality $Y_1 \circ T = \can_1 \circ \SS$, which
gives commutativity of the top square. Next, the equality
$\bar{r}_H \circ \widetilde{\mm}_{x,y} = M_H$  just expresses the
commutativity of the  diagram (\ref{E:compatMasseyAinfty}). The equality
$\overline{R} \circ \jmath = \res_x$  follows from the definition of the map
$\overline{R}$, see (\ref{E:mapRbar}).

The equality
$Y_2 \circ \ev_y = \can_2 \circ \bar{\imath}^{!} \circ \jmath$ is given by Proposition
\ref{P:diagwithresandev}, yielding the commutativity of the right lower part.
Finally, by Proposition \ref{P:keyprop} we have the equality
$
\bar{r}_H \circ \bar{\imath}^{!} = M_H \circ T \circ \overline{R}.
$
Since this equality is true for any one-dimensional subspace
$H \subseteq \Hom(\kP, \kk_y)$, Lemma \ref{L:LA3} implies  that
$\widetilde{\mm}_{x,y} \circ T \circ \overline{R} = \bar{\imath}^{!}$. This finishes
the  proof of commutativity of the above diagram.
It remains to conclude  that the commutativity of the diagram (\ref{E:maindiagram1}) follows as well
and Theorem \ref{T:main1} is proven.
\qed

\begin{corollary}\label{C:keycorol}
Let $E$ be an elliptic curve  over $\kk$, $\kP$ a simple vector bundle on $E$, $\kA = \Ad(\kP)$ and
$x, y \in R$ two distinct points. Let
$
r_{x, y} \in \kA\big|_x \otimes \kA\big|_y
$
be  the image of the linear map $\ev_y \circ \res_x^{-1} \in \Lin(\kA\big|_x, \kA\big|_y)$
under the  linear isomorphism $\Lin(\kA\big|_x, \kA\big|_y) \lar \kA\big|_x \otimes \kA\big|_y$
induced by the Killing form $\kA\big|_x \times  \kA\big|_x \lar  \kk,
(a, b) \mapsto \tr(a \circ b)$. Then $r_{x, y}$ is a solution of the
classical Yang--Baxter equation: for any pairwise distinct points $x_1, x_2, x_3 \in E$
we have:
\begin{equation}\label{E:fromGeomToCYBE}
\bigl[r^{12}_{x_1, x_2}, r^{13}_{x_1, x_3}\bigr] + \bigl[r^{12}_{x_1, x_2}, r^{23}_{x_2, x_3}\bigr]
+ \bigl[r^{12}_{x_1, x_2}, r^{13}_{x_1, x_3}\bigr] = 0,
\end{equation}
where both sides of the above identity are
viewed as elements of $\kA\big|_{x_1} \otimes \kA\big|_{x_2} \otimes \kA\big|_{x_3}$.
 Moreover, the tensor $r_{x_1, x_2}$ is unitary:
\begin{equation}\label{E:unitary1}
r_{x_2, x_1} = - \tau\bigl(r_{x_1, x_2}\bigr),
\end{equation}
where $\tau: \kA\big|_{x_1} \otimes \kA\big|_{x_2} \lar \kA\big|_{x_2} \otimes \kA\big|_{x_1}$
is the map permuting both factors.
\end{corollary}

\begin{proof}
By Theorem \ref{T:main1}, the tensor $r_{x, y}$ is the image of the tensor $m_{x, y}$ from
(\ref{E:mapm_xy}) under the isomorphism
$$
\pgl\bigl(\Hom(\kP, \kk_x)\bigr) \otimes \pgl\bigl(\Hom(\kP, \kk_y)\bigr)
\xrightarrow{\overline{Y}_2 \otimes \overline{Y}_2} \kA\big|_{x} \otimes \kA\big|_{y}.
$$
Since $\overline{Y}_2$ is an anti-isomorphism of Lie algebras, the equality (\ref{E:fromGeomToCYBE}) is a corollary
of (\ref{E:fromAinftyToCYBE}). In the same way, the equality (\ref{E:unitary1}) is a consequence of
 (\ref{E:unitary}).
\end{proof}

\noindent
Now we  generalize Corollary \ref{C:keycorol} to the case of the singular Weierstra\ss{} curves.

\section{Genus one fibrations and CYBE}\label{S:GenusOneandCYBE}

\noindent
We start with  the following geometric data.

\begin{itemize}
\item  Let $E \stackrel{p}\lar T$ be a flat \emph{projective}
  morphism of relative dimension one between algebraic varieties. We
  denote by $\breve{E}$ the regular locus  of $p$.
\item   We assume there exists a section   $\imath: T \rightarrow \breve{E}$ of
  $p$.
\item  Moreover, we assume that for all points
  $t \in T$ the fiber $E_t$ is an \emph{irreducible} Calabi--Yau curve.
\item  The fibration $E \stackrel{p}\lar T$ is embeddable  into a smooth
  fibration of projective surfaces over $T$ and $\Omega_{E/T} \cong \kO_E$.
\end{itemize}

\begin{example}
Let $E_{T} \subset  {\mathbb P}^2 \times \mathbb{A}^2 \lar \mathbb{A}^2 =:T$
be the elliptic  fibration given by the equation
$wv^2 = 4 u^3 +  g_2 uw^2 + g_3 w^3$ and let
$\Delta(g_2, g_3)  = g_2^3 + 27 g_3^2$ be the discriminant of this family.
This fibration has a section
$
(g_2, g_3) \mapsto \bigl((0:1:0), (g_2, g_3)\bigr)
$
and satisfies the condition $\Omega_{E/T} \cong \kO_E$.
\end{example}

\noindent
The following result is well-known.

\begin{lemma}\label{L:relativfamily}
Consider  $(n, d) \in \mathbb{N} \times \mathbb{Z}$ such that $\gcd(n, d) = 1$. Then there exists
$\kP \in \VB(E)$ such that for any $t \in T$ its restriction
$\kP\big|_{E_t}$ is simple of rank $n$ and degree $d$.
\end{lemma}

\noindent
\emph{Sketch of the   proof}. Let $\Sigma := \imath(T) \subset E$ and
$\kI_\Delta$ be the structure sheaf of the diagonal $\Delta \subset E \times_T E$. Let
 $\FM^{\kI_\Delta}$  be the Fourier--Mukai transform with the kernel $\kI_{\Delta}$.
By \cite[Theorem 2.12]{BK2}, $\FM^{\kI_\Delta}$ is an auto-equivalence of the derived
category $\FM^{\kI_\Delta}$.
By \cite[Proposition 4.13(iv)]{BK3}  there exists an auto-equivalence $\FF$ of the derived category $\Dbcoh(E)$,
which is a certain composition of the functors $\FM^{\kI_\Delta}$ and
$\,-\,\otimes \kO(\Sigma)$  such that $\FF(\kO_\Sigma) \cong \kP[0]$, where $\kP$ is a vector bundle
on $E$ having the required  properties. \qed

\medskip
Now we fix the following notation.
Let $\kP$ be as in Lemma \ref{L:relativfamily} and $\kA = \Ad(\kP)$.
We set
 $\overline{X} := E \times_T \breve{E} \times_T \breve{E}$ and
  $\overline{B} := \breve{E} \times_T \breve{E}$. Let
 $q: \overline{X} \lar \overline{B}$ be the canonical projection,
$\Delta \subset \breve{E} \times_T \breve{E}$  the diagonal,
$B : = \overline{B} \setminus \Delta$ and $X := q^{-1}(B)$. The elliptic fibration
$q: \overline{X} \lar \overline{B}$ has two canonical sections $h_i$, $i = 1, 2$, given by
$h_i(y_1, y_2) = (y_i, y_1, y_2)$. Let
$\Sigma_i := h_i(\overline{B})$ and $\overline\kA$ be the pull-back of $\kA$
on $\overline{X}$.
Note that the relative dualizing sheaf $\Omega = \Omega_{\overline{X}/\overline{B}}$ is trivial.
 Similarly to (\ref{E:ResSequence}) one has the following canonical short exact sequence
\begin{equation}\label{E:resrelative}
0 \lar \Omega \lar \Omega(\Sigma_1) \stackrel{\underline{\res}_{\Sigma_1}}\lar \kO_{\Sigma_1}
\lar 0,
\end{equation}
see \cite[Subsection 3.1.2]{BK4} for a precise construction.
By the assumptions from the beginning of this section,  there exists an isomorphism $\kO_{\overline{X}}
\lar \Omega_{\overline{X}/\overline{B}}$ induced by a nowhere vanishing section
$w \in H^0(\Omega_{E/T})$. It gives   the  following short exact  sequence
\begin{equation}\label{E:resrelativeVB}
0 \lar \overline{\kA} \lar \overline{\kA}(\Sigma_1) \stackrel{\underline{\res}_{\Sigma_1}^\kA}\lar
\overline{\kA}\big|_{\Sigma_1} \lar 0.
\end{equation}
In a similar way, we have another canonical sequence
\begin{equation}\label{E:evrelativeVB}
0 \lar \overline{\kA}(\Sigma_1 - \Sigma_2) \lar \overline{\kA}(\Sigma_1) \lar
\overline{\kA}(\Sigma_1)\big|_{\Sigma_2} \lar 0.
\end{equation}

\begin{proposition}\label{P:ongeomCYBE}
In the above notation, the following results are true.
\begin{itemize}
\item We have the vanishing $q_*(\overline{\kA}) = 0 = \mathbb{R}^1 q_*(\overline{\kA})$.
\item The coherent sheaf $q_*\bigl(\overline{\kA}(\Sigma_1)\bigr)$ is locally free.
\item Moreover, we have the  morphism of locally free sheaves on $B$ given by the composition
$
q_*\bigl(\overline{\kA}(\Sigma_1)\bigr) \lar q_*\bigl(\overline{\kA}(\Sigma_1)\big|_{\Sigma_2}\bigr)
\lar q_*\bigl(\overline{\kA}\big|_{\Sigma_2}\bigr),
$
which is an isomorphism outside of the closed subset
$
\Delta_n:= \bigl\{(t, x, y)  \; \big| \;  n \cdot \bigl([x] - [y]\bigr) = 0 \in J(E_t) \bigr\} \subset B.
$
\end{itemize}
\end{proposition}

\begin{proof}
Let $z = (t, x, y) \in \overline{B}$ be an arbitrary point. By the base-change formula we have:
$\mathbb{L} \imath^*_z\bigl(\mathbb{R} q_*(\overline{\kA})\bigr) \cong \mathbb{R} \Gamma(\kA\big|_{E_{t}}) = 0$, where
the last vanishing is true by Proposition \ref{P:AdonCY}. This proves the first part of the theorem.

\vspace{1mm}
\noindent
Thus, applying $q_*$ to the short
exact sequence (\ref{E:resrelativeVB}), we get an isomorphism
\begin{equation*}
\res_1 := q_*\bigl(\underline{\res}_{\Sigma_1}^{\overline{\kA}}\bigr): \quad
q_*\bigl(\overline{\kA}(\Sigma_1)\bigr) \lar q_*\bigl(\overline{\kA}\big|_{\Sigma_1}\bigr).
\end{equation*}
For $i = 1, 2$,  let $p_i: \overline{B} := \breve{E} \times \breve{E} \lar E$
be the composition of $i$-th canonical projection with the canonical inclusion $\breve{E} \subseteq E$. It is easy  to see that we have a canonical isomorphism
$\gamma: q_*\bigl(\overline{\kA}\big|_{\Sigma_i}\bigr) \lar  p_i^*(\kA)$.
Hence, the coherent sheaf $q_*\bigl(\overline{\kA}(\Sigma_1)\bigr)$ is locally free on
$\overline{B}$.

\medskip
\noindent
To prove the last part, first note that the canonical morphism
 $q_*\bigl(\overline{\kA}\big|_{\Sigma_2}\bigr) \lar q_*\bigl(\overline{\kA}(\Sigma_1)\big|_{\Sigma_2}\bigr)$ is an isomorphism on
 $B$. Moreover, by Proposition \ref{P:AdonCY}, the subset $\Delta_n$ is precisely the
 support of the complex
 $\mathbb{R} q_* \bigl(\kA(\Sigma_1 - \Sigma_2)\bigr)$. In particular, this shows that
  $\Delta_n$ is a proper closed subset of $B$.  Finally, applying $q_*$ to the short exact sequence (\ref{E:evrelativeVB}), we get
 a morphism of locally free sheaves
 $
 \ev_2: \,  q_*\bigl(\overline{\kA}(\Sigma_1)\bigr) \lar p_2^*(\kA),
 $
 which is an isomorphism on the complement  of  $\Delta_n$. This proves the proposition.
\end{proof}

\begin{theorem}\label{T:main2}
In the above notation, let
 $r \in \Gamma\bigl(\overline{B}, p_1^*(\kA) \otimes p_2^*(\kA)\bigr)$ be the meromorphic
section which is the image of  $\ev_2 \circ \res_1^{-1}$ under the canonical
 isomorphism $$\Hom\bigl(p_1^*(\kA),
p_2^*(\kA)\bigr) \lar H^0\bigl(p_1^*(\kA)^\vee \otimes p_2^*(\kA)\bigr)
\lar H^0\bigl(p_1^*(\kA) \otimes p_2^*(\kA)\bigr).
$$
 The last isomorphism above is induced by the canonical
isomorphism $\kA \lar \kA^\vee$ from Proposition \ref{P:basiconAd}.
Then the following statements are true.
\begin{itemize}
\item The poles of $r$ lie  on the divisor $\Delta$. In particular, $r$ is holomorphic on $B$.
\item Moreover,  $r$ is non-degenerate on the complement  of the set $\Delta_n$.
\item The section $r$ satisfies a version of  the classical Yang--Baxter equation:
\begin{equation*}
\bigl[r^{12}, r^{13}\bigr] + \bigl[r^{12}, r^{23}\bigr] + \bigl[r^{13}, r^{23}\bigr] = 0,
\end{equation*}
where both sides are viewed as meromorphic sections  of $p_1^*(\kA) \otimes p_2^*(\kA) \otimes p_3^*(\kA)$.
\item Moreover,  the section  $r$ is unitary. This means that
\begin{equation}
\sigma^*(r) = - \tilde{r} \in H^0\bigl(p_2^*(\kA) \otimes p_1^*(\kA)\bigr),
\end{equation}
where $\sigma$ is the canonical involution of $\overline{B} =  \breve{E} \times_T \breve{E}$
and $\tilde{r}$ is the section corresponding to the morphism $\ev_1 \circ \res_2^{-1}$.
\item In particular, Corollary \ref{C:keycorol} is also true for singular
Weierstra\ss{} cubic curves.
\end{itemize}
\end{theorem}
\begin{proof} By Proposition \ref{P:ongeomCYBE}, we have the following morphisms in
$\VB(\overline{B})$:
$$
p_1^*(\kA)  \stackrel{\res_1}\longleftarrow q_*\bigl(\overline{\kA}(\Sigma_1)\bigr) \lar
q_*\bigl(\overline{\kA}(\Sigma_1)\big|_{\Sigma_2}\bigr) \stackrel{\imath}\longleftarrow
q_*\bigl(\overline{\kA}\big|_{\Sigma_2}\bigr) \stackrel{\gamma}\lar p_2^*(\kA).
$$
Moreover, $\gamma$ is an isomorphism, whereas $\res_1$ and $\imath$ become isomorphisms
after restricting  on $B$. This shows  that the section $r\in \Gamma\bigl(\overline{B}, p_1^*(\kA) \otimes p_2^*(\kA)\bigr)$ is indeed \emph{meromorphic} with poles lying on the diagonal $\Delta$.
Since $\ev_2 \circ \res_1^{-1}$ is an isomorphism on $B \setminus \Delta_n$, the section
$r$ is non-degenerate on  $B \setminus \Delta_n$.

To prove  the last two parts of the theorem, assume first that the generic fiber of $E$ is smooth. Let $t \in T$ be such
that $E_t$ is an elliptic curve. Then in the notation of Corollary \ref{C:keycorol},
for any $z = (t, x, y) \in B$ have have:
$$
\imath_z^*(r) = r_{x, y} \in \bigl(\kA\big|_{E_t}\bigr)\Big|_{x} \otimes \bigl(\kA\big|_{E_t}\bigr)\Big|_{y},
$$
where we use the canonical isomorphism
$$\imath^*_z \bigl(p_1^*(\overline{\kA}) \otimes  p_2^*(\overline{\kA})\bigr) \lar \bigl(\kA\big|_{E_t}\bigr)\Big|_{x} \otimes
\bigl(\kA\big|_{E_t}\bigr)\Big|_{y}.$$ Let $x_1$, $x_2$ and $x_3$ be three pairwise
 distinct points of $E_t$ and $\bar{x} = (t, x_1, x_2, x_3) \in \breve{E} \times_T \breve{E} \times_T \breve{E}$. By Corollary \ref{C:keycorol} we have:
\begin{equation}\label{E:fiberCYBE}
\imath_{\bar{x}}^*\Bigl(\bigl[r^{12}, r^{13}\bigr] + \bigl[r^{12}, r^{23}\bigr] + \bigl[r^{13},
r^{23}\bigr]\Bigr) = 0.
\end{equation}
In a similar way, we have the equality:
\begin{equation}\label{E:fiberUnitary}
\imath_z^*\bigl(\sigma^*(r) + \tilde{r}\bigr) = 0.
\end{equation}
Since the section $r$ is continuous on $B$,  the equalities (\ref{E:fiberCYBE}) and
(\ref{E:fiberUnitary}) are true for the singular fibers of $E$ as well. In particular, the
statement of Corollary \ref{C:keycorol} is also true for singular Weierstra\ss{} cubic curves.
This  implies that Theorem \ref{T:main2} is true for an arbitrary genus one fibrations satisfying the
conditions from the beginning of this section.
\end{proof}

\noindent
\textbf{Summary}.  Let $E \stackrel{p}\lar T$, $\imath: T \lar E$ and
$w \in H^0\bigl(\Omega_{E/T}\bigr)$ be as at the beginning of the section,  $\kP$ be a relatively stable vector bundle
on $E$ of rank $n$ and degree $d$ (recall that we automatically have
$\gcd(n, d) = 1$) and $\kA = \Ad(\kP)$.
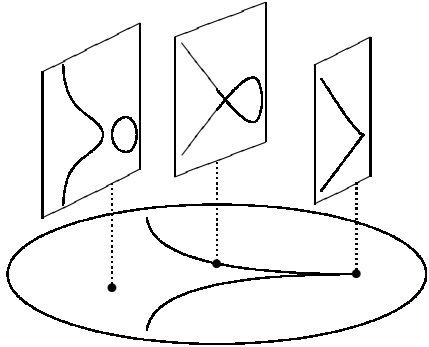
$$
{\xy 0;/r0.22pc/:
\POS(40,0);
{(25,0)\ellipse(30,10){-}},
%
{(12,20)\ellipse(1.5,2.5)_,=:a(180){-}},
{(12,20)\ellipse(2,2.5)^,=:a(180){-}},
\POS(3,30);
@(,
\POS(3,25)@+,  \POS(8,22)@+, \POS(9,20)@+,
\POS(8,18)@+, \POS(3,15)@+, \POS(3,10)@+,
**\crvs{-}
,@i @);
{\ar@{-}(0,8);(0,29)},
{\ar@{-}(0,8);(14,15)},
{\ar@{-}(0,29);(14,36)},
{\ar@{-}(14,15);(14,36)},
\POS(20,33);@-
{
(20,30)@+,  (30,20)@+, (32,25)@+
,(30,30)@+, (20,20)@+, \POS(20,17)@-,
,**\qspline{}};  
{\ar@{-}(19,14);(19,34)},
{\ar@{-}(19,14);(32,19)},
{\ar@{-}(19,34);(32,39)},
{\ar@{-}(32,19);(32,39)},
\POS(40,28); \POS(46,20);
**\crv{(40,28);
(30,25); (42,22); (45,20) };
\POS(40,12); \POS(46,20);
**\crv{(40,18);
(45,5);   (42,8); (45,20); (46,20)};
{\ar@{-}(39,10);(39,30)},
{\ar@{-}(39,10);(47,14)},
{\ar@{-}(39,30);(47,34)},
{\ar@{-}(47,14);(47,34)},
\POS(15,8); \POS(45,0)*{\bul} ;
**\crv{(16,0)};
\POS(15,-8); \POS(45,0);
**\crv{(16,0)};
\POS(15,8);
{\ar@{.}(45,13); (45,0) }
\POS(25,1.5)*{\bul};
{\ar@{.}(25,2); (25,16) }
\POS(10,-2)*{\bul};
{\ar@{.}(10,-2); (10,13) }
\endxy}
$$
 For any closed point of the base $t \in T$ let $U$ be
a small neighborhood of the point $\imath(t) \in E_{t_0}$,  $V$ be a small neighborhood
of $\bigl(t, \imath(t), \imath(t)\bigr) \in E \times_T E$, $O = \Gamma(U, \kO)$  and
$M = \Gamma(V, \, \kM)$, where $\kM$ is the sheaf of meromorphic functions
on $E \times_T E$.
Taking  an isomorphism of Lie algebras
$\xi: \kA({U}) \lar \slie_n(O)$, we get the tensor-valued meromorphic function
\begin{equation*}
r^\xi  = r^{\xi}_{(E, (n, d))} \in \slie_n(M) \otimes_M \slie_n(M),
\end{equation*}
which is the image of the \emph{canonical} meromorphic section
$r \in \Gamma\bigl(E \times_T E, p_1^*(\kA)
\otimes p_2^*(\kA)\bigr)$ from Theorem \ref{T:main2}. Then the following statements are true.
\begin{itemize}
\item The poles of $r^\xi$ lie on the diagonal $\Delta \subset E \times_T E$.
\item
Moreover, for a fixed $t \in T$
this function
 is a \emph{unitary} solution
of the classical Yang--Baxter equation (\ref{eq:CYBE})
in variables $(y_1, y_2) \in \{t\} \times
(U \cap E_t)  \times (U \cap E_t)  \subset
V \subset E \times_T E$. In other words, we get a family of solutions $r^{\xi}_t(y_1, y_2)$
of the
classical Yang--Baxter equation,  which is \emph{analytic} as the function
of the parameter $t \in T$.
\item
Let $\xi': \kA({U}) \lar \slie_n(O)$ be another isomorphism of Lie algebras
and $\rho := \xi' \circ \xi^{-1}$.  Then
we have the following commutative diagram:
$$
\xymatrix{
 & \kA(U) \ar[ld]_{\xi} \ar[rd]^{\xi'}  & \\
\slie_n(O)  \ar[rr]^{\rho} & & \slie_n(O).
}
$$
Moreover,  for any $(t, y_1, y_2) \in V \setminus \Delta$ we have:
\begin{equation*}
r^{\xi'}(y_1, y_2) = \bigl(\rho(y_1) \otimes \rho(y_2)\bigr) \cdot r^{\xi}(y_1, y_2) \cdot
\bigl(\rho^{-1}(y_1) \otimes \rho^{-1}(y_2)\bigr).
\end{equation*}
In other words, the solutions $r^{\xi}$ and $r^{\xi'}$ are \emph{gauge equivalent}.
\end{itemize}

\begin{remark}
One possibility to generalize Theorem \ref{T:main2} and  for an arbitrary
Calabi--Yau curve $E$  can be achieved by showing that any simple vector bundle on
$E$ can be obtained from the structure sheaf $\kO$ by applying an appropriate auto-equivalence
of the triangulated category $\Perf(E)$ (some progress in this direction has been recently achieved
by  Hern\'andez Ruip\'erez, L\'opez Mart\'in,
  S\'anchez G\'omez and Tejero Prieto in \cite{Salamanca}). Once it is done,
  going along the same lines as  in Lemma \ref{L:relativfamily},  one can construct a sheaf of Lie algebras
  $\kA$ on a genus one fibration $E \stackrel{p}\lar T$ such that for the smooth fibers
  $\kA\big|_{E_t} \cong \kA_{n,d}$ and for the singular ones $\kA\big|_{E_t} \cong \kA_{n, \mathbbm{d}}$
  for $n$, $d$  and $\mathbbm{d}$ as in Proposition
  \ref{P:AdonCY}.
\end{remark}

At this moment one can pose the following natural question: How constructive is the suggested method
of finding of solutions of the classical Yang--Baxter equation (\ref{eq:CYBE})?
Actually, one can work out a  completely explicit recipe to compute the tensor $r^\xi_{(E, (n, d))}$
for all types of Weierstra\ss{} cubic curves, see for example \cite{BK4},  where an analogous approach
to  the associative Yang--Baxter equation has been elaborated. The following result can be found
in \cite[Chapter 6]{BK4} and also in \cite{Polishchuk1}.
\noindent
\begin{example}\label{Ex:rmatrices}{\rm
Fix the following basis
$$
h =
\left(
\begin{array}{cc}
1 & 0 \\
0 & -1
\end{array}
\right),  \quad
e =
\left(
\begin{array}{cc}
0 & 1 \\
0 & 0
\end{array}
\right),  \quad
f =
\left(
\begin{array}{cc}
0 & 0 \\
1 & 0
\end{array}
\right)
$$
of the Lie algebra  $\mathfrak{g} = \mathfrak{sl}_2(\mathbb{C})$. For the pair
$(n, d) = (2, 1)$ we get the following solutions $r_{(E, (2, 1))}$ of the classical Yang--Baxter equation
(\ref{E:CYBE1}).

\begin{itemize}
\item In the case $E$ is elliptic, we get the  {elliptic} solution of Baxter:
\begin{equation}\label{E:Baxter}
r_{\mathrm{ell}}(z) =   \frac{\mathrm{cn}(z)}{\mathrm{sn}(z)} h \otimes h +
\frac{1+ \mathrm{dn}(z)}{\mathrm{sn}(z)}(e \otimes f +
f \otimes e)  +
\frac{1 -  \mathrm{dn}(z)}
{\mathrm{sn}(z)}(e \otimes e + f \otimes f),
\end{equation}
\item In the case $E$ is nodal,  we get the trigonometric solution  of Cherednik
\begin{equation}\label{E:Cherednik}
r_{\mathrm{trg}}(z) = \frac{1}{2} \cot(z)h\otimes h +
\frac{1}{\sin(z)}(e\otimes f + f \otimes e) + \sin(z) e \otimes e.
\end{equation}
\item  In the case $E$ is cuspidal, we get the rational solution
\begin{equation}\label{E:Stolin}
r_{\mathrm{rat}}(z) = \dfrac{1}{z}\left(
\dfrac{1}{2} h \otimes h + e_{12}
\otimes e_{21} + e_{21} \otimes e_{12}\right)
 + z(f \otimes h + h \otimes f) - z^3 f \otimes f.
\end{equation}
\end{itemize}
}
\end{example}
\begin{remark}
It is a non-trivial analytic consequence of Theorem \ref{T:main2}  that
up a certain (unknown) gauge transformation and a change of variables,
the rational solution (\ref{E:Stolin}) is a degeneration of the elliptic solution  (\ref{E:Baxter}) and
the trigonometric solution (\ref{E:Cherednik}).
\end{remark}
In the second part of this article, we
describe solutions of (\ref{eq:CYBE}) corresponding  to  the smooth respectively cuspidal Weierstra\ss{} curves.
All of them
turn out to be elliptic respectively rational. We shall recover all elliptic solutions
respectively certain distinguished rational solutions.  Note that  rational solutions of (\ref{eq:CYBE}) are most complicated
and less understood from the point of view
of the Belavin--Drinfeld classification \cite{BelavinDrinfeld}.

\section{Vector bundles on elliptic curves and elliptic solutions of the classical Yang--Baxter equation}
Let $\tau \in \CC$ be such that $\mathrm{Im}(\tau) > 0$ and
 $E = \CC/\langle 1, \tau\rangle$  the corresponding complex torus. Let  $0 < d < n$  be two
 coprime integers and $\kA = \kA_{n, d}$ be the sheaf of Lie algebras from Proposition
 \ref{P:AdonCY}.
\begin{proposition} The sheaf $\kA$ has the following complex--analytic description:
\begin{equation}
\kA \cong  \CC \times \lieg/\sim, \quad \mbox{with} \quad (z, G) \sim (z+1, X G X^{-1}) \sim
(z + \tau, Y G Y^{-1}),
\end{equation}
where $X$ and $Y$ are the matrices (\ref{E:matricesXY}).
\end{proposition}

\begin{proof}
We first recall some well-known technique to work with holomorphic vector bundles on complex tori, see for example
\cite{BirkenhageLange, MumfordTheta}.

\medskip
\noindent
$\bullet$ Let $\CC \supset \Lambda = \Lambda_\tau :=  \langle 1, \, \tau\rangle \cong \ZZ^2$. An automorphy factor is a pair
$(A, V)$, where $V$ is a finite dimensional vector space over $\CC$ and
$
A: \Lambda \times \CC \lar \GL(V)
$
is a holomorphic function such that $A(\lambda + \mu, z) = A(\lambda, z + \mu) A(\mu, z)$ for all
$\lambda, \mu \in \Lambda$ and $z \in \CC$. Such a pair defines the following  holomorphic vector bundle on
the torus $E$:
\begin{equation*}
\kE(A, V) := \CC \times V/\sim, \;
\mbox{where} \quad (z, v) \sim \bigl(z + \lambda, A(\lambda, z) v\bigr) \quad
\forall (\lambda, z, v) \in \Lambda \times \CC \times V.
\end{equation*}
Two such vector bundles $\kE(A, V)$ and $\kE(B, V)$ are isomorphic if and only if
there exists a holomorphic function $H: \CC \rightarrow \GL(V)$ such that
$$
B(\lambda, z) = H(z + \lambda) A(\lambda, z) H(z)^{-1} \quad
\mbox{for all} \quad (\lambda, z) \in \Lambda \times \CC.
$$
Assume that $\kE = \kE(\CC^n, A)$. Then $\Ad(\kE) \cong \kE\bigl(\lieg, \mathsf{ad}(A)\bigr)$,
where $\mathsf{ad}(A)(\lambda, z)(G) := A(\lambda, z) \cdot G \cdot A(\lambda, z)^{-1}$ for $G \in \lieg$.

\medskip
\noindent
$\bullet$ Quite frequently, it is  convenient to restrict ourselves  on  the following setting. Let $\Phi: \CC \longrightarrow
\GL_n(\CC)$ be a holomorphic function such that $\Phi(z+1) = \Phi(z)$ for all $z \in \CC$.
In other words, we assume that  $\Phi$ factors through the covering map
$\CC \xrightarrow{\exp(2 \pi i (-))} \CC^*$. Then one can define the automorphy factor
$(A, \CC^n)$ in the following way.

\medskip
\noindent
$-$
$A(0, z) = I_n$ is the identity matrix.

\medskip
\noindent
$-$ For any $a \in \ZZ_{>0}$ we set:
$$
A(a\tau, z) = \Phi\bigl(z + (a-1)\tau\bigr) \cdot \dots \cdot \Phi(z) \;
\mbox{and} \; A(-a\tau, z) = A(a \tau, z-a \tau)^{-1}.
$$

\noindent
$-$ For any $a,b \in \ZZ$ we set: $A(a \tau + b, z) = A(a \tau, z)$.

\medskip
\noindent
Let  $\kE(\Phi) := \kE(A, \CC^n)$ be the corresponding vector bundle on $E$.

\medskip
\noindent
$\bullet$ Consider the holomorphic function $\psi(z) =
\exp\bigl(- \pi i d \tau - \frac{2 \pi i d}{n} z\bigr)$ and the matrix
\begin{equation*}
\Psi = \left(
\begin{array}{cccc}
0 & 1 & \dots & 0 \\
\vdots & \vdots & \ddots & \vdots \\
0 & 0 & \dots & 1 \\
\psi^n & 0 & \dots & 0
\end{array}
\right).
\end{equation*}
It follows from  Oda's description of simple vector bundles on elliptic curves \cite{Oda},
that the vector bundle $\kE(\Psi)$ is simple of rank $n$ and degree $d$.
See  also \cite[Proposition 4.1.6]{BK4} for a proof of this result.

\medskip
\noindent
$\bullet$ Denote $\varepsilon = \exp(\frac{2 \pi i d}{n})$, $\eta = \varepsilon^{-1}$
 and $\rho = \exp(-\frac{2 \pi i d}{n} \tau)$.
Consider the function $H = \mathsf{diag}\bigl(\psi^{n-1}, \dots, \psi, 1\bigr): \CC \lar \GL_n(\CC)$ and
the matrices $X' = \mathsf{diag}(\eta^{n-1}, \dots, \eta, 1)$,
$Z' = \mathsf{diag}(\rho^{n-1}, \dots, \rho, 1)$,  and
\begin{equation*}
Y' = \left(
\begin{array}{cccc}
0 & \rho^{n-1} & \dots & 0 \\
\vdots & \vdots & \ddots & \vdots \\
0 & 0 & \dots & \rho \\
1 & 0 & \dots & 0
\end{array}
\right).
\end{equation*}
Let $B(\lambda, z) = H(z + \lambda) A(\lambda, z) H(z)^{-1}$, where $A(\lambda, z)$ is the automorphy factor defined by the function $\Phi$. Then we have:
$
B(1, z) = X' \, \mbox{and} \,
B(\tau, z) = \psi \cdot Y'.
$

\medskip
\noindent
$\bullet$ Note that $\mathsf{ad}(B) = \mathsf{ad}(\varphi \cdot  B) \in \End(\lieg)$
for an arbitrary  holomorphic function $\varphi$. Hence,
after the conjugation of  $X'$ and $Y'$ with an appropriate constant diagonal matrix and a subsequent rescaling,  we
get:  $\kA \cong \kE\bigl(\mathsf{ad}(C), \lieg\bigr)$, where
$
C(1, z) = X \, \mbox{and} \, C(\tau, z) = Y.
$
This concludes the proof.
\end{proof}

\medskip
\noindent
Let
$I := \bigl\{(p, q) \in \ZZ^2 \big| 0 \le p \le n-1, 0 \le q \le n-1, (p, q) \ne (0, 0)\bigr\}$.
For any $(k, l) \in I$ denote
$
Z_{k, l} = Y^k X^{-l} \quad \mbox{and} \quad
Z_{k, l}^\vee  = \frac{1}{n} X^{l} Y^{-k}.
$
Recall the following standard result.
\begin{lemma}\label{L:HeisenbergBasis} The following is true.
\begin{itemize}
\item The operators $\mathsf{ad}(X), \mathsf{ad}(Y) \in \End(\lieg)$ commute.
\item The set $\bigl\{Z_{k, l}\bigr\}_{(k, l) \in I}$ is a basis of $\lieg$.
\item Moreover, for any $(k, l) \in I$ we have:
$$
\mathsf{ad}(X)(Z_{k, l}) = \varepsilon^{k} Z_{k, l} \quad
\mbox{and} \quad
\mathsf{ad}(Y)(Z_{k, l}) = \varepsilon^{l} Z_{k, l}.
$$
\item Let $\can: \lieg \otimes \lieg \lar \End(\lieg)$ be the canonical isomorphism
sending a simple tensor $G' \otimes G''$ to the linear map $G \mapsto \tr(G' \cdot G) \cdot G''$.
Then we have:
$$
\can(Z_{k,l}^\vee \otimes Z_{k, l})(Z_{k', l'}) = \left\{
\begin{array}{cl}
Z_{k, l} & \mbox{if} \; (k', l') = (k, l) \\
0 & \mbox{otherwise}.
\end{array}
\right.
$$
\end{itemize}
\end{lemma}

\noindent
Next, recall the definition of the first and  third Jacobian theta-functions \cite{MumfordTheta}.
\begin{equation}\label{E:Theta}
\left\{\begin{array}{l}
\bar{\theta}(z) = \theta_1(z|\tau) = 2 q^{\frac{1}{4}} \sum\limits_{n=0}^\infty (-1)^n q^{n(n+1)}
\sin\bigl((2n+1)\pi z\bigr), \\
\theta(z) = \theta_3(z|\tau) = 1+  2 \sum\limits_{n=1}^\infty  q^{n^2}
\cos(2\pi nz),
\end{array}
\right.
\end{equation}
where $q = \exp(\pi i \tau)$. They are related by the following identity:
\begin{equation}\label{E:RelBetwTheta}
\theta\Bigl(z + \frac{1 + \tau}{2}\Bigr) = i \exp\Bigl(-\pi i \bigl(z + \frac{\tau}{4}\bigr)\Bigr) \bar{\theta}(z).
\end{equation}
For any $x \in \CC$ consider the function $\varphi_x(z) = - \exp\bigl(- 2\pi i (z + \tau - x)\bigr)$.
The next  result is well-known, see
\cite{MumfordTheta} or \cite[Section 4.1]{BK4}.

\begin{lemma}\label{L:AceofBase} The following results are true.
\begin{itemize}
\item The vector space
$$
\left\{\CC \stackrel{f}\lar \CC
\left|
\begin{array}{l}
f \, \mbox{ is holomorphic} \\
f(z+1) = f(z) \\
f(z+\tau) = \varphi_x(z) f(z)
\end{array}
\right.
  \right\}
  $$
  is one-dimensional and generated by the theta-function $\theta_x(z) := \theta\bigl(z + \frac{1 + \tau}{2} - x\bigr)$.
  \item
We have:
$
\kE(\varphi_x) \cong \kO_E\bigl([x]\bigr).
$
\end{itemize}
\end{lemma}

\noindent
 Let $U \subset \CC$ be a small open neighborhood of $0$ and
$O = \Gamma(U, \kO_\CC)$ be the ring of holomorphic functions on $U$. Let $z$ be a coordinate on $U$,
$\CC \stackrel{\pi} \lar E$ the canonical covering map, $w = dz \in H^0(E, \Omega)$
 a no-where vanishing differential form on $E$,
$\Gamma(U, \kA) \stackrel{\xi}\lar \slie_n(O)$ the canonical isomorphism induced by the automorphy
data $(X, Y)$ and
$x, y \in U$ a pair of  distinct points. Consider the following vector space
$$
\Sol\bigl((n, d), x\bigr) =
\left\{\CC \stackrel{F}\lar \lieg
\left|
\begin{array}{l}
F \mbox{\textrm{\quad is holomorphic}} \\
F(z+1) = X F(z) X^{-1}\\
F(z+\tau) = \varphi_x(z) Y  F(z) Y^{-1}
\end{array}
\right.
  \right\}.
  $$
\begin{proposition}\label{P:SolutionsElliptiques} The following diagram
\begin{equation*}
\begin{array}{c}
\xymatrix{
\kA\big|_x \ar[d]_{\jmath_x} & & H^0\bigl(\kA(x)\bigr) \ar[ll]_-{\res^\kA_x(w)} \ar[rr]^-{\ev^\kA_y} \ar[d]^{\jmath} & & \kA\big|_y \ar[d]^{\jmath_y} \\
\lieg & & \Sol\bigl((n, d), x\bigr) \ar[ll]_-{\overline{\res}_x} \ar[rr]^-{\overline{\ev}_y} &  & \lieg
}
\end{array}
\end{equation*}
is commutative, where for $F \in \Sol\bigl((n, d), x\bigr)$ we have:
\begin{equation*}
\overline{\res}_x(F) = \frac{F(x)}{\theta'\bigl(\frac{1 + \tau}{2}\bigr)}
\quad \mbox{and} \quad
\overline{\ev}_y(F) = \frac{F(y)}{\theta\bigl(y - x + \frac{1 + \tau}{2}\bigr)}.
\end{equation*}
The linear isomorphism $\jmath$ is induced by the pull-back map $\pi^*$.
\end{proposition}

\noindent
\emph{Comment on the proof}. This result can be proven along the same lines as in \cite[Section 4.2]{BK4},
see in particular \cite[Corollary 4.2.1]{BK4}, hence we omit the details here.  \qed

\medskip
\noindent
Now we are ready to prove the main result of this section.

\begin{theorem}\label{T:main3}
The solution  $r_{(E, (n, d))}(x, y)$ of the classical Yang--Baxter equation (\ref{eq:CYBE}) constructed in Section \ref{S:GenusOneandCYBE},  is given by the following expression:
\begin{equation}\label{E:SolutionsElliptiques}
r_{(E, (n, d))}(x, y) = \sum\limits_{(k, l) \in I}
\exp\Bigl(-\frac{2 \pi i d}{n} k v\Bigr)
\sigma\Bigl(\frac{d}{n}\bigl(l - k \tau\bigr),  v\Bigr)
Z_{k, l}^\vee \otimes Z_{k, l},
\end{equation}
where $v = x-y$ and
 $\sigma(u, z)$ is the Kronecker elliptic function (\ref{E:KroneckerEllFunct}).
\end{theorem}
\begin{proof}
We first have to compute an explicit basis of the vector space $\Sol\bigl((n, d), x\bigr)$.
For this, we write:
$$
F(z) = \sum\limits_{(k, l) \in I} f_{k, l}(z)  Z_{k, l}.
$$
The condition $F \in \Sol\bigl((n, d), x\bigr)$ yields the following constraints on
 the coefficients $f_{k, l}$:
\begin{equation}\label{E:EllSystem}
\left\{
\begin{array}{lcl}
f_{k, l}(z+1) & = &  \varepsilon^k f_{k, l}(z) \\
f_{k, l}(z+\tau) & = &  \varepsilon^l \varphi_x(z) f_{k, l}(z).
\end{array}
\right.
\end{equation}
It follows from Lemma \ref{L:AceofBase} that the space of solutions of the system (\ref{E:EllSystem})
is one--dimensional and  generated by the holomorphic function
$$
f_{k, l}(z) = \exp\Bigl(- \frac{2 \pi i d}{n} kz\Bigr) \theta\Bigl(
z + \frac{1+ \tau}{2} - x - \frac{d}{n}(k \tau - l)\Bigr).
$$
From Proposition \ref{P:SolutionsElliptiques} and Lemma \ref{L:HeisenbergBasis} it follows that
the solution $r_{(E, (n, d))}(x, y)$
is given by the following formula:
\begin{equation*}
r_{(E, (n, d))}(x, y) = \sum\limits_{(k, l) \in I}
r_{k, l}(v)
Z_{k, l}^\vee \otimes Z_{k, l},
\end{equation*}
where $v = y - x$ and
$$
r_{k, l}(v) = \exp\Bigl(-\frac{2 \pi i d}{n} k v\Bigr)
\frac{\displaystyle \theta'\Bigl(\frac{1 + \tau}{2}\Bigr) \theta\Bigl(v + \frac{1 + \tau}{2}- \frac{d}{n}(k\tau - l)\Bigr)}{\displaystyle \theta\Bigl(- \frac{d}{n}(k\tau - l)\Bigr) \theta(v)}.
$$
Relation (\ref{E:RelBetwTheta}) implies that
$$
r_{k, l}(v) = \exp\Bigl(-\frac{2 \pi i d}{n} k v\Bigr)
\frac{\displaystyle \bar{\theta}'(0) \bar{\theta}\Bigl(v - \frac{d}{n}(k\tau - l)\Bigr)}{
\displaystyle \bar\theta\Bigl(-\frac{d}{n}(k\tau - l)\Bigr) \bar\theta(v)}
$$
Let $\sigma(u, z)$ be the Kronecker elliptic function (\ref{E:KroneckerEllFunct}).
It remains to observe that  formula (\ref{E:SolutionsElliptiques}) follows now from the  identity
\begin{equation*}
\sigma(u,  x)
 = \frac{\displaystyle \bar\theta'(0) \bar\theta_1(u+x)}{\displaystyle \bar\theta(u) \bar\theta(x)}.
\end{equation*}
\end{proof}

\section{Vector bundles on the  cuspidal Weierstra\ss{} curve and the classical Yang--Baxter equation}

The goal of this section is to derive  an explicit  algorithm to compute the solution
 $r_{(E, (n, d))}$ of (\ref{eq:CYBE}),
  corresponding to a pair of coprime integers $0 < d < n$
 and the cuspidal Weierstra\ss{} curve $E$, which has been  constructed in Section \ref{S:GenusOneandCYBE}.

\subsection{Some results on  vector bundles on singular curves}
\label{SS:catoftriples}

We first recall some general technique to describe vector bundles on singular projective curves,
see \cite{Survey, Thesis, DrozdGreuel} and especially
\cite[Section 5.1]{BK4}.

\medskip
\noindent
Let $X$ be a reduced singular (projective)  curve, $\pi:  \widetilde{X} \lar   X$
its normalisation, $\kI :=
{\mathcal Hom}_\kO\bigl(\pi_*(\kO_{\widetilde{X}}), \kO\bigr) =
{\mathcal A}nn_\kO\bigl(\pi_*(\kO_{\widetilde{X}})/\kO\bigr)$
the conductor ideal sheaf.
Denote  by $\eta: Z = V(\kI) \lar X$ the
closed Artinian subscheme   defined by $\kI$
(its topological support is precisely the singular locus of $X$)  and by
$\tilde\eta: \widetilde{Z} \lar \widetilde{X}$ its preimage in
$\widetilde{X}$, defined by the Cartesian  diagram
\begin{equation}\label{E:keydiag}
\begin{array}{c}
\xymatrix
{\widetilde{Z} \ar[r]^{\tilde{\eta}} \ar[d]_{\tilde{\pi}}
& \widetilde{X} \ar[d]^\pi \\
Z \ar[r]^\eta & X.
}
\end{array}
\end{equation}
In what follows we shall denote $\nu = \eta \tilde\pi = \pi \tilde\eta$.

\medskip
\noindent
In order to relate vector bundles on $X$ and $\widetilde{X}$ we need the
following construction.

\begin{definition}\label{D:Triples}
  The category $\Tri(X)$ is defined as follows.
  \begin{itemize}
  \item Its objects are triples
    $\bigl(\widetilde\kF, \kV, \theta\bigr)$, where
    $\widetilde\kF \in \VB(\widetilde{X})$, $\kV \in \VB(Z)$ and
    $$\theta:  \; \tilde{\pi}^*\kV \lar \tilde{\eta}^*\widetilde\kF$$ is an
    isomorphism of $\kO_{\widetilde{Z}}$--modules.
  \item The set of morphisms
    $\Hom_{\Tri(X)}\bigl((\widetilde\kF_1, \kV_1, \theta_1),
    (\widetilde\kF_2, \kV_2, \theta_2)\bigr)$ consists of all pairs
    $(f, g)$, where $f: \widetilde\kF_1 \lar  \widetilde\kF_2$ and
    $g: \kV_1 \lar \kV_2$
    are morphisms of vector bundles such that the following
    commutative
    $$
    \xymatrix
    {\tilde{\pi}^*\kV_1 \ar[rr]^-{\theta_1}\ar[d]_{\tilde{\pi}^*(g)} & &
      \tilde{\eta}^*\widetilde\kF_1
      \ar[d]^{\tilde{\eta}^*(f)} \\
      \tilde{\pi}^*\kV_2 \ar[rr]^-{\theta_2}
      & & \tilde{\eta}^*\widetilde\kF_2
    }
    $$
    is commutative.
  \end{itemize}
\end{definition}

\noindent
The importance of Definition \ref{D:Triples} is  explained  by the following theorem.

\begin{theorem}\label{T:keyonVB}
  Let $X$ be a reduced curve. Then the following results are true.
\begin{itemize}
\item
  Let  $\mathbb{F}: \VB(X) \lar \Tri(X)$ be the functor assigning  to a
    vector bundle $\kF$ the triple  $(\pi^*\kF, \eta^*\kF,
    \theta_{\kF})$, where $\theta_{\kF}:
    \tilde{\pi}^*(\eta^*\kF) \lar \tilde\eta^*(\pi^*\kF)$ is the canonical
    isomorphism. Then $\mathbb{F}$ is  an equivalence of categories.
\item
 Let $\mathbb{G}: \Tri(X) \lar \Coh(X)$ be the functor
    assigning to  a triple
    $(\widetilde\kF, \kV, \theta)$ the coherent sheaf
    $
    \kF := \ker\bigl(\pi_*\widetilde\kF \oplus  \eta_*\kV \xrightarrow{
      (\cc, \, \, -\bar\theta)
    }
    \nu_* \tilde\eta^*\widetilde\kF\bigr),
    $
    where $\cc = \cc^{\widetilde\kF}$ is the canonical morphism
    $
    \pi_*\widetilde\kF \lar \pi_* \tilde\eta_* \tilde\eta^*\tilde\kF
    = \nu_* \tilde\eta^*\widetilde\kF
    $
    and $\bar\theta$ is the composition
    $
    \eta_*\kV \stackrel{\can}
    \lar \eta_* \tilde\pi_* \tilde\pi^*\kV
    \stackrel{=}\lar \nu_* \tilde\pi^*\kV
    \xrightarrow{\nu_*(\theta)} \nu_* \tilde\eta^*\widetilde\kF.
    $
    Then the coherent sheaf $\kF$ is locally free. Moreover, the functor
    $\mathbb{G}$ is quasi-inverse to $\mathbb{F}$.
\end{itemize}
\end{theorem}

\noindent
A proof of this Theorem can be found in  \cite[Theorem 1.3]{Thesis}. \qed

\medskip
\noindent
Let $\kT = \bigl(\widetilde\kF, \kV, \theta\bigr)$ be an object of $\Tri(X)$.
Consider the morphism
$$
\overline{\conj}(\theta): \quad {\mathcal End}_{\widetilde{Z}}(\tilde\pi^*\kV) \lar
{\mathcal End}_{\widetilde{Z}}(\tilde\eta^*\widetilde\kF),
$$
sending a local section $\varphi$ to $\theta \, \varphi \, \theta^{-1}$. Then we have the
following result.

\begin{proposition}\label{P:ComputingAd(P)}
Let
$\kF := \GG(\kT)$. Then we have:
\begin{equation*}
{\mathcal End}_X(\kF) \cong
\GG\bigl({\mathcal End}_{\widetilde{X}}(\widetilde{\kF}), {\mathcal End}_Z(\kV),
\conj(\theta)\bigr),
\end{equation*}
where $\conj(\theta)$ is the morphism making the following diagram
$$
\xymatrix{
\tilde{\pi}^*{\mathcal End}_Z(\kV) \ar[rr]^{\conj(\theta)} \ar[d]_{\can} & & \ar[d]^{\can}
\tilde\eta^*{\mathcal End}_{\widetilde{X}}(\widetilde{\kF}) \\
{\mathcal End}_{\widetilde{Z}}(\tilde\pi^*\kV) \ar[rr]^{\overline{\conj}(\theta)} & &
{\mathcal End}_{\widetilde{Z}}(\tilde\eta^*\widetilde\kF)
}
$$
commutative.
Similarly, we have:
$
\Ad(\kF) \cong
\GG\bigl(\Ad(\widetilde{\kF}), \Ad(\kV),
\conj(\theta)\bigr).
$
\end{proposition}

\noindent
A proof of Proposition \ref{P:ComputingAd(P)} can be deduced from Theorem \ref{T:keyonVB} using the standard technique of sheaf theory
  and is  therefore omitted.

\subsection{Simple vector bundles on the cuspidal Weierstra\ss{}  curve}\label{SS:SimplVBcusp}
 Now we recall the description of the simple vector bundles on the cuspidal Weierstra\ss{} curve following the approach of Bodnarchuk and Drozd \cite{BodnarchukDrozd}, see also
\cite[Section 5.1.3]{BK4}.

\medskip
\noindent
1. Throughout this section, $E = V(wv^2 - u^3) \subseteq \mathbb{P}^2$ is the cuspidal Weierstra\ss{}  curve.

\medskip
\noindent
2. Let $\pi: \PP^1 \lar E$ be the normalization of $E$. We choose homogeneous coordinates
$(z_0: z_1)$ on $\PP^1$ in such a way that $\pi\bigl((0:1)\bigr)$ is the singular point of $E$.
In what follows, we denote $\infty = (0: 1)$ and $0 = (1: 0)$.
Abusing the notation, for any $x \in \kk$ we also denote by $x \in \breve{E}$ the image of the point
$\tilde{x} = (1:x) \in \PP^1$,  identifying in this  way $\breve{E}$ with
$\mathbb{A}^1 = \PP^1 \setminus \{\infty\} =:  U_\infty$. Let  $t = \frac{\displaystyle z_0}{
\displaystyle z_1}$,  then we have:
$\kk[U_\infty] = \kk[t]$. Let $R = \kk[\varepsilon]/\varepsilon^2$  and
$\kk[t] \rightarrow  R$  be the canonical projection.
Then in the notation of the previous subsection we have: $Z \cong  \Spec(\kk)$ and
$\widetilde{Z} \cong  \Spec(R)$.

\medskip
\noindent
3. By the theorem of Birkhoff--Grothendieck, for any $\kF \in \VB(E)$ we have:
\begin{equation*}
\pi^*\kF \cong  \bigoplus\limits_{c \in \ZZ} \kO_{\PP^1}(c)^{\oplus n_c}.
\end{equation*}
A choice of homogeneous coordinates on $\PP^1$ yields two distinguished sections
$z_0, z_1 \in H^0\bigl(\kO_{\PP^1}(1)\bigr)$. Hence, for any $e \in \mathbbm{N}$ we get
a distinguished basis of the vector space $\Hom_{\PP^1}\bigl(\kO_{\PP^1}, \kO_{\PP^1}(e)\bigr)$ given by
the monomials $z_0^e, z_0^{e-1} z_1, \dots, z_1^e$.
Next,  for any $c \in \ZZ$ we  fix the following isomorphism
\begin{equation*}
\zeta^{\kO_{\PP^1}(c)}: \quad \kO_{\PP^1}(c)\big|_{\widetilde{Z}} \lar \kO_{\widetilde{Z}}
\end{equation*}
sending a local section $p$ to $\frac{\displaystyle p}{\displaystyle z_1^c\big|_{\widetilde{Z}}}$. Thus, for any
vector bundle $\widetilde\kF = \bigoplus\limits_{c \in \ZZ} \kO_{\PP^1}(c)^{\oplus n_c}$
of rank $n$ on $\PP^1$ we have the induced isomorphism
$\zeta^{\widetilde\kF}:  \; \widetilde\kF\big|_{\widetilde{Z}} \lar \kO_{\widetilde{Z}}^{\oplus n}$.

\medskip
\noindent
4. Consider  the set
$\Sigma  := \bigl\{(a, b) \in \mathbb{N} \times \mathbb{N} \, \big| \, \gcd(a, b) = 1 \bigr\}$
 and for any  $(a, b) \in \Sigma \setminus \bigl\{(1, 1)\bigr\}$ denote:
\[ \epsilon(a,b)=
\left\{ \begin{array}{c c}
(a-b,b), & a>b\\
(a,b-a), & a<b.
\end{array}
\right.
\]
Now, starting with a pair $(e, d) \in \Sigma$,  we construct
a finite sequence of elements of $\Sigma$  ending with $(1,1)$, defined as follows.
We put $(a_{0},b_{0})=(e, d)$ and, as long as $(a_{i},b_{i})\neq(1,1)$,
we set $(a_{i+1},b_{i+1})=\epsilon(a_{i},b_{i})$. Let
\begin{equation}\label{E:inductiveDef}
J_{(1,1)}=\left(
{\begin{array}{c | c}
0 & 1 \\ \hline
0 & 0
\end{array}}
\right) \in \Mat_{2 \times 2}(\mathbb{C}).
\end{equation}
Assume that the matrix  \[  J_{(a,b)} =\left(
{\begin{array}{c | c}
A_{1} & A_{2} \\ \hline
0 & A_{3}
\end{array}}
\right)
\] with $A_{1}\in\mbox{Mat}_{a\times a}(\kk)$ and $A_{3}\in\mbox{Mat}_{b\times b}(\kk)$
has already been defined. Then for $(p, q) \in \Sigma$ such that
$\epsilon(p,q) = (a, b)$, we set
\begin{equation}\label{E:formcanonique}
J_{(p,q)} =
\left\{
\begin{array}{c c}
\left(
{\begin{array}{c | c c}
0 & \mathbbm{1} & 0 \\ \hline
0 & A_{1} & A_{2} \\
0 & 0 & A_{3}
\end{array}}
\right), & p=a\\
\\
\left(
{\begin{array}{c c | c}
A_{1} & A_{2} & 0 \\
0 & A_{3} & \mathbbm{1} \\ \hline
0 & 0 & 0
\end{array}}
\right), & q=b.
\end{array}
\right.
\end{equation}

\noindent Hence, to any tuple $(e,d) \in \Sigma$ we can assign a certain uniquely determined
  matrix
$J = J_{(e,d)}$  of size $(e +d) \times (e + d)$,
obtained by the above recursive procedure from  the sequence
$\bigl\{(e,d),...,(1,1)\bigr\}$.

\begin{example}
Let $(e,d) = (3,2)$. Then the corresponding  sequence  of elements of $\Sigma$ is $\bigl\{(3,2),(1,2),(1,1)\bigr\}$
and the matrix $J = J_{(3,2)}$ is constructed as follows
\[ \left( {\begin{array}{c | c}
0 & 1 \\ \hline
0 & 0
\end{array}}
\right)
\rightarrow
\left( {\begin{array}{c | c c}
0 & 1 & 0\\ \hline
0 & 0 & 1\\
0 & 0 & 0
\end{array}}
\right)
\rightarrow
\left( {\begin{array}{c c c | c c}
0 & 1 & 0 & 0 & 0 \\
0 & 0 & 1 & 1 & 0 \\
0 & 0 & 0 & 0 & 1 \\ \hline
0 & 0 & 0 & 0 & 0 \\
0 & 0 & 0 & 0 & 0 \\
\end{array}}
\right).
\]
\end{example}

\medskip
\noindent
5. Given $0 <  d < n$ mutually prime and $\lambda \in \kk$, we take the matrix
\begin{equation}\label{E:gluingmatrixFor(n,d)}
\Theta_\lambda = \Theta_{n, d, \lambda} = \mathbbm{1} + \varepsilon\bigl(\lambda \mathbbm{1} +
J_{(e, d)}\bigr) \in \GL_n(R), \; e = n-d.
\end{equation}
The matrix $\Theta_\lambda$ defines a morphism $\bar\theta_\lambda: \eta_* \kO_Z \lar \nu_* \kO_{\widetilde{Z}}$.
Let $\widetilde{\kP} = \widetilde{\kP}_{n, d} := \kO_{\PP^1}^{\oplus e} \oplus
\kO_{\PP^1}(1)^{\oplus d}$. Consider the following vector bundle $\kP_\lambda = \kP_{n, d, \lambda}$ on $E$:
\begin{equation}\label{E:univfamily}
0 \lar \kP_\lambda \stackrel{\left(\begin{smallmatrix} \imath \\ q \end{smallmatrix}\right)}\lar
\pi_* \widetilde\kP \oplus \eta_* \kO_{Z}^{\oplus n}
\xrightarrow{(\zeta^{\widetilde\kP},  \;  -\bar\theta_\lambda)} \nu_* \kO_{\widetilde{Z}}^{\oplus n}
\lar 0.
\end{equation}
Then $\kP_\lambda$ is simple with rank $n$ and degree $d$.
 Moreover, in an appropriate sense, $\bigl\{\kP_\lambda\bigr\}_{\lambda \in \kk^*}$ is  a
universal family of simple  vector bundles of rank $n$ and degree $d$ on the  curve $E$, see \cite[Theorem 5.1.40]{BK4}. The next  result follows from Proposition \ref{P:ComputingAd(P)}.

\begin{corollary}
Let $0 <  d < n$ be a  pair of coprime integers, $e = n-d$  and $J =
J_{(e, d)} \in \Mat_{n \times n}(\kk)$ be the matrix given by the recursion  (\ref{E:formcanonique}).
Consider the vector bundle $\kA$ given by the following short exact sequence
\begin{equation}\label{E:definitionAd(P)}
0 \lar \kA \stackrel{\left(\begin{smallmatrix} \jmath \\ r \end{smallmatrix}\right)}\lar \pi_* \widetilde\kA \oplus \eta_* \bigl(\Ad(\kO_Z^{\oplus n})\bigr)
\xrightarrow{\bigl(\zeta^{\Ad(\widetilde\kP)}, \;  -\conj(\Theta_0)\bigr)}
\eta_*\bigl(\Ad(\kO_{\widetilde{Z}}^{\oplus n})\bigr) \lar 0,
\end{equation}
where
$\widetilde\kA = \Ad(\widetilde\kP)$. Then $\kA \cong \Ad(\kP_0)$. Moreover,  for any trivialization
$\xi: \widetilde\kP\big|_{U_{\infty}} \lar \kO_{U_\infty}^{\oplus n}$ we get the following
isomorphisms  of sheaves
of Lie algebras
\begin{equation}\label{E:trivialAd(P)}
\kA\big|_{\breve{E}} \stackrel{\jmath}\lar \pi_*\bigl(\Ad(\widetilde\kP)\bigr)\big|_{\breve{E}}
\lar \pi_* \Ad\bigl(\kO_{U_\infty}^{\oplus n}\bigr) \stackrel{\can}\lar \Ad\bigl(\kO_{\breve{E}}^{\oplus n}\bigr),
\end{equation}
where the second morphism is induced by  $\xi$.
\end{corollary}

\noindent
6. In the above notation, for any $x \in \breve{E} \cong \mathbb{A}^1$ the corresponding  line bundle
$\kO_E\bigl([x]\bigr)$ is given by the triple $\bigl(\kO_{\PP^1}(1), \kk, 1 - x \cdot \varepsilon\bigr)$, see \cite[Lemma 5.1.27]{BK4}.

\subsection{From simple vector bundles on the cuspidal Weierstra\ss{}  curve to solutions of the
classical Yang--Baxter equation}

In this subsection we derive  the recipe to compute  the solution
of the classical Yang--Baxter equation corresponding to the triple $(E, (n, d))$, where
$E$ is the cuspidal Weierstra\ss{}  curve and $0 < d < n$ is a pair of coprime integers.
Keeping  the same notation as  in Subsection \ref{SS:SimplVBcusp}, we additionally introduce the
following one.

\medskip
\noindent
1.  We choose the following  regular differential one-form $w := dz$ on $E$,
where $z = \frac{\displaystyle z_1}{\displaystyle z_0}$ is a coordinate on the open chart $U_0$.

\medskip
\noindent
2. Let $\lieg[z] := \lieg \otimes \kk[z]$. Then for any $x \in \kk$ we have the $\kk$--linear
evaluation map
$\phi_x: \lieg[z] \rightarrow \lieg$, where $\lieg[z] \ni a z^p \mapsto x^p \cdot a \in \lieg$  for
 $a \in \lieg$.  For $x \ne y \in \kk$ consider the following $\kk$--linear maps:
\begin{equation}
\overline{\res}_x := \phi_x \quad \mbox{and} \quad
\overline{\ev}_y := \frac{1}{y-x}\phi_y.
\end{equation}
\medskip
\noindent
3. Let $(e, d)$ be a pair of coprime positive integers, $n = e + d$ and
$\liee := {\Mat}_{n\times n}(\kk)$.
For the block partition  of $\liee$
induced by the decomposition $n = e + d$, consider the following subspace
of $\lieg[z]$:
\begin{equation}
V_{e,d}= \left\{
F =
\left(
\begin{array}{ c | c}
W & X \\ \hline
Y & Z
\end{array}
\right)
+
\left(
\begin{array}{ c | c}
W' & 0 \\ \hline
Y' & Z'
\end{array}
\right)z
+
\left(
\begin{array}{ c | c}
$0$ & $0$ \\ \hline
Y'' & $0$
\end{array}
\right)z^2
\right\}.
\end{equation}
For a given  $F \in V_{e,d}$ denote
\begin{equation}\label{E:evalonZ}
F_{0}=
\left(
\begin{array}{c | c} W' & X \\ \hline Y'' & Z'
\end{array}
\right) \mbox{ and  } \,
F_{\epsilon}=\left({\begin{array}{c | c}
W & $0$ \\ \hline Y' & Z
\end{array}}\right).
\end{equation}

\medskip
\noindent
4. For $x \in \kk$  consider the following subspace of $V_{e,d}$:
\begin{equation}\label{E:DefSol(e,d,x)}
\Sol\bigl((e, d), x\bigr) := \Bigl\{
F \in V_{e,d} \,
\Bigl| \,
\left[F_{0},J \right]+ x F_{0}+F_{\epsilon}=0
\Bigr\}.
\end{equation}
The following theorem is the main result of this section.
\begin{theorem}\label{T:recipe}
Let $\kA$ be the sheaf of Lie algebras given by (\ref{E:definitionAd(P)}) and $x, y \in \breve{E}$ a pair of distinct points.  Then there exists
an isomorphism of Lie algebras
$\jmath^\kA: \Gamma(\breve{E}, \kA) \rightarrow \lieg[z]$ and a $\kk$--linear isomorphism
$\jmath: H^0\bigl(\kA(x)) \rightarrow \Sol\bigl((e, d), x\bigr)$ such that the following diagram
\begin{equation*}
\begin{array}{c}
\xymatrix{
\kA\big|_x \ar[d]_{\jmath^\kA_x} & & H^0\bigl(\kA(x)\bigr) \ar[ll]_-{\res^\kA_x(w)} \ar[rr]^-{\ev^\kA_y} \ar[d]^{\jmath} & & \kA\big|_y \ar[d]^{\jmath^\kA_y} \\
\lieg & & \Sol\bigl((e, d), x\bigr) \ar[ll]_-{\overline{\res}_x} \ar[rr]^-{\overline{\ev}_y} &  & \lieg
}
\end{array}
\end{equation*}
is commutative.
\end{theorem}

\begin{proof}
We first introduce  the following (final) portion of notations.

\medskip
\noindent
1. For $x \in \kk$ consider the section $\sigma = z_1 - x z_0 \in H^0\bigl(\kO_{\PP^1}(1)\bigr)$.
Using the identification $\kk \stackrel{\cong}\longrightarrow U_0$, $\kk \ni x \mapsto \tilde{x}: = (1: x) \in \PP^1$, the section $\sigma$ induces
an isomorphism of line bundles $t_\sigma:
\kO_{\PP^1}\bigl([\tilde{x}]\bigr) \longrightarrow \kO_{\PP^1}(1)$.

\medskip
\noindent
2. For any $c \in \ZZ$ fix the trivialization $\xi^{\kO_{\PP^1}(c)}: \kO_{\PP^1}(c)\big|_{U_0} \lar
\kO_{U_0}$ given on the level of local sections by the rule $p \mapsto \frac{\displaystyle p}{\displaystyle z_0^c\big|_{U_0}}$. Thus, for any
vector bundle $\widetilde\kF = \oplus \kO_{\PP^1}(c)^{\oplus n_c}$ of rank $n$ we get the induced  trivialization
$\xi^{\widetilde\kF}:  \widetilde{\kF}\big|_{U_0} \lar \kO_{U_0}^{\oplus n}$.

\medskip
\noindent
3. Let $\widetilde\kE = \left(
\begin{array}{cc}
\widetilde{\kE}_1 & \widetilde{\kE}_2 \\
\widetilde{\kE}_3 & \widetilde{\kE}_4
\end{array}
\right)$ be the sheaf of algebras on $\PP^1$ with
 $\widetilde{\kE}_1 = {\mathcal Mat}_{e \times e}(\kO_{\PP^1})$,
$\widetilde{\kE}_4 = {\mathcal Mat}_{d \times d}(\kO_{\PP^1})$,
$\widetilde{\kE}_2 = {\mathcal Mat}_{e \times d}(\kO_{\PP^1}(-1))$ and
$\widetilde{\kE}_3 = {\mathcal Mat}_{d \times e}(\kO_{\PP^1}(1))$.
The ring structure on $\widetilde\kE$ is induced by the canonical isomorphism
$\kO_{\PP^1}(-1) \otimes \kO_{\PP^1}(1) \stackrel{\can}\lar \kO_{\PP^1}$. Let
$\widetilde\kA = \mathsf{\ker}\bigl(\widetilde\kE \stackrel{\tr}\lar \kO_{\PP^1}\bigr)$,
where $\tr$ only involves the diagonal entries of $\widetilde\kE$ and is given
by the matrix $(1, 1, \dots, 1)$. Of coarse,
$\widetilde\kE \cong {\mathcal End}(\widetilde\kP)$ and
$\widetilde\kA \cong \Ad(\widetilde\kP)$ for
$\widetilde\kP = \kO_{\PP^1}^{\oplus e} \oplus \kO_{\PP^1}(1)^{\oplus d}$.

\medskip
\noindent
4. Consider the sheaf of algebras  $\kE$ on $E$ given by the short exact sequence
\begin{equation*}
0 \lar \kE \stackrel{\left(\begin{smallmatrix} \jmath \\ r \end{smallmatrix}\right)}\lar
\pi_* \widetilde\kE \oplus \eta_* \bigl({\mathcal M}_n(Z)\bigr)
\xrightarrow{\bigl(\zeta^{\widetilde\kE}, \;  -\conj(\Theta_0)\bigr)}
\eta_*\bigl({\mathcal M}_n(\widetilde{Z})\bigr) \lar 0,
\end{equation*}
where ${\mathcal M}_n(T) := {\mathcal End}_T(\kO_T^{\oplus n})$
for a scheme $T$. Of coarse $\kE \cong {\mathcal End}_E(\kP_0)$, where
$\kP_0$ is the simple vector bundle of rank $n$ and degree $d$ on $E$ given by
(\ref{E:univfamily}).

\medskip
\noindent
5. In the above notation we have:
\begin{equation}\label{E:descriptSections}
H^0\bigl(\widetilde\kE(1)\bigr) =
\left\{F =
\left(
\begin{array}{c|c}
z_0 W + z_1 W' & X \\
\hline
z_0^2 Y + z_0 z_1 Y' + z_1^2 Y'' & z_0 Z + z_1 Z'
\end{array}
\right)
\right\},
\end{equation}
where $W, W' \in \Mat_{e \times e}(\kk)$, $Z, Z' \in \Mat_{d \times d}(\kk)$, $Y, Y', Y'' \in
\Mat_{d \times e}(\kk)$ and $ X \in \Mat_{e \times d}(\kk)$.

\medskip
\noindent
6.
For any $F \in H^0\bigl(\widetilde{\kE}(1)\bigr)$ as in (\ref{E:descriptSections}) we denote:
\begin{equation}\label{E:resandevexplicit}
\overline{\res}_x(F) =  F(1, x) \quad \mbox{and} \quad
\overline{\ev}_y(F) = \frac{1}{y-x} F(1, y).
\end{equation}

\noindent
7. Finally, let $\jmath^\kE: \kE\big|_{\breve{E}} \lar
\kM_n(\breve{E})$ be the trivialization induced by the trivialization
$\xi^{\widetilde\kE}: \widetilde\kE\big|_{U_0} \lar \kM_n(U_0)$.  This trivialization
 actually induces
an  isomorphism of Lie algebras  $\jmath^\kA: \Gamma(\breve{E}, \kA) \lar \lieg[z]$ we are looking for.

\medskip
\noindent
Now observe  that the following diagram is commutative:
\begin{equation}\label{E:bigDiagram}
\begin{array}{c}
\xymatrix{
& \kA\big|_{x} \ar[d]_{\hat{\pi}^*_x} \ar@/_2pc/[lddd]_-{\jmath^\kA_x} &
H^0\bigl(\kA(x)\bigr) \ar[l]_-{\res^\kA_x(w)}
\ar[r]^-{\ev^\kA_y} \ar[d]_{\hat{\pi}^*}  & \kA\big|_{y} \ar[d]^{\hat{\pi}^*_y}
\ar@/^2pc/[rddd]^-{\jmath^\kA_y} & \\
& \widetilde{\kA}\big|_{\tilde{x}}
\ar@{_{(}->}[d] & H^0\bigl(\widetilde{\kA}(\tilde{x})\bigr) \ar[r]^-{\ev^{\widetilde{\kA}}_{\tilde{y}}} \ar[l]_-{\res^{\widetilde\kA}_{\tilde{x}}(w)} \ar@{^{(}->}[d] &
\widetilde{\kA}\big|_{\tilde{y}} \ar@{^{(}->}[d]  & \\
& \widetilde{\kE}\big|_{\tilde{x}} \ar[d]_-{\xi^{\widetilde{\kE}}_{\tilde{x}}} & H^0\bigl(\widetilde{\kE}(\tilde{x})\bigr)
\ar[l]_-{\res^{\widetilde\kE}_{\tilde{x}}(w)} \ar[r]^-{\ev^{\widetilde{\kE}}_{\tilde{y}}} \ar[d]^-{(t_\sigma)_*} &
\widetilde{\kE}\big|_{\tilde{y}} \ar[d]^-{\xi^{\widetilde{\kE}}_{\tilde{y}}} & \\
\lieg \ar@{^{(}->}[r] & \liee & H^0\bigl(\widetilde{\kE}(1)\bigr)
\ar[l]_{\overline{\res}_x} \ar[r]^{\overline{\ev}_y} & \liee &
\lieg. \ar@{_{(}->}[l]
}
\end{array}
\end{equation}
Following the notation of (\ref{E:definitionAd(P)}), the composition
$$\gamma_\kA: \quad  \pi^* \kA \stackrel{\pi^*(\imath)}\lar \pi^* \pi_* \widetilde{\kA} \stackrel{\can}\lar \widetilde{\kA}$$
is an isomorphism of vector bundles on $\PP^1$. The morphisms $\hat{\pi}^*_x$ and
and $\hat{\pi}^*_y$ are the  maps, obtained by composing
$\pi^*$ and $\gamma_\kA$ and then taking the induced map in the corresponding fibers.
Similarly, $\hat{\pi}^*$ is the induced map of global sections.
The commutativity of both  top squares of (\ref{E:bigDiagram}) follows from the ``locality'' of the morphisms
${\res}^\kA_x(w)$ and $\underline{\ev}^\kA_y$, see
\cite[Proposition 2.2.8 and Proposition 2.2.12]{BK4} as well as \cite[Section 5.2]{BK4} for a detailed proof.

 \medskip
 \noindent
 The commutativity of both
 middle squares of (\ref{E:bigDiagram}) is obvious.
 The commutativity of both  lower squares follows from  \cite[Corollary 5.2.1]{BK4} and
 \cite[Corollary 5.2.2]{BK4} respectively. In particular, the explicit formulae (\ref{E:resandevexplicit}) for the
 maps $\overline{\res}_x$ and $\overline{\ev}_y$ are given  there.
 Finally, see \cite[Subsection 5.2.2]{BK4} for the
 proof of commutativity of both  side diagrams.

\medskip
 \noindent
 Now we have to describe the image of the linear map
 $H^0\bigl(\kA(x)\bigr) \lar H^0\bigl(\widetilde\kE(1)\bigr)$ obtained by composing of the three
 middle vertical arrows in (\ref{E:bigDiagram}). It is convenient to describe first
 the image of the corresponding linear map $H^0\bigl(\kE(x)\bigr) \lar H^0\bigl(\widetilde\kE(1)\bigr)$.
 Recall that
 \begin{itemize}
 \item The sheaf $\kE$ is given by the triple
 $\bigl(\widetilde\kE, \Mat_n(\kk), \conj(\Theta_0)\bigr)$.
 \item The line bundle $\kO_E\bigl([x]\bigr)$ is  given by  the triple
 $\bigl(\kO_{\PP^1}(1), \kk, \mathbbm{1} - x \cdot \varepsilon\bigr)$.
 \item The tensor product in $\VB(E)$ corresponds to the tensor product in $\Tri(E)$.
 \end{itemize}
 These facts lead to  the following consequence. Let
 $F \in H^0\bigl(\widetilde\kE(1)\bigr)$ be written as in (\ref{E:descriptSections}).
 Then $F$ belongs to the image of the linear map $H^0\bigl(\kE(x)\bigr) \lar H^0\bigl(\widetilde\kE(1)\bigr)$ if any only if there exists some
 $A \in \liee$ such that the following equality in $\liee[\varepsilon]$ is true:
  \begin{equation}\label{E:SolTriples}
F\big|_{\widetilde{Z}} = (1 - x\cdot \varepsilon) \cdot \Theta_0 \cdot A \cdot \Theta_0^{-1},
 \end{equation}
 where $ F\big|_{\widetilde{Z}}
:= F_0 + \varepsilon F_\epsilon$ and  $F_0, F_\epsilon$ are given by (\ref{E:evalonZ}).
Since  $\Theta_0^{-1} = \mathbbm{1} - \varepsilon J_{(e, d)}$, the equation
(\ref{E:SolTriples}) is equivalent to  the following constraint:
$$
\left[F_{0},J_{(e,d)} \right]+ x F_{0} + F_{\epsilon}=0.
$$
See
also \cite[Subsection 5.2.5]{BK4} for a computation in a similar situation.

\medskip
\noindent
Finally, consider the following commutative diagram:
\begin{equation*}
\xymatrix{
0 \ar[r] & H^0\bigl(\kA(x)\bigr) \ar[r] \ar[d]_{\hat\pi} & H^0\bigl(\kE(x)\bigr) \ar[r] \ar[d]^{\hat\pi} &
H^0\bigl(\kO_E(x)\bigr) \ar[r] \ar[d]^{\hat\pi} & 0 \\
0 \ar[r] & H^0\bigl(\widetilde{\kA}(\tilde{x})\bigr) \ar[r] \ar[d]_{(t_\sigma)_*} & H^0\bigl(\widetilde{\kE}(\tilde{x})\bigr) \ar[r] \ar[d]^{(t_\sigma)_*} &
H^0\bigl(\kO_{\PP^1}(\tilde{x})\bigr) \ar[r] \ar[d]^{(t_\sigma)_*} & 0 \\
0 \ar[r] & H^0\bigl(\widetilde{\kA}(1)\bigr) \ar[r] & H^0\bigl(\widetilde{\kE}(1)\bigr) \ar[r]^{T} &
H^0\bigl(\kO_{\PP^1}(1)\bigr) \ar[r] & 0,
}
\end{equation*}
where $$T\Bigl(\Bigl(\begin{array}{cc} A_0 z_0 + A_1 z_1 & * \\ * & B_0 z_0 + B_1 z_1
\end{array}\Bigr)\Bigr) = \bigl(\tr(A_0) + \tr(B_0)\bigr) z_0 + \bigl(\tr(A_1) + \tr(B_1)\bigr) z_1.$$
Let $\Sol\bigl((e, d), x\bigr): = \mathsf{Im}\Bigl(H^0\bigl(\kA(x)\bigr)  \lar  H^0\bigl(\widetilde\kE(1)\bigr)\Bigr)$. Note that we have
$$
\Sol\bigl((e, d), x\bigr) =
\mathsf{Ker}(T) \cap \mathsf{Im}\Bigl(H^0\bigl(\kE(x)\bigr)  \lar  H^0\bigl(\widetilde\kE(1)\bigr)\Bigr)
$$
Let $J: H^0\bigl(\kA(x)\bigr) \lar \lieg[z]$ be the composition
of  $H^0\bigl(\kA(x)\bigr) \lar H^0\bigl(\widetilde\kE(1)\bigr)$ with
the embedding $H^0\bigl(\widetilde\kE(1)\bigr) \lar \liee[z]$ (sending $z_0$ to $1$ and $z_1$ to $z$).
Identifying  $\Sol\bigl((e, d), x\bigr)$ with the corresponding subspace of $\lieg[z]$, we conclude the proof of Theorem \ref{T:recipe}.
\end{proof}

\begin{algorithm}\label{Algorithm geometric} Let $E$ be the cuspidal Weierstra\ss{} curve,
$0 < d < n$ a pair of coprime integers and  $e = n - d$. The solution $r_{(E, (n, d))}$ of the classical
Yang--Baxter equation (\ref{eq:CYBE}) can be obtained along the following lines.
\begin{itemize}
\item First compute the matrix $J = J_{(e, d)}$ given by the recursion (\ref{E:formcanonique}).
\item For $x\in \kk$  determine  the $\kk$--linear subspace $\Sol\bigl((e, d), x\bigr) \subset
\lieg[z]$ introduced
in (\ref{E:DefSol(e,d,x)}).
\item Choose a basis of  $\lieg$ and compute the images of the basis vectors  under the linear  map
$$
\lieg \stackrel{{\overline{\res}}^{-1}_{x}}\lar \Sol\bigl((e, d), x\bigr)
\stackrel{{\overline\ev}_{y}}\lar \lieg.
$$
Here,  $\overline{\res}_{x}(F)=F(x)$ and $\overline{\ev}_{y}(F)=\frac{\displaystyle 1}{\displaystyle y-x}F(y)$.
\item  For fixed $x \ne y \in \kk^*$, set $r_{(E,(n,d))}(x,y)= {\can}^{-1}\bigl({\overline{\ev}}_{y}\circ{\overline{\res}}_{x}^{-1}\bigr) \in  \lieg \otimes \lieg$, where
$\can$ is the  canonical isomorphism of vector spaces \[
\lieg \otimes \lieg \rightarrow  {\End}(\lieg) ,\, X\otimes Y\mapsto\bigl(Z\mapsto {\tr}(XZ)Y\bigr).
\]
\item Then $r_{(E, (n, d))}$ is the solution of the classical Yang--Baxter equation (\ref{eq:CYBE}) corresponding to the triple $(E, (n, d))$.
\qed
\end{itemize}
\end{algorithm}

\noindent
It will be necessary to  have  a more concrete expression  for the coefficients of the tensor
$r_{(E, (n, d))}$. In what follows, we take  the standard basis $\{e_{i,j}\}_{1\leq i\neq j\leq n}\cup\{h_{l}\}_{1\leq l\leq n-1}$
of the Lie algebra $\lieg$. Since the linear map $\overline{\res}_{x}:  \Sol\bigl((e, d), x\bigr) \rightarrow \lieg$
given by $F \mapsto F(x)$ is an isomorphism, we have:
\[
\left\{
\begin{array}{lcll}
{\overline{\res}}_{x}^{-1}(e_{i,j}) & = & e_{i,j}+G_{i,j}^{x}(z) & 1\leq i\neq j\le n, \\
{\overline{\res}}_{x}^{-1}(h_{l}) & = & h_{l}+G_{l}^{x}(z) & 1\leq l\leq n-1,\end{array}
\right.
\]
where the elements $G_{i,j}^{x}(z),G_{l}^{x}(z)\in V_{e,d}$ are uniquely
determined by the properties
\begin{equation}\label{eq: def of Gij}
e_{i,j}+G_{i,j}^{x}(z),h_{l}+G_{l}^{x}(z)\in \Sol\bigl((e, d), x\bigr), \quad
G_{i,j}^{x}(x)=0=G_{l}^{x}(x).
\end{equation}

\begin{lemma}
\label{pro: explicit formula for c(u,v)}In the notations as above,
we have
\begin{equation*}
r_{(E,(n,d))}(x,y)=\frac{1}{y-x}\left[c+\left(\sum_{1\leq i\ne j\leq n}e_{j,i}\otimes G_{i,j}^{x}(y)\right)+\left(\sum_{1\leq l\leq n-1}\check{h}_{l}\otimes G_{l}^{x}(y)\right)\right],
\end{equation*}
where $\check{h}_{l}$ is  the dual of $h_{l}$ with respect to
the trace form and $c$ is  the Casimir element in $\lieg \otimes \lieg$.
In particular, $r_{(E,(n,d))}$ is a
rational solution of (\ref{eq:CYBE}) in the sense of  \cite{Stolin, Stolin2}.
\end{lemma}
\begin{proof}
It follows directly from the definitions that \[
\left\{
\begin{array}{llll}
{\overline{\ev}}_{y}\circ{\overline{\res}}_{x}^{-1}\left(e_{i,j}\right) & = & \frac{\displaystyle 1}{\displaystyle y-x}\left(e_{i,j}+G_{i,j}^{x}(y)\right)
& 1\leq i\neq j\leq n\\
\overline{\ev}_{y}\circ \, {\overline{\res}}_{x}^{-1}\left(h_{l}\right) & = & \frac{\displaystyle 1}{\displaystyle y-x}\left(h_{l}+G_{l}^{x}(y)\right)
& 1\leq l \leq n-1. \end{array}
\right.\]
Since $e_{j,i}$ respectively $\check{h}_{l}$ is the dual of $e_{i,j}$ respectively $h_{l}$ with respect to the trace form on $\lieg$, the linear map  ${\can}^{-1}$ acts as follows:
$$
\left\{
\begin{array}{lcl}
\End(\lieg) \ni \Bigl(e_{i,j}\mapsto\frac{\displaystyle 1}{\displaystyle y-x}\left(e_{i,j}+G_{i,j}^{x}(y)\right)\Bigr) & \mapsto & e_{j,i}\otimes\frac{\displaystyle 1}{\displaystyle y-x}\left(e_{i,j}+G_{i,j}^{x}(y)\right) \in \lieg \otimes \lieg \\
\End(\lieg) \ni \Bigl(h_{l}\mapsto\frac{\displaystyle 1}{\displaystyle y-x}\left(h_{l}+G_{l}^{x}(y)\right)\Bigr) & \mapsto & \check{h}_{l}
\otimes\frac{\displaystyle 1}{\displaystyle y-x}\left(h_{l}+G_{l}^{x}(y)\right) \in \lieg \otimes \lieg
\end{array}
\right.
$$
for $1\leq i\neq j\leq n$ and $ 1 \leq l\leq n-1$. It remains to recall that the Casimir element in
$\lieg \otimes \lieg$ is given by the formula
\begin{equation}\label{E:Casimir}
c =\sum_{1\leq i\neq j\leq n}e_{i,j}\otimes e_{j,i}+\sum_{1\leq l\leq n-1}\check{h}_{l}\otimes h_{l}.
\end{equation}
Lemma is proven.
\end{proof}

\section{Frobenius structure on parabolic subalgebras}

\begin{definition}[see \cite{Ooms}]
A finite dimensional Lie algebra $\lief$ over $\kk$ is \emph{Frobenius} if there exists
a functional $\hat{l} \in \lief^*$ such that the skew-symmetric bilinear form
\begin{equation}
\lief \times \lief \lar \kk \quad (a, b) \mapsto \hat{l}([a, b])
\end{equation}
is non-degenerate.
\end{definition}

\noindent
Let $(e, d)$ be a pair of coprime positive integers, $n = e + d$ and $\liep = \liep_e$
be the $e$-th parabolic subalgebra of  $\lieg = \slie_n(\kk)$, i.e.
\begin{equation}
\liep : =
\left\{\left(\begin{array}{c|c} A & B \\
\hline
 0 & C
\end{array}\right)
 \Big|  \begin{array}{c} A \in \Mat_{e \times e}(\kk), B \in \Mat_{e \times d}(\kk) \\
   C \in \Mat_{d \times d}(\kk) \end{array}
 \; \mbox{and} \;  \tr(A) + \tr(C) = 0 \right\}.
\end{equation}
The goal of this section is to prove the following result.
\begin{theorem}\label{T:FrobAlg}
Let $J = J_{(e, d)}$ be the matrix from   (\ref{E:formcanonique}). Then the pairing
\begin{equation}\label{E:PairingDeFrobenus}
\omega_J: \liep \times \liep \lar \kk, \quad
(a, b) \mapsto \tr\bigl(J^t \cdot [a, b]\bigr)
\end{equation}
is non-degenerate. In other words, $\liep$ is a Frobenius Lie algebra and
\begin{equation}\label{E:FrobFunct}
l_J: \liep \rightarrow \kk,  \quad a \mapsto \tr(J^t \cdot a)
\end{equation}
is a Frobenius functional on $\liep$.
\end{theorem}

\noindent
In this section, we shall use  the following notations and conventions.
For a finite dimensional vector space $\liew$ be denote by $\liew^*$ the  dual vector space.
If $\liew = \liew_1 \oplus \liew_2$ then we have a canonical isomorphism
$\liew^* \cong \liew_1^* \oplus \liew_2^*$. For a functional $\hat{w}_i \in \liew_i^*, i = 1, 2$
we denote by the same symbol its \emph{extension by zero} on the whole  $\liew$.

 \medskip
 \noindent
 Assume we have the following set-up.
\begin{itemize}
\item $\lief$ is a finite dimensional Lie algebra.
\item $\liel\subset \lief$ is a Lie subalgebra and $\lien \subset \lief$ is a commutative Lie ideal
such that
$
\lief = \liel \dotplus \lien,
$
i.e.~$\lief = \liel + \lien$ and $\liel \cap \lien = 0$.
\item There exists  $\hat{n} \in \lien^*$  such that for any $\hat{n}' \in \lien^*$ there
exists $l \in \liel$ such that $\hat{n}' = \hat{n}\bigl([-, l])$ in $\lief^*$. Note that
$\hat{n}\bigl([l', l]) = 0$ for any $l' \in \liel$, hence it is sufficient to check that
for any $m \in \lien$ we have: $\hat{n}'(m) = \hat{n}\bigl([m, l])$. The relation $\hat{n}' = \hat{n}\bigl([-, l])$
  is compatible with the above convention on zero
extension of  functionals from  $\liel$ to $\lief$.
\end{itemize}
First note  the following easy fact.
\begin{lemma}
Let $\hat{m} \in \lien^*$ be any functional and
$
\lies = \lies_{\hat{m}} :=
\Bigl\{l \in \liel \big| \lief^* \ni \hat{m}\bigl([\,-\,,\,l]\bigr) = 0
\Bigr\}.
$
Then $\lies$ is a Lie subalgebra of $\liel$.
\end{lemma}

\noindent
A version of the  following result is due  to Elashvili \cite{Elashvili1982}. It was  explained to  us by Stolin.
\begin{proposition}\label{P:Elashvili}
Let $\lief = \liel \dotplus  \lien$ and $\hat{n} \in \lien^*$ be as above.  Assume there exists
 $\hat{s}\in \liel^*$ such that
its restriction on $\lies = \lies_{\hat{n}}$ is Frobenius. Then
$\hat{s} + \hat{n}$ is a Frobenius functional on $\lief$.
\end{proposition}
\begin{proof}
Assume $\hat{s} + \hat{n}$ is not Frobenius. Then there exist $l_1\in \liel$ and $n_1 \in \lien$ such that
$$
\lief^* \ni (\hat{s} + \hat{n})\bigl([l_1 + n_1, \, -\,]\bigr) = 0.
$$
It is equivalent to say  that for all $l_2\in \liel$ and $n_2 \in \lien$ we have:
\begin{equation}\label{E:intermed1}
\hat{n}\bigl([l_1, n_2] + [n_1, l_2]\bigr) + \hat{s}\bigl([l_1, l_2]\bigr) = 0.
\end{equation}
At the first step, take $l_2 = 0$. Then the equality (\ref{E:intermed1}) implies that
 for all $n_2 \in \lien$
we have: $\hat{n}\bigl([l_1, n_2]\bigr) = 0$. This means that
$\lief^* \ni \hat{n}\bigl([\,-\,,l_1]\bigr) = 0$ and hence, by definition of $\lies$,  $l_1 \in \lies$.
Assume $l_1 \ne 0$. By assumption,
 $\hat{s}\big|_{\lies}$ is a Frobenius functional. Hence, there exists
$s_1 \in \lies$ such that $\hat{s}\bigl([l_1, s_1]\bigr) \ne 0$.
Since $s_1 \in \lies$, we have: $\hat{n}\bigl([n_1, s_1]\bigr) = 0$. Altogether,
it implies:
$$
(\hat{s} + \hat{n})\bigl([l_1 + n_1, s_1]\bigr) =
\hat{s}\bigl([l_1, s_1]\bigr)\ne 0.
$$
Contradiction. Hence, $l_1 = 0$ and the equation (\ref{E:intermed1}) reads as follows:
$$
\hat{n}\bigl([n_1, l_2]\bigr) = 0 \quad \mbox{for  all} \quad l_2 \in \liel.
$$
Assume $n_1 \ne 0$. Then there exists a functional $\hat{n}_1 \in \lien^*$ such that
$\hat{n}_1(n_1) \ne 0$. However, by our assumptions, $\hat{n}_1 =
\hat{n}\bigl([\,-\,,l]\bigr)$ for some $l \in \liel$. But this implies that
$$
\hat{n}_1(n_1) = \hat{n}\bigl([n_1, l]\bigr)  \ne 0.
$$
We again obtain a contradiction. Thus, $n_1 = 0$ as well,  what finishes the proof.
\end{proof}

\medskip
\noindent
Consider the following nilpotent subalgebras of $\lieg$:
\begin{equation}\label{E:DeofOfN}
\lien = \Bigl\{N = \left(
\begin{array}{c|c}
0 & A \\
\hline
0 & 0
\end{array}\right) \Big| A \in \Mat_{e \times d}(\kk)\Bigr\} \quad
\bar{\lien} = \Bigl\{\bar{N} = \left(
\begin{array}{c|c}
0 & 0 \\
\hline
\bar{A} & 0
\end{array}\right) \Big| A \in \Mat_{d \times e}(\kk)\Bigr\}.
\end{equation}
Note the following easy fact.
\begin{lemma}\label{L:DualOfN}
The linear map $
\bar{\lien} \lar \lien^*, \; \bar{N} \mapsto \tr\bigl(\bar{N} \cdot \,-\,\bigr)
$
is an isomorphism.
\end{lemma}

\medskip
\noindent
Next, consider the following Lie algebra
\begin{equation}\label{E:DeofOfL}
\liel = \Bigl\{L =
\left(\begin{array}{c|c}
L_1 & 0 \\
\hline  0 & L_2
\end{array}\right) \Big| \begin{array}{l}L_1 \in \Mat_{e \times e}(\kk) \\
 L_2 \in \Mat_{e \times e}(\kk) \end{array} \;
\tr(L_1) + \tr(L_2) = 0\Bigr\}.
\end{equation}
Obviously,
$
\liep = \liel \dotplus  \lien,
$
$\liep$ is a Lie subalgebra of $\liep$ and  $\lien$ is a commutative Lie ideal in $\liep$.
\begin{lemma}\label{L:ConcreteCond}
Let $\bar{N} \in \bar\lien$ and $\hat{n}  = \tr(\bar{N} \cdot \,-\,) \in \lien^*$ be the
corresponding functional. Then the  condition that \textsl{for any $\hat{n}' \in \lien^*$ there
exists $L \in \liel$ such that $\hat{n}' = \hat{n}\bigl([L, -])$ in $\lief^*$} reads as follows: \textsl{for any $\bar{N}' \in \bar\lien$ there exists $L \in \liel$ such that
$\bar{N}' = [\bar{N}, L]$}.
\end{lemma}
\begin{proof} By Lemma \ref{L:DualOfN} there exists $\bar{N}' \in \bar\lien$
such that $\hat{u} = \tr(\bar{N}' \cdot \,-\,)$. Note that
$$
\tr\bigl(\bar{N}\cdot[L, \,-\,]\bigr) = \tr\bigl([\bar{N}, L] \cdot \,-\,\bigr).
$$
The equality of functionals
$
\tr(\bar{N}' \cdot \,-\,) = \tr\bigl([\bar{N}, L]\cdot \,-\,\bigr)
$
implies that $\bar{N}' = [\bar{N}, L]$.
\end{proof}

\medskip
\noindent
\emph{Proof of Theorem \ref{T:FrobAlg}}. We prove this result by induction on
$$(e, d) \in
\Sigma=\bigl\{(a, b) \in \mathbb{N} \times \mathbb{N} \, \big| \, \gcd(a, b) = 1 \bigr\}.$$

\noindent
\textsl{Basis of induction}. Let $(e, d) = (1, 1)$. Then we have:
$
J = J_{(1, 1)} =
\left(
\begin{array}{cc}
0 & 1 \\
0 & 0
\end{array}
\right).
$
Let  $a =
\left(
\begin{array}{cr}
\alpha_1 & \alpha_2 \\
0 & -\alpha_1
\end{array}
\right)
$
and $b =
\left(
\begin{array}{cr}
\beta_1 & \beta_2 \\
0 & -\beta_1
\end{array}
\right)
$
be two elements of $\liep$.
Then we have:
$$
\omega_J(a, b) = 2\cdot(\alpha_1 \beta_2 - \beta_1 \alpha_2).
$$
This form is obviously non-degenerate.

\medskip
\noindent
\textsl{Induction step}. Assume the result is proven for $(e, d) \in \Sigma$.
 Recall that for
\[  J_{(e, d)} =\left(
{\begin{array}{c | c}
A_{1} & A_{2} \\ \hline
0 & A_{3}
\end{array}}
\right)
\] with $A_{1}\in\mbox{Mat}_{e\times e}(\kk)$ and $A_{3}\in\mbox{Mat}_{d\times d}(\kk)$
we have:
$$
J_{(e, d + e)} =
\left(
{\begin{array}{c || c |c}
0 & \mathbbm{1} & 0 \\ \hline \hline
0 & A_{1} & A_{2} \\
\hline
0 & 0 & A_{3}
\end{array}}
\right)
\quad \mbox{and} \quad
J_{(d + e, d)} =
\left(
{\begin{array}{c | c || c}
A_{1} & A_{2} & 0 \\
\hline
0 & A_{3} & \mathbbm{1} \\ \hline \hline
0 & 0 & 0
\end{array}}
\right).
$$
For simplicity, we shall only treat the implication  $(e, d) \Longrightarrow (e, d + e)$. Consider the matrix
$$
\bar{N} =
\left(
\begin{array}{c||c|c}
0 & 0 & 0 \\
\hline
\hline
\mathbbm{1} & 0 & 0 \\
\hline
0 & 0 & 0
\end{array}
\right) \in \bar\lien.
$$
Then the following facts follows from a direct computation:
\begin{itemize}
\item $\bar{N}$ satisfies the condition of Lemma \ref{L:ConcreteCond}.
\item The Lie subalgebra $\lies = \lies_{\bar{N}}$ has the following description:
\begin{equation}\label{E:AlgebraS}
\lies = \left\{
\left(
\begin{array}{c||c|c}
A & 0 & 0 \\
\hline
\hline
0 & A & B \\
\hline
0 & 0 & C
\end{array}
\right) \Big|  \begin{array}{c} A \in \Mat_{e \times e}(\kk), B \in \Mat_{e \times d}(\kk), \\
 C \in \Mat_{d \times d}(\kk)
 \end{array}\; 2 \tr(A) + \tr(C) = 0
 \right\}.
\end{equation}
\end{itemize}
The implication $(e, d) \Longrightarrow (e, d + e)$ follows from
Proposition \ref{P:Elashvili} and the following  result.

\begin{lemma} Let $\hat{J} = \left(
\begin{array}{c||c|c}
0 & 0 & 0 \\
\hline
\hline
0 & A & B \\
\hline
0 & 0 & C
\end{array}
\right)$. Then there exists an isomorphism of Lie algebras $\nu: \liep \lar \lies$ such
that for any $P \in \liep$ we have: $\tr(J^t \cdot P) = \tr\bigl(\hat{J}^t \cdot \nu(P)\bigr)$.
\end{lemma}

\medskip
\noindent
The proof of this lemma is lengthy but
 completely elementary, therefore we leave it  to an interested reader.
Theorem \ref{T:FrobAlg} is proven.\qed

\begin{lemma}\label{L:FrobeniusSplitting}
For any $G \in \lieg$ there exist uniquely determined  $P \in \liep$ and $N \in \lien$ such that
\begin{equation*}
G = \bigl[J^t, P\bigr] + N.
\end{equation*}
\end{lemma}
\begin{proof} Consider the functional $\tr(G \cdot \,-\,) \in \liep^*$. Since the functional $l_J \in \liep^*$
from (\ref{E:FrobFunct}) is Frobenius, there exists a uniquely determined $P \in \liep$ such
that $\tr(G \cdot \,-\,) = \tr\bigl([J^t, P] \cdot \,-\, \bigr)$ viewed as elements of $\liep^*$.
Note that we have a short exact sequence of vector spaces
$$
0 \lar \lien \stackrel{\imath}\lar \lieg^* \stackrel{\rho}\lar \liep^* \lar 0,
$$
where $\rho$ maps a functional on $\lieg$ to its restriction on $\liep$ and
$\imath(N) = \tr(N \cdot \,-\,)$. Thus, for some uniquely determined $N \in \lien$, we get the following equality in $\lieg^*$:
$\tr(G \cdot \,-\,)
= \tr\bigl(([J^t, P] + N) \cdot \,-\,)$.
 Since the trace form is non-degenerate on $\lieg$, we get the  claim.
\end{proof}

\section{Review of Stolin's theory of rational solutions of the classical Yang--Baxter equation}

\noindent
In this section,  we  review Stolin's results on the classification of rational solutions of the classical
Yang--Baxter equation for the Lie algebra $\lieg = \slie_n(\CC)$, see \cite{Stolin, Stolin3, Stolin2}.

\begin{definition}
A solution $r: (\mathbb{C}^2, 0) \lar \lieg \otimes \lieg$ of  (\ref{eq:CYBE}) is called rational if it is non-degenerate, unitary and of  the form
\begin{equation}\label{E:AnsatzStolin}
r(x,y)=\frac{c}{y-x}+ s(x,y),
\end{equation}
where $c \in \lieg \otimes \lieg$ is the Casimir element and
$s(x, y) \in \mathfrak{g}[x]\otimes\mathfrak{g}[y]$.
\end{definition}

\subsection{Lagrangian orders}
Let $\widehat\lieg = \lieg((z^{-1}))$. Consider the following non-degenerate $\CC$--bilinear form
on $\widehat\lieg$:
\begin{equation}\label{E:KacMoodyPairing}
(\,-\,,\,-\,): \quad \widehat\lieg \times \widehat\lieg \lar \mathbbm{C}, \quad
(a, b) \mapsto \res_{z = 0}\bigl(\tr(a b)\bigr).
\end{equation}
\begin{definition}
A Lie subalgebra $\liew \subset \widehat\lieg$ is a  \emph{Lagrangian order} if the following
three conditions are satisfied.
\begin{itemize}
\item $\liew \dotplus \lieg[z] = \widehat\lieg$.
\item $\liew = \liew^\perp$ with respect to the pairing (\ref{E:KacMoodyPairing}).
\item There exists $p \ge 0$ such that
$z^{-p-2} \lieg\llbracket z^{-1}\rrbracket \subseteq \liew$.
\end{itemize}
\end{definition}
\noindent
Observe  that from this Definition automatically follows that
\begin{equation*}
\liew = \liew^\perp \subseteq \bigl(z^{-p-2} \lieg\llbracket z^{-1}\rrbracket\bigr)^\perp
= z^p \lieg\llbracket z^{-1} \rrbracket.
\end{equation*}
Moreover, the restricted pairing
\begin{equation}\label{E:RestrPairing}
(\,-\,,\,-\,): \; \liew \times \lieg[z] \lar \CC
\end{equation}
is non-degenerate, too.  Let $\bigl\{\alpha_l\bigr\}_{l = 1}^{n^2-1}$ be a basis of $\lieg$ and
$\alpha_{l, k} = \alpha_l z^k \in \lieg[z]$ for $1 \le l \le n^2-1$, $k\ge 0$. Let
$
\beta_{l, k} := \alpha_{l, k}^\vee  \in \liew
$
be the dual element of $\alpha_{l, k} \in \lieg[z]$
with respect to the pairing (\ref{E:RestrPairing}). Consider the following formal power series:
\begin{equation}\label{E:SolAttToOrder}
r_{\liew}(x,y) =
\sum_{k = 0}^\infty x^{k}\left(\sum_{l=1}^{n^{2}-1}\alpha_{l}\otimes {\beta}_{l, k}(y)\right).
\end{equation}

\begin{theorem}[see \cite{Stolin, Stolin3}] The following results are true.
\begin{itemize}
\item The formal power series (\ref{E:SolAttToOrder})  converges to a rational function.
\item Moreover, $r_{\liew}$  is a rational solution
of (\ref{eq:CYBE})
satisfying Ansatz (\ref{E:AnsatzStolin}).
\item A different choice of a basis
of $\lieg$ leads to a gauge-equivalent solution.
\item Other way around,  for any solution $r$ of (\ref{eq:CYBE}) satisfying (\ref{E:AnsatzStolin}),
there exists a Lagrangian order $\liew \subset \widehat\lieg$ such that $r = r_{\liew}$.
\item Let $\sigma$ be any $\CC[z]$--linear automorphism of $\lieg[z]$ and
$\lieu = \sigma(\liew) \subset \widehat\lieg$ be the transformed order. Then the solutions
$r_{\liew}$ and $r_{\lieu}$ are gauge-equivalent:
\begin{equation*}
r_{\lieu}(x, y) = \bigl(\sigma(x) \otimes \sigma(y)\bigr) r_{\liew}(x, y).
\end{equation*}
\item The described  correspondence $\liew \mapsto r_{\liew}$ provides a bijection between  the gauge equivalence
classes of rational solutions of (\ref{eq:CYBE}) satisfying (\ref{E:AnsatzStolin})  and the
orbits of Lagrangian orders in $\widehat\lieg$ with respect to the action of
of
$\Aut_{\CC[z]}\bigl(\lieg[z]\bigr)$.
\end{itemize}
\end{theorem}

\begin{example}
Let $\liew = z^{-1} \lieg\llbracket z^{-1}\rrbracket$. It is easy to see that $\liew$ is a Lagrangian order in $\widehat\lieg$. Let  $\bigl\{\alpha_l\bigr\}_{l = 1}^{n^2-1}$ be any  basis of $\lieg$. Then we have:
$
\beta_{l, k} :=  \bigl(\alpha_l z^k\bigr)^\vee = \alpha_l^\vee z^{-k-1}.
$
This implies:
\begin{equation}
r_{\liew}(x, y) =
\sum\limits_{k=0}^\infty x^k \sum\limits_{l = 1}^{n^2-1} \alpha_l \otimes \alpha_l^\vee y^{-k-1}
= \frac{c}{y-x},
\end{equation}
where $c \in \lieg \otimes \lieg$ is the Casimir element. The tensor-valued function
 $r_{\liew}$ is the celebrated Yang's solution
of the classical Yang--Baxter equation (\ref{eq:CYBE}).
\end{example}

\begin{lemma}\label{L:DecompInW}
For any $1 \le l \le n^2-1$ and $k \ge 0$ there exists a unique
$w_{l, k} \in \lieg[z]$ such that
\begin{equation*}
\beta_{l, k} = z^{-k-1}\alpha_l^\vee  + w_{l, k}.
\end{equation*}
\end{lemma}

\begin{proof}
It is an easy consequence of the assumption $\liew \dotplus \lieg[z] = \widehat\lieg$ and
the fact that the pairing (\ref{E:RestrPairing}) is non-degenerate.
\end{proof}

\subsection{Stolin triples} As we have seen in the previous subsection, the classification
of rational solutions of (\ref{eq:CYBE}) reduces to a description of Lagrangian orders.
This correspondence is actually valid for arbitrary simple complex Lie algebras \cite{Stolin3}.
In the special case $\lieg = \slie_n(\CC)$, there is an explicit parametrization of Lagrangian orders
in the following Lie--theoretic terms \cite{Stolin, Stolin2}.

\begin{definition}
A \emph{Stolin
triple} $(\mathfrak{l},k, \omega)$ consists of

\begin{itemize}
\item a Lie subalgebra $\mathfrak{l}\subseteq\mathfrak{g}$,
\item an integer $k$ such that $0 \leq k \leq n$,
\item a skew symmetric bilinear form $\omega:\mathfrak{l}\times\mathfrak{l}\rightarrow\mathbb{C}$
which is a 2-cocycle, i.e.
\[
\omega\bigl(\left[a,b\right],c\bigr)+\omega\bigl(\left[b,c\right],a\bigr)+
\omega\bigl(\left[c,a\right],b\bigr)=0
\]
for all $a, b, c \in \mathfrak{l}$,
\end{itemize}
such that for the $k$-th parabolic Lie subalgebra
 $\mathfrak{p}_{k}$ of $\mathfrak{g}$ (with $\liep_0 = \liep_n = \lieg$)
the following two conditions are fulfilled:
\begin{itemize}
\item $\mathfrak{l}+\mathfrak{p}_{k}=\mathfrak{g}$,
\item $\omega$ is non-degenerate on $\left(\mathfrak{l}\cap\mathfrak{p}_{k}\right)\times\left(\mathfrak{l}\cap\mathfrak{p}_{k}\right)$.
\end{itemize}
\end{definition}

\noindent
According to Stolin \cite{Stolin}, up to the action of $\Aut_{\CC[z]}\bigl(\lieg[z]\bigr)$,
any Lagrangian order in $\widehat\lieg$ is given by some triple $(\liel, k, \omega)$.
 In this article, we shall only need the case $\liel = \lieg$.

\begin{algorithm}\label{A:FromTripletoOrder} One can pass from a Stolin triple $(\lieg, k, \omega)$ to the corresponding Lagrangian
order  $\liew \subset \lieg((z^{-1}))$ in the following way.
\begin{itemize}
\item Consider the following linear subspace
\begin{equation}
\liev_{\omega} = \bigl\{z^{-1}a +b \, \big| \, \tr(a\cdot \,-\,) = \omega(b, \,-\,)  \in \liel^* \bigr\} \subset z^{-1} \lieg \dotplus \liel \subset
 z^{-1} \lieg \dotplus \lieg \subset \widehat\lieg.
\end{equation}
\item The subspace $\liev_\omega$ defines the following linear subspace
\begin{equation}\label{E:DefOfWPrime}
\liew' = z^{-2}\lieg\llbracket z^{-1}\rrbracket \dotplus \liev_\omega
 \subset \widehat\lieg.
\end{equation}
\item Consider the matrix
\begin{equation}
\eta = \left(
\begin{array}{c | c}
\mathbbm{1}_{k \times k} & 0 \\ \hline
0 & z \cdot \mathbbm{1}_{(n-k) \times (n-k)}
\end{array}
\right)
\in \GL_{n}\bigl(\CC[z, z^{-1}]\bigr).
\end{equation}
and put:
\begin{equation}
\liew = \liew_{(\liel, k, \omega)}:= \eta^{-1} \liew' \eta \subset \widehat\lieg.
\end{equation}
\end{itemize}
\end{algorithm}

\noindent
The next theorem  is due to Stolin \cite{Stolin, Stolin3}, see also
\cite[Section 3.2]{ChariPressley} for a more detailed account of  the  theory of rational
solutions of the classical Yang--Baxter equation (\ref{eq:CYBE}).

\begin{theorem}\label{T:Stolin} The following results are true.
\begin{itemize}
\item The linear subspace $\liew \subset \widehat\lieg$ is a Lagrangian order.
\item  For any Lagrangian order $\liew \subset \widehat\lieg$ there exists  $\alpha \in \Aut_{\CC[z]}\bigl(\lieg[z]\bigr)$
and a  Stolin triple   $(\liel, k, \omega)$ such that $\liew = \alpha\bigl(\liew_{(\liel, k, \omega)}\bigr)$.
\item
Two Stolin triples $(\liel, k, \omega)$ and $(\liel', k, \omega')$ define equivalent  Lagrangian
orders  in $\widehat\lieg$ with respect  to the $\Aut_{\CC[z]}(\lieg[z])$--action  if and
only if there exists a Lie algebra automorphism $\gamma$ of $\lieg$ such that
$\gamma(\liel) = \liel'$ and $\gamma^*\bigl([\omega]\bigr) = \omega' \in H^2(\liel)$.
\end{itemize}
\end{theorem}

\begin{remark}\label{R:Defect} Unfortunately,
the described  correspondence between Stolin triples and Lagrangian orders
has the following defect:  the parameter $k$ is not an invariant of $\liew$. This
leads to the fact that two completely different Stolin triples
$(\liel, k, \omega)$ and $(\liel', k', \omega')$ can define  the same Lagrangian order $\liew$.
\end{remark}

\begin{remark}\label{R:wild}
Consider an arbitrary even-dimensional abelian Lie subalgebra $\lieb \subset \lieg$ equipped with
an arbitrary non-degenerate skew-symmetric bilinear form $\omega: \lieb \times \lieb \lar \CC$.
Obviously, $\omega$ is a two-cocycle and we get a Stolin triple
$(\lieb, 0, \omega)$. Two such triples $(\lieb, 0, \omega)$ and $(\lieb', 0, \omega')$
define equivalent Lagrangian orders if and only if there exists $\alpha \in \Aut(\lieg)$ such that
$\alpha(\lieb) = \lieb'$. However, the classification of abelian subalgebras in $\lieg$ is essentially
equivalent to the classification of finite dimensional $\CC[u, v]$--modules. By a result of Drozd
 \cite{Drozd}, the last problem is  \emph{representation-wild}. Thus, as it was already pointed out by Belavin and Drinfeld in \cite[Section 7]{BelavinDrinfeld}, one can not hope
 to achieve  a full classification
 of all rational solutions of the classical Yang--Baxter equation  (\ref{eq:CYBE}).
\end{remark}

\begin{remark} In this article, we only deal with  those  Stolin triple $(\lieg, e, \omega)$
for which $\liel = \lieg$.  It leads to the following significant simplifications.
Consider the linear map
\begin{equation}\label{E:DefMapChi}
\chi: \;  \lieg \stackrel{l_\omega}\lar \lieg^* \stackrel{\tr}\lar \lieg,
\end{equation}
where $l_\omega(a) = \omega(a, \,-\,)$ and $\tr$ is the isomorphism
induced by the trace form.
Then
\begin{equation*}
\liev_\omega = \bigl\langle\alpha+ z^{-1}\chi(\alpha)\bigr\rangle _{\alpha \in \mathfrak{\lieg}}.
 \end{equation*}
 Next, by Whitehead's Theorem, we have the vanishing  $H^2(\lieg) = 0$. This means that for any two-cocycle
 $\omega: \lieg \times \lieg \lar \CC$ there exist a matrix $K \in \Mat_{n \times n}(\CC)$ such that
 for all $a, b \in \lieg$ we have: $
 \omega(a, b) = \omega_K(a, b) := \tr(K^t \cdot \bigl([a, b]\bigr).$
 Let $1 \le e \le n$ be such that
 $\gcd(n, e) = 1$. Then  the parabolic subalgebra $\liep_e$ is Frobenius.
 If $(\lieg, e, \omega)$ is a Stolin triple then $\omega_K$ has to define a Frobenius pairing
 on $\liep_e$. If $K' \in \Mat_{n \times n}(\CC)$ is any other matrix such that
 $\omega_{K'}$ is non-degenerate on $\liep_e \times \liep_e$  then  the triples
 $(\lieg, e, \omega_{K})$ and $(\lieg, e, \omega_{K'})$ define gauge equivalent solutions of the
 classical Yang--Baxter equation. This means that the gauge equivalence class of the solution
 $r_{(\lieg, e, \omega)}$ does not depend on a particular choice of $\omega$! However, in order
 to get nice closed formulae for  $r_{(\lieg, e, \omega)}$, we actually need the canonical
 matrix  $J_{(e, d)} \in \Mat_{n \times n}(\CC)$ constructed by recursion  (\ref{E:formcanonique}).
\end{remark}

\section{From vector bundles to the cuspidal Weierstra\ss{} curve to Stolin triples}

\noindent
For reader's convenience,  we recall   once again our notation.
\begin{itemize}
\item $E$ is  the cuspidal Weierstra\ss{} curve.
\item  $(e, d)$  is a pair of positive coprime e integers and  $n = e + d$.
\item  $\lieg = \slie_n(\CC)$, $\liee = \mathfrak{gl}_n(\CC)$,
$\liep = \liep_e \subset \lieg$ is the $e$-th parabolic subalgebra of $\lieg$.
We have a decomposition $\liep = \liel \dotplus \lien$, where $\lien$ (respectively $\liel$) is  defined by (\ref{E:DeofOfN}) (respectively (\ref{E:DeofOfL})), $\bar\lien$ is the transpose of $\lien$.
 \item $J = J_{(e, d)} \in
\liee$ is the matrix constructed by  recursion (\ref{E:formcanonique})
 and $\omega: \liep \times \liep \lar \CC$ is the corresponding
Frobenius pairing (\ref{E:PairingDeFrobenus}).
\item
For  $1 \le i, j \le n$,  let $e_{i, j} \in \liee$ be the corresponding matrix unit,
$h_l = e_{l, l} - e_{l+1, l+1}$ for $1 \le l \le n-1$ and $\check{h}_l$ be its dual with respect to
the trace form.  Let  $c \in \lieg \otimes \lieg$ be the Casimir element with respect to the trace form.
\item Finally,  the decomposition $n = e + d$ divides the set  $\bigl\{(i, j)  \in
\mathbb{Z}^2 \,|\, 1 \le i, j \le n\bigr\}$
in four parts, according to the following convention:
$
\left(
\begin{array}{c | c} \mathrm{\rom{4}} &  \mathrm{\rom{1}} \\ \hline \mathrm{\rom{3}} &
\mathrm{\rom{2}}
\end{array}
\right).
$
\end{itemize}
The main results of this section are the following:
\begin{itemize}
\item We derive an explicit formula for the rational solution $r_{(\lieg, e, \omega)}$
of the classical Yang--Baxter equation (\ref{eq:CYBE}) attached to  Stolin triple $(\lieg, e, \omega)$.
\item We prove that the solutions $r_{(E, (n, d))}$ and  $r_{(\lieg, e, \omega)}$ are gauge--equivalent.
\end{itemize}

\subsection{Description of the rational solution $r_{(\lieg, e, \omega)}$}

\begin{lemma}
The linear map $\chi: \lieg \lar \lieg$ from (\ref{E:DefMapChi})
 is given by the rule
$a \mapsto \bigl[J^t, a\bigr]$.
\end{lemma}
\begin{proof}
For $a, b \in \lieg$ we have:
$
\omega(a, b) = \tr\bigl(J^t \cdot [a, b]\bigr) = \tr\bigl([J^t, a] \cdot b\bigr).
$
Hence, the linear map $l_\omega: \lieg \lar \lieg^*$ is given by the formula
$a \mapsto \tr\bigl([J^t, a] \cdot \,-\,\bigr)$. This implies the claim.
\end{proof}

\begin{lemma}
Let $\liew \subset \widehat\lieg$ be the Lagrangian order constructed from   Stolin triple
$(\lieg, e, \omega)$ following  Algorithm \ref{A:FromTripletoOrder}.
Then we have the following inclusions:
\begin{equation*}
\liew_1 :=
z^{-3} \bar{n}\llbracket z^{-1}\rrbracket \oplus
z^{-2} \liel\llbracket z^{-1}\rrbracket \oplus
z^{-1} {n}\llbracket z^{-1}\rrbracket \subset
\liew
\subset
z^{-1} \bar{n}\llbracket z^{-1}\rrbracket \oplus
 \liel\llbracket z^{-1}\rrbracket \oplus
z \lien\llbracket z^{-1}\rrbracket := \liew_2.
\end{equation*}
\end{lemma}

\begin{proof}
This result is an immediate consequence of the inclusions
$
z^{-2} \lieg\llbracket z^{-1} \rrbracket \subset \liew' \subset
\lieg\llbracket z^{-1} \rrbracket,
$
and the fact that $\liew = \eta^{-1} \, \liew' \, \eta$.
\end{proof}

\begin{lemma}
For any $1 \le i \ne j \le n$, $1 \le l \le n-1$ and $k \ge 0$,  consider the elements
$u_{(i, j; k)}, u_{(l; k)} \in \lieg[z]$ such that
\begin{equation}
(z^k e_{i, j})^\vee = z^{-k-1} e_{j, i} + u_{(i, j; k)} \in \liew
\quad \mbox{and} \quad
(z^k \check{h}_l)^\vee = z^{-k-1} h_l  + u_{(l; k)} \in \liew.
\end{equation}
Then the following statements are true.
\begin{itemize}
\item For all $1 \le i \ne j \le n$ and $k \ge 2$ we have: $u_{(i, j; k)} = 0$.
\item For all $(i, j) \in \mathrm{\rom{2}} \cup \mathrm{\rom{4}}$, $i \ne j$,
we have: $u_{(i, j; 1)} = 0$.
\item Similarly, for
 all $1 \le l \le n-1$ and $k \ge 1$ we have: $u_{(l; k)} = 0$.
\item For all $(i, j) \in \mathrm{\rom{3}}$ and $k = 0, 1$ we have:
$u_{(i, j; k)} = 0$.
\item Finally, all non-zero elements $u_{(i, j;k)}$ and $u_{(l; k)}$ belong
to $\liep \dotplus z \lien$.
\end{itemize}
\end{lemma}

\begin{proof}
According to Lemma \ref{L:DecompInW},
 the elements $u_{(i, j; k)}$ (respectively  $u_{(l; k)}$) are uniquely  determined by the
property that $z^{-k-1} e_{j, i} + u_{(i, j; k)} \in \liew$ (respectively,
$z^{-k-1} h_l  + u_{(l; k)} \in \liew$). Hence,  the first four  statements
are immediate corollaries of the inclusion $\liew_1 \subset \liew$.  On the other hand, the last result follows
from the inclusion $\liew \subset \liew_2$.
\end{proof}

\noindent
In order to get a more concrete description of non-zero elements $u_{(i, j; k)}$ and $u_{(l; k)}$,
 note the following result.

\begin{lemma} Let $K \in \Mat_{n \times n}(\CC)$ be any matrix defining
a non-degenerate pairing $\omega_K$ on $\liep \times \liep$.
The following statements  are  true.
\begin{itemize}
\item For any $(i, j) \in \mathrm{\rom{2}} \, \cup \, \mathrm{\rom{4}}$, $i \ne j$,  there exist
uniquely determined
$\left(\begin{array}{c|c} A_{(i, j)}^{(0)} & B_{(i, j)}^{(0)} \\ \hline  0 & D_{(i, j)}^{(0)}\end{array}\right) \in \liep$ and
$\left(\begin{array}{c|c} 0 & \widetilde{B}_{(i, j)}^{(0)} \\ \hline 0 & 0\end{array}\right) \in \lien$ such that
\begin{equation}\label{E:dec1}
e_{j, i} - \left[
K^t, \left(\begin{array}{c|c} A_{(i, j)}^{(0)} & \widetilde{B}_{(i, j)}^{(0)} \\ \hline 0 & D_{(i, j)}^{(0)}\end{array}\right)
\right]
+ \left(\begin{array}{c|c} 0 & B_{(i, j)}^{(0)} \\ \hline  0 & 0
\end{array}\right) = 0.
\end{equation}
\item Similarly, for any $1 \le l \le n-1$, there exist
uniquely determined
$\left(\begin{array}{c|c} A_{(l)} & B_{(l)} \\ \hline  0 & D_{(l)}\end{array}\right) \in \liep$ and
$\left(\begin{array}{c|c} 0 & \widetilde{B}_{(l)} \\ \hline 0 & 0\end{array}\right) \in \lien$ such that
\begin{equation}\label{E:dec2}
h_l - \left[
K^t, \left(\begin{array}{c|c} A_{(l)} & \widetilde{B}_{(l)} \\ \hline 0 & D_{(l)}\end{array}\right)
\right]
+ \left(\begin{array}{c|c} 0 & B_{(l)} \\ \hline  0 & 0
\end{array}\right) = 0.
\end{equation}
\item Finally, for any $(i, j) \in \mathrm{\rom{1}}$ and $k = 0, 1$, there exist
uniquely determined matrices
$\left(\begin{array}{c|c} A_{(i, j)}^{(k)} & B_{(i, j)}^{(k)} \\ \hline  0 & D_{(i, j)}^{(k)}
\end{array}\right) \in \liep$ and
$\left(\begin{array}{c|c} 0 & \widetilde{B}_{(i, j)}^{(k)} \\ \hline 0 & 0\end{array}\right) \in \lien$ such that
\begin{equation}\label{E:dec3}
\left[
K^t,  e_{j, i} + \left(\begin{array}{c|c} A_{(i, j)}^{(0)} & \widetilde{B}_{(i, j)}^{(0)} \\ \hline  0 & D_{(i, j)}^{(0)}
\end{array}\right)
\right] = \left(\begin{array}{c|c} 0 & {B}_{(i, j)}^{(0)} \\ \hline 0 & 0\end{array}\right)
\end{equation}
and
\begin{equation}\label{E:dec4}
e_{j, i} - \left[
K^t, \left(\begin{array}{c|c} A_{(i, j)}^{(1)} & \widetilde{B}_{(i, j)}^{(1)} \\ \hline  0 & D_{(i, j)}^{(1)}
\end{array}\right)
\right] + \left(\begin{array}{c|c} 0 & {B}_{(i, j)}^{(1)} \\ \hline 0 & 0\end{array}\right) = 0.
\end{equation}
\end{itemize}
\end{lemma}

\begin{proof}
All these results follow directly from Lemma \ref{L:FrobeniusSplitting}.
\end{proof}

\begin{definition}\label{D:coeffTensorR}
Consider the following elements in the Lie algebra $\lieg[z]$:
\begin{itemize}
\item For $(i, j) \in \mathrm{\rom{3}}$,  we put : ${w}_{(i, j; 0)} =  0 = {w}_{(i, j; 1)}$.
\item For $(i, j) \in \mathrm{\rom{2}} \, \cup \,  \mathrm{\rom{4}}$ such that  $ i \ne j$, we set:
\begin{equation}\label{E:def1}
{w}_{(i, j; 0)} = \left(\begin{array}{c|c} A_{(i, j)}^{(0)} & B_{(i, j)}^{(0)} \\ \hline  0 & D_{(i, j)}^{(0)} \end{array}\right) +
z \left(\begin{array}{c|c} 0 & \widetilde{B}_{(i, j)}^{(0)} \\ \hline 0 & 0
\end{array}\right),
\end{equation}
where $\left(\begin{array}{c|c} A_{(i, j)}^{(0)} & B_{(i, j)}^{(0)} \\ \hline  0 & D_{(i, j)}^{(0)} \end{array}\right)$ and
$\left(\begin{array}{c|c} 0 & \widetilde{B}_{(i, j)}^{(0)} \\ \hline 0 & 0
\end{array}\right)$  are given by (\ref{E:dec1}). Moreover, we set ${w}_{(i, j; 1)} = 0$.
\item Similarly, for $1 \le l \le n-1$, following (\ref{E:dec2}),  we put
\begin{equation}\label{E:def2}
{w}_{(l; 0)} = \left(\begin{array}{c|c} A_{(l)} & B_{(l)} \\ \hline  0 &
D_{(l)}\end{array}\right) +
z \left(\begin{array}{c|c} 0 & \widetilde{B}_{(l)} \\ \hline 0 & 0
\end{array}\right),
\end{equation}
whereas $w_{(l; 1)} = 0$.
\item Finally, for $(i, j) \in \mathrm{\rom{1}}$  and $k = 0, 1$,  following
(\ref{E:dec3}) and (\ref{E:dec4}), we write
\begin{equation}\label{E:def3}
{w}_{(i, j; k)} = \left(\begin{array}{c|c} A_{(i, j)}^{(k)} & B_{(i, j)}^{(k)} \\ \hline  0 &
D_{(i, j)}^{(k)}\end{array}\right) +
z \left(\begin{array}{c|c} 0 & \widetilde{B}_{(i, j)}^{(k)} \\ \hline 0 & 0
\end{array}\right).
\end{equation}
\end{itemize}
\end{definition}

\medskip
\noindent
Now we are ready to prove the main result of this subsection.
\begin{theorem} Stolin triple $(\lieg, e, \omega_K)$ defines  the following solution of (\ref{eq:CYBE}):
{\scriptsize \begin{equation*}
r_{(\lieg, e, \omega_K)}(x, y)  =
\frac{c}{y-x} +
\sum\limits_{1 \le i \ne j \le n} e_{i, j} \otimes {w}_{(i, j; 0)}(y)
+ \sum\limits_{1 \le l \le n-1} \check{h}_l \otimes {w}_{(l; 0)}(y)
+ x \sum\limits_{1 \le i \ne j \le n} e_{i, j} \otimes {w}_{(i, j; 1)}(y).
\end{equation*}}
\end{theorem}

\begin{proof}
It is sufficient to show that for any  $1 \le i \ne j \le n$, $1 \le l \le n-1$
 and $k = 0, 1$,  we have the following equalities:
\begin{equation}\label{E:comparison}
u_{(i, j; k)} = w_{(i, j; k)}
\quad
\mbox{and}
\quad
u_{(l; k)} = w_{(l; k)}.
\end{equation}
Recall that $\liew = \liew_1 \dotplus  \eta^{-1} \, \bigl\langle\alpha + z^{-1} \chi(\alpha)
\bigr\rangle_{\alpha \in \lieg} \, \eta \subset \widehat\lieg$. It implies that
\begin{itemize}
\item For any $(i, j) \in \mathrm{\rom{2}} \, \cup \,  \mathrm{\rom{4}}$, $ i \ne j$,
there exists $\mu_{i, j} \in \lieg$ such that
\begin{equation}\label{E:eqn1}
z^{-1} e_{j, i} + u_{(i, j; 0)} = \eta^{-1}\bigl(\mu_{i, j} + z^{-1}
\bigl[K^t, \mu_{i, j}\bigr]\bigr) \eta.
\end{equation}
\item Similarly, for any $1 \le l \le n-1$, there exists $\nu_l \in \lieg$ such that
\begin{equation}\label{E:eqn2}
z^{-1} h_l + u_{(l; 0)} = \eta^{-1}\bigl(\nu_{l} + z^{-1}
\bigl[K^t, \nu_{l}\bigr]\bigr) \eta.
\end{equation}
\item Finally, for any $(i, j) \in \mathrm{\rom{1}}$ and $k = 0, 1$,  there exists
$\kappa_{i, j}^{(k)} \in \lieg$ such that
\begin{equation}\label{E:eqn3}
z^{-k-1} e_{j, i} + u_{(i, j; k)} = \eta^{-1}\bigl(\kappa_{i, j}^{(k)} + z^{-1}
\bigl[K^t, \kappa_{i, j}^{(k)}\bigr]\bigr) \eta.
\end{equation}
\end{itemize}
A straightforward case-by-case analysis shows that equation
(\ref{E:eqn1}) (respectively,  (\ref{E:eqn2}) and (\ref{E:eqn3})) is equivalent to
equation (\ref{E:dec1}) (respectively, (\ref{E:dec2}) and (\ref{E:dec3}), (\ref{E:dec4})).
Thus, equalities (\ref{E:comparison}) are true and theorem is proven.
\end{proof}

\begin{example}\label{E:explicitSolut}{\rm  Let  $e = n - 1$. We take the matrix
$$
K = J_{(n-1, 1)} =
\left(
\begin{array}{cccc|c}
0 & 1 & 0 & \dots & 0 \\
0 & 0 & 1 & \dots & 0 \\
\vdots & \vdots & \ddots & \ddots & \vdots \\
0 & 0 & \dots & 0  & 1 \\
\hline
0 & 0 & \dots & 0  & 0
\end{array}
\right).
$$
Solving the equations (\ref{E:dec1})--(\ref{E:dec4}) yields  the following closed formula:
 {\scriptsize{\[
r_{(\lieg, e, \omega_K)} =
\frac{c}{y-x} +
\]
\[
 + x \left[e_{{1}, {2}}\otimes\check{h}_{1}-\sum_{j=3}^{n}e_{1,j}\otimes\left(\sum_{k=1}^{n-j+1}
e_{j+k-1,k+1}\right)\right]
- y \left[\check{h}_{1}\otimes e_{1,2}-
\sum_{j=3}^{n}\left(\sum_{k=1}^{n-j+1}e_{j+k-1,k+1}\right)\otimes e_{1,j}\right]
\]
\[
+\sum_{j=2}^{n-1}e_{1,j}\otimes\left(\sum_{k=1}^{n-j}e_{j+k,k+1}\right)+
\sum_{i=2}^{n-1}e_{i,i+1}\otimes\check{h}_{i}
-\sum_{j=2}^{n-1}\left(\sum_{k=1}^{n-j}e_{j+k,k+1}\right)\otimes e_{1,j}-\sum_{i=2}^{n-1}\check{h}_{i}\otimes e_{i,i+1} \]
\[
+\sum_{i=2}^{n-2}\left(\sum_{k=2}^{n-i}\left(\sum_{l=1}^{n-i-k+1}e_{i+k+l-1,l+i}\right)\otimes e_{i,i+k}\right)
-\sum_{i=2}^{n-2}\left(\sum_{k=2}^{n-i}e_{i,i+k}\otimes\left(\sum_{l=1}^{n-i-k+1}
e_{i+k+l-1,l+i}\right)\right).
\]
}
}

\noindent
In particular, for $n = 2$,  we get   the following rational solution
$$
{r}(x, y) = \frac{1}{y - x}\left(
\frac{1}{2} h \otimes h + e_{12}
\otimes e_{21} + e_{21} \otimes e_{12}\right) +
\frac{x}{2} e_{12} \otimes h - \frac{y}{2} h \otimes e_{21}.
$$
This solution was first time discovered by Stolin in \cite{Stolin}. It is gauge equivalent to
the solution (\ref{E:Stolin}).
}
\end{example}

\subsection{Comparison Theorem} Now  we prove the third  main result of this
article.

\begin{theorem}\label{T:ThmB} Consider the involutive  Lie algebra automorphism
 $\tilde\varphi: \lieg \rightarrow \lieg, A \mapsto - A^t$. Then
we have:
$
(\tilde{\varphi} \otimes \tilde{\varphi}) r_{(E, (n, d))} = r_{(\lieg, e, \omega_{K})},
$
where
$K = -J_{(e, d)}$.
\end{theorem}

\begin{proof}
For $x \in \CC$, $1 \le i \ne j \le n$ and $1 \le l \le n-1$ consider the following elements
of $\lieg[z]$:
\begin{equation*}
U_{(i, j)}^{(x)} = (z-x)\bigl(w_{(i, j; 0)} + x w_{(i, j; 1)}\bigr)
\quad \mbox{and} \quad
U_{(l)}^{(x)} = (z - x) w_{(l; 0)},
\end{equation*}
where $w_{(i, j; k)}$ and $w_{(l; 0)}$ are element introduced in Definition \ref{D:coeffTensorR}.
Then we have:
\begin{equation*}
r_{(\lieg, e, \omega_K)}
= \frac{1}{y-x}\Bigl[
c +
\sum\limits_{1 \le i \ne j \le n} e_{i, j} \otimes {U}_{(i, j)}^{(x)}(y)
+ \sum\limits_{1 \le l \le n-1} \check{h}_l \otimes {U}_{(l)}^{(x)}(y)
\Bigr].
\end{equation*}
Note that  for $(i, j) \in \mathrm{\rom{3}}$ we have: $U_{(i, j)}^{(x)} = 0$.

\noindent
In what follows,  instead of $\tilde\varphi$
we shall  use the anti-isomorphism of Lie algebras $\varphi = - \tilde\varphi$.
We have: $\varphi(e_{i, j}) = e_{j, i}, \varphi(h_l) = h_l, \varphi(\check{h}_l) = \check{h}_l$
and $\varphi \otimes \varphi = \tilde\varphi \otimes \tilde\varphi \in \End(\lieg \otimes \lieg)$.
Hence, we need to show that for all $1 \le i \ne j \le n$ and $1 \le l \le n-1$ we have:
\begin{equation*}
\varphi\bigl(G^{x}_{(i, j)}\bigr) = U_{(i, j)}^{(x)} \quad
\mbox{and}
\quad
\varphi\bigl(G_{(l)}^{(x)}\bigr) = U_l^{(x)}.
\end{equation*}
From the definition of elements $G_{(i, j)}^{(x)}$ and $G_{(l)}^{(x)}$ it follows that these equalities
are equivalent to the following statements.
\begin{itemize}
\item $U_{(i, j)}^{(x)}(x) = 0 = U_{(l)}^{(x)}$ and
\item $e_{j, i} + U_{(i, j)}^{(x)},
h_l + U_{(l)}^{(x)} \in \overline{\Sol\bigl((e, d), x\bigr)} := \varphi\Bigl(\Sol\bigl((e, d), x\bigr)\Bigr)$.
\end{itemize}
The first equality is obviously fulfilled. To show the second, observe that
\begin{equation*}
\overline{\Sol\bigl((e, d), x\bigr)} := \Bigl\{
P \in \overline{V}_{e,d} \,
\Bigl| \,
\left[J^t, P_0 \right]+ x P_{0}+P_{\epsilon}=0
\Bigr\} \subset \overline{V}_{e,d},
\end{equation*}
where
\begin{equation*}
\overline{V}_{e,d} = \left\{
P =
\left(
\begin{array}{ c | c}
W & Y \\ \hline
X & Z
\end{array}
\right)
+
\left(
\begin{array}{ c | c}
W' & Y' \\ \hline
0 & Z'
\end{array}
\right)z
+
\left(
\begin{array}{ c | c}
0  & Y'' \\ \hline
0  & 0
\end{array}
\right)z^2
\right\} \subset \lieg[z]
\end{equation*}
and  for a given $P \in \overline{V}_{e, d}$ we denote:
\begin{equation}\label{E:defoCompon}
P_{0}=
\left(
\begin{array}{c | c} W' & Y'' \\ \hline X & Z'
\end{array}
\right) \mbox{ and  } \,
P_{\epsilon}=\left({\begin{array}{c | c}
W & Y' \\ \hline 0 & Z
\end{array}}\right).
\end{equation}
Observe that in the above notations, there are no constraints on  the matrix $Y$.

\medskip
\noindent
For any $1 \le i \ne j \le n$ denote:
$A_{(i, j)} = A_{(i, j)}^{(0)} + x A_{(i, j)}^{(1)}$. Similarly, we set
$B_{(i, j)} = B_{(i, j)}^{(0)} + x B_{(i, j)}^{(1)}$,
$\widetilde{B}_{(i, j)} = \widetilde{B}_{(i, j)}^{(0)} +
x \widetilde{B}_{(i, j)}^{(1)}$ and
$D_{(i, j)} = D_{(i, j)}^{(0)} + x D_{(i, j)}^{(1)}$. Then we have:
\begin{equation*}
{U}_{(i, j)}^{(x)} =
-x \left(\begin{array}{c|c}  A_{(i, j)} & B_{(i, j)} \\ \hline  0 & D_{(i, j)}\end{array}\right) +
z \left(\begin{array}{c|c} A_{(i, j)}  & B_{(i, j)} - x \widetilde{B}_{(i, j)} \\ \hline 0 & D_{(i, j)}
\end{array}\right) +
z^2 \left(\begin{array}{c|c} 0  &  \widetilde{B}_{(i, j)} \\ \hline 0 & 0
\end{array}\right).
\end{equation*}
Similarly,
\begin{equation*}
{U}_{(i, j)}^{(x)} =
-x \left(\begin{array}{c|c}  A_{(l)} & B_{(l)} \\ \hline  0 & D_{(l)}\end{array}\right) +
z \left(\begin{array}{c|c} A_{(l)}  & B_{(l)} - x \widetilde{B}_{(l)} \\ \hline 0 & D_{(l)}
\end{array}\right) +
z^2 \left(\begin{array}{c|c} 0  &  \widetilde{B}_{(l)} \\ \hline 0 & 0
\end{array}\right).
\end{equation*}
 First observe that  for $(i, j) \in \mathrm{\rom{3}}$ we have:
$U_{(i, j)}^{(x)} =  0$. Since $e_{j, i} \in \overline{\Sol\bigl((e, d), x\bigr)}$, we are done with this
case. Now we assume that $(i, j) \in \mathrm{\rom{2}} \cup \mathrm{\rom{3}} \cup \mathrm{\rom{4}}$
and $i \ne j$. Then in the notations of (\ref{E:defoCompon}), for
$e_{j, i} + U_{(i, j)}^{(x)} \in \overline{V}_{e, d}$  we have:
\begin{equation*}
P_0^{(i, j)} = \left(\begin{array}{c|c}  A_{(i, j)} & \widetilde{B}_{(i, j)} \\ \hline  0 & D_{(i, j)}\end{array}\right) + \delta_{\mathrm{\rom{1}}}(i, j) e_{j, i}
\end{equation*}
and
\begin{equation*}
P_{\epsilon}^{(i, j)} = \left(\begin{array}{c|c}  - x A_{(i, j)} & B_{(i, j)} - x \widetilde{B}_{(i, j)} \\ \hline  0 & - x D_{(i, j)}\end{array}\right) + \bigl(\delta_{\mathrm{\rom{2}}} + \delta_{\mathrm{\rom{4}}}\bigr)(i, j)
e_{j, i}.
\end{equation*}
Here,  $\delta_{\mathrm{\rom{1}}}(i, j) = 1 $ if $(i, j) \in  \mathrm{\rom{1}}$ and zero otherwise,
whereas
$\delta_{\mathrm{\rom{2}}}$ and $\delta_{\mathrm{\rom{4}}}$
 have a  similar meaning. The condition $e_{j, i} + U_{(i, j)}^{(x)} \in \overline{\Sol\bigl((e, d), x\bigr)}$ is equivalent to the equality
 \begin{equation*}
 \Bigl[J^t,
 \left(\begin{array}{c|c} A_{(i, j)}  &  \widetilde{B}_{(i, j)} \\ \hline 0 & D_{(i, j)}
\end{array}\right) + \delta_{\mathrm{\rom{1}}}(i, j) e_{j, i}
 \Bigr] +
 x \delta_{\mathrm{\rom{1}}}(i, j)
e_{j, i} + \bigl(\delta_{\mathrm{\rom{2}}} + \delta_{\mathrm{\rom{4}}}\bigr)(i, j)
e_{j, i} + \left(\begin{array}{c|c}  0  &  \widetilde{B}_{(i, j)} \\ \hline  0 &
0\end{array}\right) = 0.
 \end{equation*}
Considering separately the case  $(i, j) \in  \mathrm{\rom{1}}$ and
$(i, j) \in  \mathrm{\rom{2}} \cup \mathrm{\rom{4}}$,
one can verify  that this equation follows from the equations (\ref{E:dec1}),
  (\ref{E:dec3}) and (\ref{E:dec4}). A similar argument shows that the condition
  $h_l + U_{(l)}^{(x)} \in  \overline{\Sol\bigl((e, d), x\bigr)}$ is equivalent to
   (\ref{E:dec2}). Theorem is proven.
 \end{proof}

\begin{remark}
Since the solutions $r_{(\lieg, e, \omega_K)}$ and  $r_{(\lieg, e, \omega_J)}$ are gauge equivalent,
we obtain a gauge equivalence of $r_{(\lieg, e, \omega_J)}$ and $r_{(E, (n, d))}$.
\end{remark}

\begin{corollary}\label{C:AnalyticConseq}
{\rm 
It follows now from Theorem \ref{T:main1} that up to a (not explicitly known)  gauge transformation and 
a change of variables, the rational solution from Example \ref{E:explicitSolut} is a degeneration
of the Belavin's elliptic $r$--matrix (\ref{E:SolutionsElliptiques}) for 
$\varepsilon = \exp\bigl(\frac{2 \pi i }{n}\bigr)$. It seems to be quite difficult to prove this result
using just direct analytic methods. } 
\end{corollary}

\noindent
We conclude this paper by the following result, which has been pointed out to us by Alexander Stolin.
\begin{proposition}
The solutions $r_{(E, (n, d))}$ and $r_{(E, (n, e))}$ are gauge equivalent.
\end{proposition}
\begin{proof}
Consider the Lie algebra automorphism $\liee \stackrel{\psi} \liee,
e_{i, j} \mapsto e_{n+1-i, n+1-j}$. Obviously, $\psi$ is an automorphism of $\lieg$, too. Moreover,
it is not difficult to see that $\psi(J_{(e, d)}) = J_{(d, e)}$.
The automorphism $\psi$ extends to an automorphism of $\lieg[z]$. Moreover, the following diagram is commutative:
\begin{equation*}
\xymatrix{
\lieg \ar[d]_{\psi} & & \Sol\bigl((e, d), x\bigr) \ar[d]^\psi \ar[ll]_-{\overline{\res}_x}  \ar[rr]^-{\overline{\ev}_y} & &\lieg \ar[d]^{\psi}\\
\lieg & & \Sol\bigl((d, e), x\bigr)  \ar[ll]_-{\overline{\res}_x}  \ar[rr]^-{\overline{\ev}_y} & &  \lieg
}
\end{equation*}
Hence, $(\psi \otimes \psi) r_{(E, (n, d))}  = r_{(E, (n, e))}$. Proposition is proven.
\end{proof}

\end{document}